%% file: main.tex
\documentclass[10pt, letterpaper, oneside, reqno]{amsart}

\input{style}
\begin{document}
\onehalfspacing
\title{Asymptotically Optimal Distributionally Robust Solutions through Forecasting and Operations Decentralization}

\author{
Yue Lin, 
Daniel Zhuoyu Long,
Viet Anh Nguyen,
Jin Qi}
\thanks{\small The authors are with the Chinese University of Hong Kong (\texttt{yuelin@link.cuhk.edu.hk, zylong@se.cuhk.edu.hk, nguyen@se.cuhk.edu.hk}) and Hong Kong University of Science and Technology (\texttt{jinqi@ust.hk})}

\begin{abstract}
Two-stage risk-averse distributionally robust optimization (DRO) problems are ubiquitous across many engineering and business applications. In these problems, decision-makers commit to capacitated first-stage decisions, anticipating that they can execute second-stage recourse decisions after observing the realization of the uncertain parameters. 
Despite their promising resilience, two-stage DRO problems are generally computationally intractable. To address this challenge, we propose a simple framework by decentralizing the decision-making process into two specialized teams: forecasting and operations. This decentralization aligns with prevalent organizational practices, in which the operations team uses the information communicated from the forecasting team as input to make decisions. 
We formalize this decentralized procedure as a bilevel problem to design a communicated distribution that can yield asymptotic optimal solutions to original two-stage risk-averse DRO problems. 
We identify an optimal solution that is surprisingly simple: The forecasting team only needs to communicate a two-point distribution to the operations team. 
Consequently, the operations team can solve a highly tractable and scalable optimization problem to identify asymptotic optimal solutions. 
Specifically, as the magnitude of the problem parameters (including the uncertain parameters and the first-stage capacity) increases to infinity at an appropriate rate, the cost ratio between our induced solution and the original optimal solution converges to one, indicating that our decentralized approach yields high-quality solutions. 
We compare our decentralized approach against the truncated linear decision rule approximation and demonstrate that our approach has broader applicability and superior computational efficiency while maintaining competitive performance. 
Using real-world sales data, we have demonstrated the practical effectiveness of our strategy. The finely tuned solution significantly outperforms traditional sample-average approximation methods in out-of-sample performance. 
\end{abstract}

\maketitle

\thispagestyle{firstpage}

\section{Introduction}

Many businesses make decisions in a two-stage process to hedge against the inherent uncertainty of the environment. In an assemble-to-order (ATO) system facing uncertain demand, managers first determine component order quantities based on demand forecasts. After the demand is realized, they deploy the assembly plans that minimize \textit{dis}utility while adhering to the available component inventory~\citep{ref:devalve2020primal}. 
This sequential decision-making pattern extends beyond ATO systems to various operational contexts, including multi-item inventory control, facility location, and appointment scheduling (see~\citet{ref:hanasusanto2015distributionally},~\citet{ ref:snyder2006facility}, and~\citet{ref:gupta2008appointment}). 

Operational success in sequential decision-making fundamentally relies on a company's ability to predict and hedge against uncertainty. 
The characteristics of the available forecasts crucially shape the hedging strategies employed.
Companies with probabilistic forecasts of uncertain parameters can employ stochastic programming to optimize expected outcomes~\citep{ref:birge2011introduction, ref:shapiro2021lectures}. 
However, precise distributional information rarely exists in practice, and estimated distributions often yield poor or harmful decisions, which can suffer from estimation errors and distributional shifts~\citep{ref:smith2006optimizer}.
When companies only have the support set information, robust optimization offers an alternative approach that guards against the worst-case outcome by optimizing for the most adverse scenario within the support set. Although robust optimization improves reliability under severe uncertainty, it may lead to overly conservative decisions when support sets are too large or incorrectly specified~\citep{ref:bertsimas2018data}. 

When historical data are available, companies can estimate partial distributional characteristics, such as support, moments, and frequency. In such cases, distributionally robust optimization (DRO), also known as robust stochastic optimization, has emerged as a promising framework to manage uncertainty~\citep{ref:delage2010distributionally, ref:rahimian2019distributionally}. 
DRO first constructs an ambiguity set using the estimated partial knowledge and then optimizes against the worst-case distribution within this set. It effectively combines the risk-averse principle of robust optimization with the probabilistic modeling of stochastic programming.
Although single-stage DRO problems have seen significant progress, two-stage DRO problems within sequential decision-making frameworks remain computationally challenging~\citep{ref:bertsimas2010models}. 
The fundamental difficulty lies in second-stage decisions being infinite-dimensional functions of uncertainty, which typically render these problems intractable.
Researchers have proposed restricting second-stage decisions to parametric functions, thereby recovering tractability by reducing the optimization to a finite number of parameters.
~\citet{ref:see2010robust} applied a truncated linear replenishment policy to the multi-period inventory control problem under ambiguous demands and showed that the objective value is closer to optimal than the static and linear replenishment policy.
~\citet{ref:bertsimas2019adaptive} focused on a class of second-order conic representable ambiguity set and approximated the problem using an extended linear decision rule by incorporating auxiliary random variables. They showed that the approximation is optimal when the problem has complete recourse and one-dimensional adaptive decision.
~\citet{ref:long2024supermodularity} solved a class of two-stage DRO problems under the scenario-based ambiguity set and investigated the optimality of the segregated affine decision rules when second-stage problems have supermodularity properties. 

We propose a novel framework to address the complexity inherent in two-stage DRO problems by decentralizing the decision-making process into two specialized teams: forecasting and operations. This separation follows practical organizational structures, where many companies strategically separate forecasting from operational decision-making to capitalize on specialized expertise~\citep{ref:scheele2018designing}. For instance, consumer goods companies typically assign demand forecasting to sales and marketing teams, while production and inventory decisions fall under operations teams~\citep{ref:Shapiro1977,ref:lee2017task}. In logistics companies, network planners design and optimize transportation routes based on the provided load forecasts~\citep{ref:lindsey2016improved,ref:bruys2024confidence}. This separation helps prevent informational deficiencies and functional biases often arising from poor cross-functional communication, as elaborated by~\citet{ref:oliva2009managing}. 
Consequently, the forecasting team predicts uncertainties, while the operations team uses these forecasts to make decisions. The sequential interaction between two teams constitutes a one-leader-one-follower Stackelberg game, where the forecasting team acts as the leader by providing uncertainty information, and the operations team acts as the follower by making decisions based on the received information. Stackelberg game models have applications in various complicated problems, such as decentralized supply chain, smart grid, and telecommunication~\citep{ref:kang2012price, ref:maharjan2013dependable, ref:fu2018profit}. 

Within the decentralized framework, the forecasting team communicates a probability distribution derived from the available knowledge to model uncertainty. 
The operations team then uses this communicated distribution as input to solve a two-stage stochastic programming problem. 
When sufficient and reliable historical data exist, the forecasting team can communicate the empirical distribution, leading to the sample-average approximation (SAA) method. However, real-world data are often limited, noisy, or not fully representative of the underlying distribution. 
In such cases, communicating empirical distributions can lead to suboptimal or unreliable decisions~\citep{ref:smith2006optimizer}. 
Hence, this paper works on designing an appropriate distribution for the forecasting team to communicate, aiming to steer the operations team to develop high-quality solutions to the original two-stage DRO problems and thereby ensure the resilience of these solutions.

Given the intractability of the original two-stage DRO problems, we evaluate the solution quality through asymptotic analysis under specific scaling schemes. 
Asymptotic analysis is a powerful tool for quantifying the degree of suboptimality and has applications in many academic disciplines, including probability theory, operations research, statistics, and computer science. 
It examines the performance of suboptimal solutions relative to optimal ones when problem parameters - arrival rate, unit cost, time horizon, network size, and capacity - grow to infinity~\citep{ref:chen2019welfare, ref:wang2022constant, ref:ledvina2022new, ref:bu2023asymptotic}. 
Large problem parameters often pose significant computational challenges, prompting researchers to develop heuristic solutions and analyze their asymptotic performance to ensure scalability and efficiency. 
In the context of DRO problems,~\citet{ref:van2021data} considered a single-stage optimization problem and demonstrated that the DRO solution is asymptotic optimal in the sense that the out-of-sample disappointment vanishes as the sample size goes to infinity. 
This result implies that the DRO approach effectively learns the data-generating distribution as more data become available and the optimality gap vanishes. 
The asymptotic regime of our paper is orthogonal to that in~\citet{ref:van2021data} as we scale the magnitude of random variables while maintaining considerable uncertainty about the underlying distribution. 
This scaling scheme reflects real-world situations where operational scales increase (e.g., larger markets, higher demands, expanded networks) while the inherent uncertainty remains pronounced due to complex and unpredictable factors. 

In addition to aversion to distributional ambiguity, we incorporate risk aversion into our model. It is essential to distinguish between risk aversion and ambiguity aversion~\citep{ref:bertsimas2010models, ref:fu2018profit, ref:cai2023distributionally}. Risk aversion refers to situations where the probabilities of possible outcomes are known (known risks). In contrast, ambiguity aversion pertains to situations where the probability distribution is unknown (i.e., unknown risks). Consequently, this paper concentrates on two-stage risk-averse DRO problems, where the second-stage deployment decisions are subject to the first-stage planning decisions and uncertainty realizations. 

\noindent\textbf{Contributions.} Our contributions are summarized as follows:
\begin{itemize}
    \item We propose a decentralized framework to simplify two-stage DRO problems. This framework models the collaboration between the forecasting and operations teams in the leader-follower dynamic of the Stackelberg games. The forecasting team leads by choosing which probability distribution to communicate, while the operations team follows by making decisions based on this communicated distribution. This decentralization aligns with prevalent organizational structures in the industry.
    We formalize this decentralized decision-making framework as a bilevel optimization problem. 
    Given a communicated distribution, the operations team solves a two-stage stochastic program to induce the first-stage decision at the lower level. At the upper level, the forecasting team determines which distribution to communicate by optimizing the limiting cost ratio between the induced first-stage (lower-level) solution and the optimal solution of the original two-stage DRO problem. 
    \item We explore two popular specifications of the ambiguity sets -- the moment-based set characterized by the mean and standard deviation and the data-driven Wasserstein ambiguity set. In both cases, we identify an optimal solution to the bilevel problem that is surprisingly simple: a two-point distribution suffices. This result has profound practical implications because the operations team can solve a tractable optimization problem (which is reduced to a linear program under mild assumptions) to obtain high-quality first-stage solutions. Our decentralized framework eliminates the computational \textit{in}tractability famously associated with two-stage DRO problems. Moreover, we prove that the optimal value of the bilevel problem is one, which implies that our induced first-stage solution is \textit{asymptotically optimal} to the original two-stage DRO problem. 

    \item We conduct extensive numerical experiments to showcase the performance of our framework. In the most important experiment, we collect real sales data from a leading Chinese near-expiration goods company, which employs a decentralized decision-making structure with dedicated forecasting and operations teams. Our numerical experiments focus on risk-averse scenarios, where we show that our finely-tuned solution significantly outperforms traditional sample-average approximation solutions in the out-of-sample performance. 
\end{itemize}

This paper unfolds as follows. Section~\ref{sec:2-stage-DRO} overviews preliminaries about the centralized two-stage risk-averse DRO problems and outlines necessary assumptions. Section~\ref{sec:forecast-operations} introduces the decentralized framework and formalizes the bilevel optimization problem. Section~\ref{sec:moment} designs optimal mechanisms for moment-based ambiguity sets and provides the optimality analysis. In Section~\ref{sec:wass}, we propose optimal mechanisms for the data-driven Wasserstein robustness setting. We present the numerical studies in Section~\ref{sec:numerical} and conclude our work in Section~\ref{sec:conclusion}.

\noindent
\textbf{Notations.} Let $\R^N$ represent the $N$-dimensional space of real vectors, with $\R_+^N$ indicating the subspace of non-negative real vectors. Uppercase letters (e.g., $A$, $G$, $H$, $I$) denote matrices, while lowercase letters (e.g., $x$, $b$, $p$) represent scalars or vectors, with $x_i$ indicating the $i$-th element of vector $x$. The boldface lower-case letters such as $\mathbf{y}$ represent functions. A random vector and an arbitrary distribution are denoted with a lowercase letter with a tilde and $\PP$, respectively, such as $\tilde d \sim \PP$, where $\tilde d\in\R^{\Nd}$ is a random vector, and $\PP$ represents the distribution of $\tilde d$. We denote the space of Borel-measurable functions from $\R^{\Nd}$ to $\R$ by $\mathcal L_0$. We denote by $\mathcal M_2$ the set of all probability distributions on the Borel $\sigma$-algebra on $\R^{\Nd}$ with finite second moments, i.e., $\E_{\PP}[\lVert \tilde d \rVert^2]<\infty$ for any $\PP\in\mathcal{M}_2$. Unless otherwise specified, the norm $\lVert \cdot \rVert$ represents the Euclidean norm. Vector operations are elementwise; for example, $u\le v$ implies that each $u_i \le v_i$, and $w=\max\{u,v\}$ signifies that $w_i = \max\{u_i,v_i\}$ for all components.

\section{Preliminaries} \label{sec:2-stage-DRO}

Consider a two-stage decision-making process where a centralized decision-maker must determine the first-stage (here-and-now) decision $x\in\R_+^{\Nx}$ \textit{before} the realization of a non-negative random demand vector $\tilde d \in \R_+^{\Nd}$. The feasible set of first-stage decisions is a nonempty polyhedron defined by $\mathcal X (b) = \{x \in \R_+^{\Nx}: Gx\leq b\}$, where $G\in\R^{K\times\Nx}$ and $b\in \R_+^K$ are known parameters. The parameter $b\in\R_+^{K}$ represents the operational budget, capturing constraints such as warehouse capacity, financial allocations, and other resource limitations. Let $c\in\R_+^{\Nx}$ denote the unit costs of the first-stage decisions. Upon observing the demand realization $d$, the decision-maker finds the optimal second-stage (wait-and-see) decision $y\in\R_+^{\Ny}$ by solving the linear program
\begin{equation} \label{eq:g}
g(x, d) = \left\{
    \begin{array}{cl}
    \min & -p^\top y \\
    \st  & y\in\R_+^{\Ny} \\
        & Ay\leq x,~ Hy\leq d,
    \end{array}
    \right.
\end{equation}
where $p \in \R_+^{\Ny}$ is the unit price vector. The non-negative matrices $A \in \R_+^{\Nx \times \Ny}$ and $H \in \R_+^{\Nd \times \Ny}$ specify the linear constraints that link the second-stage decision $y$ with the first-stage decision $x$ and the realization of demand $d$, respectively. 
To avoid degeneracy, we assume that $A$ and $H$ are nonzero matrices. The second-stage problem~\eqref{eq:g} has fixed recourse as uncertain parameters appear only on the right-hand side of the constraints and not within the matrices $A$ and $H$. 
Furthermore, for any first-stage decision $x\in\mathcal{X}(b)$ and possible demand realization $d$, problem~\eqref{eq:g} has a nonempty and bounded feasible region, ensured by non-negative parameters $(p, A, H)$ and decision vector $y$. Consequently, $g(x,d)$ is always finite as the optimal value of a linear program with a nonempty bounded feasible region, which implies that the two-stage problem has relatively complete recourse.

Suppose that $\tilde d $ follows a $\PP$, i.e., $\tilde d \sim \PP$, we use a risk measure $\varrho_{\PP}(\cdot)$ to quantify the risk under the distribution $\PP$. 
In stochastic programming, the goal is to minimize the sum of first-stage deterministic cost $c^\top x$ and the second-stage risk $\varrho_{\PP}(g(x,\tilde d))$. 
However, the complete demand distribution is rarely available in practice. 
We handle distributional uncertainty through an ambiguity set $\mathcal A(\theta)$, which contains all plausible distributions of $\tilde d$ that decision-makers aim to hedge against. 
The parameter $\theta$ encodes the available uncertainty information used to define the ambiguity set. 
For instance, in a moment-based DRO framework, $\theta$ includes moment information such as means and standard deviations of the marginal distributions. In a data-driven DRO framework, $\theta$ consists of a nominal distribution and a radius that define the center and size of the ambiguity set, respectively. 

Under the budget $b$ and the uncertainty information $\theta$, for any feasible first-stage decision $x \in \mathcal{X}(b)$, we define the objective function as
\begin{equation}\label{eq:obj}
\mathrm{Obj}(x,b,\theta) \Let c^\top x + \sup\limits_{\PP \in \mathcal A(\theta)} \varrho_{\PP}(g(x,\tilde d) ),
\end{equation}
which represents the total cost comprising the first-stage deterministic cost and the worst-case risk arising from the second-stage optimal cost function $g$. 
The terms $\sup_{\PP\in \mathcal A(\theta)}$ and $\varrho_{\PP}$ capture ambiguity aversion and risk aversion, respectively. 
Although the parameter $b$ does not appear explicitly on the right-hand side of~\eqref{eq:obj}, we include it to emphasize that $x\in\mathcal{X}(b)$.
The two-stage risk-averse DRO problem then seeks to minimize this objective as
\begin{equation} \label{eq:2stage-DRO}
\mathrm{OPT}(b,\theta)\Let
\min_{x \in \mathcal X (b)}~ \Big\{ c^\top x+ \sup\limits_{\PP \in \mathcal A(\theta)} \varrho_{\PP}(g(x,\tilde d ) ) \Big\}.
\end{equation}
Problem~\eqref{eq:2stage-DRO} faces significant computational challenges~\citep{ref:bertsimas2010models} due to two key factors: the min-max structure of the objective function and the second-stage optimal cost function $g(x,\cdot)$.
We will overcome these difficulties through a decentralized decision-making framework that restores tractability while preserving solution quality.

Before diving into the decentralized framework, we establish several standard assumptions regarding the risk measure $\varrho$. 
As there are infinitely many probability measures within $\mathcal A(\theta)$, we need to consider not a single risk measure, but a \textit{family} of risk measure  $\{\varrho_{\PP}\}_{\PP \in \mathcal M_2}$, where $\mathcal M_2$ is the set of probability distributions with finite second moments. 
We use the notion of a law-invariant family of risk measures, which extends the law-invariant property to various probability measures.
\begin{definition}[Law-invariant family of risk measures] \label{def:risk:measure:family}
    The family of risk measures $\{\varrho_{\PP}\}_{\PP \in \mathcal M_2}$ is law-invariant if $\varrho_{\PP_1}(\ell_1)=\varrho_{\PP_2}(\ell_2)$ for any measurable loss functions $\ell_1,~\ell_2\in\mathcal L_{0}$ and probability distributions $\PP_1,~\PP_2\in\mathcal M_2$ satisfying that the distribution of $\ell_1(\tilde d)$ under $\PP_1$ matches the
    distribution of $\ell_2(\tilde d)$ under $\PP_2$.
\end{definition}
To avoid clutter, we will omit the subscript $\PP \in \mathcal M_2$ in the family of risk measures $\{ \varrho_{\PP}\}$. Next, we review several basic properties of risk measures.
\begin{definition}[Properties of risk measures] \label{def:risk:measure}
	A risk measure $\varrho_{\PP}$ associated with a probability measure~$\PP$ is
	\begin{itemize}[leftmargin=.2in]
		\item translation invariant  if $\varrho_\PP(\tilde Z + t) = \varrho_\PP(\tilde Z) + t $ for all $\tilde Z \in \mathcal L_{0}$, $t\in \mathbb{R}$;
		\item positive homogeneous if $\varrho_\PP(\lambda \tilde Z) = \lambda \varrho_\PP(\tilde Z)$ for all $\tilde Z \in \mathcal L_{0}$, $ \lambda \in \mathbb{R}_+$;
		\item monotonic if $\varrho_\PP(\tilde Z ) \leq \varrho_\PP(\tilde Z ^\prime )$ for all $\tilde Z, \tilde Z^\prime \in \mathcal L_{0}$ such that $\tilde Z \leq \tilde Z ^\prime $ $\PP$-almost surely;
		\item closed if the set $\{ \tilde Z \in \mathcal L_{0} | \varrho_{\PP} (\tilde Z) \le t \}$ is closed \footnote{A sequence of random variables $\tilde Z ^k,k \in \mathbb{Z}_+$ converges to a random variable $\tilde Z$ if $\lVert \tilde Z^k - \tilde Z \rVert \to 0$ where $\lVert \tilde Z \rVert \Let (\E[\tilde Z^2])^{1/2}$.} for all $t\in\R$.
	\end{itemize}
\end{definition}
The first three properties are common and firmly established in the pioneering work of~\citet{ref:artzner1999coherent}. To define coherent risk measure rigorously,~\citet{ref:rockafellar2007coherent} emphasized the notion of closedness to include the case where the underlying distribution $\PP$ has an infinite support set. The closedness says that if a random variable $\tilde Z$ can be approximated arbitrarily closely by acceptable random variables $\tilde Z_k$ satisfying $\varrho_{\PP} (\tilde Z_k) \le t$, then $\tilde Z$ should also be acceptable, i.e., $\varrho_{\PP} (\tilde Z) \le t $.
Many popular risk measures satisfy the property of closedness; see~\citet {ref:rockafellar2007coherent} for a discussion.

A law-invariant family of risk measures admits a scalar, known as the standard risk coefficient, which connects the risk measure with moment information~\citep{ref:yu2009general, ref:nguyen2021mean}.
\begin{definition}[Standard risk coefficient]\label{def:standard:risk:coefficient}
    The standard risk coefficient $\alpha$ of a family of risk measures $\{\varrho_{\PP} \}$ is the worst-case risk induced by any random variable with zero mean and unit variance: 
    \[
    \alpha \Let \sup \{  \varrho_{\PP}(\tilde Z) ~:~ \E_{\PP}[\tilde Z] = 0,~\E_{\PP}[\tilde Z^2] = 1 \}.
    \]
\end{definition}
We formalize our risk measure assumptions below without requiring convexity. 
\begin{assumption}[Risk measures] \label{a:risk-measure}
    The family $\{\varrho_{\PP}\}$ is a family of law-invariant, translation-invariant, positive homogeneous, monotonic, and closed risk measures. Further, its standard risk coefficient $\alpha$ is non-negative.
\end{assumption}
Many risk measures satisfy Assumption~\ref{a:risk-measure}. Table~\ref{table:risk} presents three examples with their standard risk coefficients. Additional examples can be found in~\citet{ref:nguyen2021mean}.
\begin{table}[h!]
\centering
\caption{Popular risk measures and their standard risk coefficients.}
\label{table:risk}
\begin{tabular}{ccc}
\hline
Risk measure & Definition and parameters & $\alpha$ \\
\hline
Conditional Value-at-Risk &
$
\inf_{t \in \mathbb{R}}
\left\{
t + \frac{1}{\beta} \mathbb{E}_\PP[\max\{\ell(\tilde d) - t, 0 \}]
\right\},
\beta \in (0,1)
$ &
$
\sqrt{\frac{1 - \beta}{\beta} }
$ \\
Value-at-Risk &
$
\inf_{ t \in \mathbb{R}} \left\{ \PP(\ell(\tilde{d}) \leq t ) \geq 1 - \beta\right\},
\beta \in (0,1)
$ &
$
\sqrt{\frac{1 - \beta}{\beta} } 
$ \\
Mean-semideviation &
$
\mathbb{E}_\PP[\ell(\tilde d)] + \lambda \mathbb{E}_\PP\big[\big(\ell(\tilde d) - \mathbb{E}_\PP[\ell(\tilde d)]\big)^\beta_+\big]^{\frac{1}{\beta}},
\lambda \in[0,1], \beta \in [1,2]
$ &
$\lambda$ \\
\hline
\end{tabular}
\end{table}

\section{Decentralization and Mechanism Design via Bilevel Programming} \label{sec:forecast-operations}

In many companies, different departments are responsible for forecasting and operational decision-making~\citep{ref:lee2017task,ref:scheele2018designing}.
This organizational structure motivates us to decentralize the two-stage DRO solution process. Figure~\ref{fig:framework} compares the centralized (conventional) two-stage DRO solution process and our decentralized leader-follower proposal. 
\begin{figure}[!h]

    \centering
    \includegraphics[trim = 280pt 80pt 0pt 100pt, clip,width=1.15\linewidth]{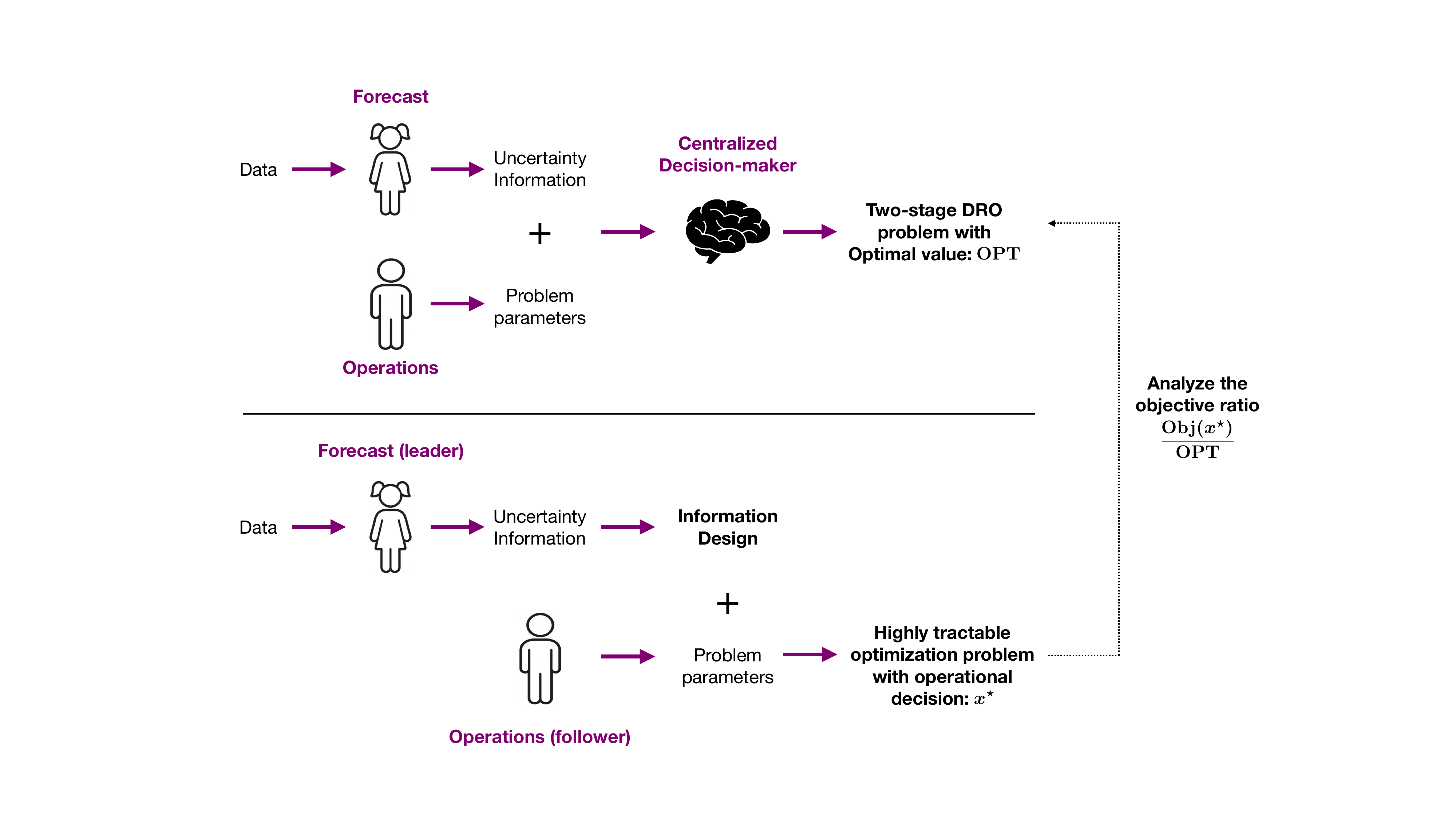}
    \caption{Conventional process (top) requires forecasting and operations teams to solve the problem together, leading to a two-stage DRO formulation that is computationally intensive. The leader-follower decentralization (bottom) allows the forecasting team (leader) to design the information sent to the operations team (follower). This decentralization can lead to simpler optimization problems to be solved by the operations team while guaranteeing that the induced first-stage decision is competitive.}
    \label{fig:framework}
    
\end{figure}

The forecasting team is the leader in the decentralized setting; it designs a mechanism that converts the uncertainty information into a simpler distribution. Subsequently, the operations team, in the follower role, uses the communicated distribution as input to optimize decisions. The critical challenge lies in finding a simple mechanism that allows the operations team to solve a highly tractable optimization problem while achieving near-optimal objective values relative to the theoretical optimum $\mathrm{OPT}(b,\theta)$. 
In this section, we formalize the decentralized framework as a bilevel problem to design mechanisms yielding asymptotic optimal first-stage solutions to two-stage DRO problems.

Let $\mathcal M$ denote the space of all the mechanisms that map the uncertainty information $\theta$ to a probability distribution communicated to the operations team. If we restrict the communicated distribution within the ambiguity set $\mathcal{A}(\theta)$, we ensure its consistency with the prescribed uncertainty information. The forecasting team's task is to select a mechanism $M\in\mathcal M$ that steers the operations team to produce high-quality first-stage solutions. To this end, we first introduce the performance measure and then define the objective of the bilevel problem. For a given first-stage decision $x\in\mathcal{X}(b)$, we evaluate its performance by the ratio between its cost $\mathrm{Obj}(x,b,\theta)$ and the theoretical optimal cost $\mathrm{OPT}(b,\theta)$. Given computational challenges, we investigate the asymptotic performance of first-stage solutions under scaling schemes reminiscent of heavy-traffic scenarios in queueing systems. Specifically, we examine a sequence of problems indexed by $k$, with parameters $(b^{(k)},\theta^{(k)})$ scaling to infinity at appropriate rates (detailed in the following sections). We analyze the limiting behavior of the cost ratio $\mathrm{Obj}(x^{(k)},b^{(k)},\theta^{(k)})/\mathrm{OPT}(b^{(k)}, \theta^{(k)})$, where $x^{(k)}$ represents the first-stage solution induced by the operations team using the communicated distribution. A convergence to one indicates the asymptotic optimality of the first-stage solution to the two-stage DRO problem.

We assume that $\mathrm{OPT}(b^{(k)}, \theta^{(k)}) < 0$ for all sufficiently large $k$, implying that the business has the potential to generate profits as the market size increases. When designing the mechanism, we must exclude those that lead to an implausible first-stage decision $x^{(k)}$ with $\mathrm{Obj}(x^{(k)},b^{(k)},\theta^{(k)}) > 0$, which results in losses despite the fact that profitable opportunities are available under optimal decisions.
In the desirable scenario where $\mathrm{Obj}(x^{(k)},b^{(k)},\theta^{(k)}) <0
$, the cost ratio between $\mathrm{Obj}(x^{(k)},b^{(k)},\theta^{(k)})$ and $\mathrm{OPT}(b^{(k)}, \theta^{(k)})$ is positive and is bounded above by one as 
\[
|\mathrm{OPT}(b^{(k)}, \theta^{(k)}) | \ge | \mathrm{Obj}(x^{(k)},b^{(k)},\theta^{(k)}) |.
\]
Therefore, under the condition that $\mathrm{OPT}(b^{(k)}, \theta^{(k)}) < 0$, a larger cost ratio indicates a smaller suboptimality, with one serving as the upper bound. Consequently, the forecaster's optimal mechanism solves the following bilevel optimization problem
\begin{equation} 
    \begin{array}{cl}
        \max\limits_{M \in \mathcal M} & \lim\limits_{k \to \infty}~ \displaystyle \frac{\mathrm{Obj}(x^{(k)},b^{(k)},\theta^{(k)}) }{\mathrm{OPT}(b^{(k)}, \theta^{(k)})}   \\ [2ex]
        \st   & \forall k: x^{(k)} \in 
           \arg\min \left\{ c^\top x + \varrho_{M(\theta^{(k)})}( g (x,\tilde d)) ~:~ x \in \mathcal X (b^{(k)}) \right\}.
  \end{array}
  \label{eq:bilevel}
\end{equation}
This bilevel structure reflects the decision hierarchy: at the upper level, the forecaster selects a mechanism to maximize the limiting cost ratio, while at the lower level, the operations team solves a two-stage stochastic programming problem based on the provided distribution $M(\theta^{(k)})$ and returns the corresponding optimal first-stage solution $x^{(k)}$. 
Generally, the lower-level problem is \#P-hard even if $g(x,d)$ is the negative part of an affine function and $\tilde d$ is uniformly distributed on the unit hypercube~\citep{ref:hanasusanto2016comment}. 
Even for discrete distributions with finitely many scenarios for $\tilde d$, two-stage stochastic programming problems suffer from the curse of dimensionality as the computational complexity increases with the dimension of uncertainty~\citep{ref:shapiro2005complexity}.
Despite these hardness results, we will show that a simple mechanism is optimal for problem~\eqref{eq:bilevel} under two types of distributional ambiguity.

\section{Optimal Mechanisms for Moment-Based Ambiguity Sets}\label{sec:moment}
We begin with moment-based ambiguity sets defined through first- and second-moment information. Specifically, we focus on marginal means and standard deviations $(\mu,\sigma)$, as these statistics are more practical to estimate than full covariance matrices~\citep{ref:gao2017data}.
These marginal moments form the basis of our moment-based ambiguity set.
\begin{assumption}[Moment-based Ambiguity set] \label{a:ambiguity}
    Given $\theta=(\mu,\sigma)\in\R_+^{2\Nd}$, where $\mu$ is the mean vector, and $\sigma$ is the upper bound for marginal standard deviation vectors, the moment-based ambiguity set $\mathcal{A}(\theta)$ is
\begin{equation} \label{eq:ambiguity}
    \mathcal{A}(\theta) \Let \{\mathbb P \in \mathcal M_2 : ~ \PP(\tilde d \in \R_+^{\Nd}) =1,~       \mathbb{E}_{\PP}[ \tilde d] = \mu, ~
        \mathbb E_{\mathbb P}[(\tilde d_i - \mu_i)^2] \leq \sigma_i^2 ~\forall i \in [N_d] 
    \}.
\end{equation}
\end{assumption} 

For the scaling scheme, we first fix an arbitrary base $b^{(1)} = b \in\R_+^K$ and $\theta^{(1)}  = ( \mu, \sigma) \in \R_+^{2\Nd}$. In the $k$-th problem, the scaled parameters are 
\begin{equation} \label{eq:scaling}
    b^{(k)} = kb \quad \text{ and } \quad \theta^{(k)} = (k \mu, k^{\frac{s}{2}} \sigma) \text{ for an $s \in [1,2)$}.
\end{equation}
This scaling scheme incorporates two key features. First, the budget parameter $b^{(k)}$ and the feasible region $\mathcal X(b^{(k)})$ scale linearly with $k$. Second, while the mean vector scales linearly to $k\mu$, standard deviations scale \textit{sub}linearly to $k^{\frac{s}{2}}\sigma$. 
The motivation for this scaling approach comes from real-world observations. As the market size increases, the customers' arrival rate increases, requiring proportional budget adjustments to meet demand. Moreover, the scaling follows Taylor's law~\citep{ref:taylor1961aggregation}, which states that the variance of arrival counts within a period is proportional to the mean raised to the power of $s\in[1,2)$. This scheme accommodates the Poisson distributed demand, where the variance equals the mean, thus $s=1$.

\subsection{An Optimal Mechanism}

For a fixed base $\theta^{(1)}  = ( \mu, \sigma)$ and the scaling scheme~\eqref{eq:scaling}, we consider the mechanism $M_{\varsigma, \tau}$ parameterized by two parameters $(\varsigma, \tau)\in\R_+^{\Nd}\times \R_+$:
\begin{subequations}\label{eq:Mall}
\be \label{eq:Mtau}
    M_{\varsigma, \tau}(\theta^{(k)}) \Let (1-\tau) \delta_{d_l} + \tau \delta_{d_h},
\ee
where the locations $d_l$ and $d_h$ are defined as
\be
d_l \Let k\mu -  \sqrt{\tau/ (1-\tau) } k^{\frac{s}{2}}\varsigma\text{ and }d_h \Let k\mu + \sqrt{(1-\tau)/\tau} k^{\frac{s}{2}} \varsigma.
\ee
To avoid cluttered notation, we do not explicitly present the dependence of the locations on $\varsigma, \tau$ and $\theta^{(k)}$. Further, we restrict the parameters $(\varsigma, \tau)$ to the range
\begin{equation} \label{eq:param-range}
    (\varsigma, \tau) \in [0, \sigma] \times [0, \tau_{\max}], \quad \text{ where } \tau_{\max} \Let \frac{ \gamma^2}{1+ \gamma^2} \text{ and } \gamma \triangleq \min_{i\in[\Nd]}  \frac{\mu_{i}}{ \varsigma_{i} },
\end{equation}
\end{subequations}
and we emphasize that the above range depends only on the base demand $\theta^{(1)}$ and the choice of $\varsigma$, but not on the scaling index $k$. Note that $\varsigma$ is a vector, and the condition $\varsigma \in [0, \sigma]$ means componentwise inclusion, i.e., $\varsigma_i \in [0 ,\sigma_i]$ for all $i\in[\Nd]$. Moreover, the choice of $\varsigma$ influences $\tau_{\max}$, the upper bound for feasible $\tau$.

The mechanism $M_{\varsigma, \tau}$ outputs a two-point distribution with probability $1-\tau$ at a low demand location $d_l$ and probability $\tau$ at a high demand location $d_h$. For a clear exposition, we define $x_{\varsigma, \tau}^{(k)}$ as a solution to the lower-level problem in~\eqref{eq:bilevel} under the mechanism $M_{\varsigma, \tau}(\theta^{(k)})$ in~\eqref{eq:Mall}:
\begin{equation} \label{eq:x-opt}
x_{\varsigma, \tau}^{(k)} \in \arg\min \left\{ c^\top x + \varrho_{M_{\varsigma, \tau}(\theta^{(k)})}( g (x,\tilde d)) ~:~ x \in \mathcal X (b^{(k)})
       \right\}.
\end{equation}
We can verify that when $\varsigma = 0$ or $\tau = 0$, the mechanism $M_{\varsigma, \tau}(\theta^{(k)})$ outputs a Dirac distribution located at the scaled mean $k\mu$, with the high demand atom vanishing. Because this Dirac distribution no longer depends on $\varsigma$, we define $M_{0} (\theta^{(k)}) \Let \delta_{k\mu}$ with a slight abuse of notation, and denote by $x_{0}^{(k)}$ the lower-level solution under $M_{0} (\theta^{(k)})$ in~\eqref{eq:x-opt}. Throughout, we will refer to $M_0$ as the expectation mechanism. We next show that the family of mechanisms $M_{\varsigma, \tau}$ defined in~\eqref{eq:Mall} is optimal for problem~\eqref{eq:bilevel} under the following additional assumption. 
\begin{assumption}[Model parameters] \label{a:model-param}
Suppose that the following conditions are satisfied.
\begin{enumerate}[label=(\roman*)]
    \item \label{a:model-param-1} Each column of the matrix $H$ has one nonzero element.
    \item \label{a:model-param-2} When $k=1$, the lower-level problem under the expectation mechanism $M_{0}$ has non-trivial solutions, i.e., $c^\top x_{0}^{(1)} + \varrho_{M_{0}(\theta^{(1)})}( g (x_{0}^{(1)},\tilde d)) < 0$.
\end{enumerate}
\end{assumption}

\begin{theorem}[Optimal mechanism]\label{thm:opt}
Suppose that Assumptions~\ref{a:risk-measure},~\ref{a:ambiguity}, and~\ref{a:model-param} are satisfied, and the scaling~\eqref{eq:scaling} is in force. The mechanism $M_{\varsigma, \tau}$ defined in~\eqref{eq:Mall} is optimal for the bilevel problem~\eqref{eq:bilevel}. This means that the corresponding lower-level solution, denoted by $ x_{\varsigma, \tau} ^{(k)}$, satisfies the following asymptotic optimality property:
\[
\lim\limits_{k \to \infty}~ \displaystyle \frac{\mathrm{Obj}(x_{\varsigma, \tau}^{(k)} , b^{(k)}, \theta^{(k)}) }{\mathrm{OPT}(b^{(k)}, \theta^{(k)})} = 1.
\]
\end{theorem}

We present the technical results for Theorem~\ref{thm:opt} in the following subsection. 
Before that, we discuss the prerequisite Assumption~\ref{a:model-param} and analyze the properties of the mechanism $M_{\varsigma, \tau}$.
Assumption~\ref{a:model-param} summarizes conditions on model parameters that hold in many cases of practical importance. 
Assumption~\ref{a:model-param}\ref{a:model-param-1} can be satisfied in various applications by lifting the dimension of second-stage decisions, see Appendix~\ref{appendix:application}. 
Assumption~\ref{a:model-param}\ref{a:model-param-2} is widely recognized in the operations management literature~\citep{ref:kerimov2024dynamic, ref:jasin2012re} and is satisfied whenever $p \nleqslant A^\top c$. 
We will demonstrate two crucial properties of Assumption~\ref{a:model-param}\ref{a:model-param-2}. First, it implies that the lower-level problems under the expectation mechanism have non-trivial solutions for all $k \ge 1$. Second, it ensures the assumption that $\mathrm{OPT}(b^{(k)}, \theta^{(k)} ) < 0$ is satisfied for large $k$, which was presumed in establishing the bilevel problem~\eqref{eq:bilevel}. 

 \begin{figure}[htbp]
    
\centering\includegraphics[width=0.6\textwidth]{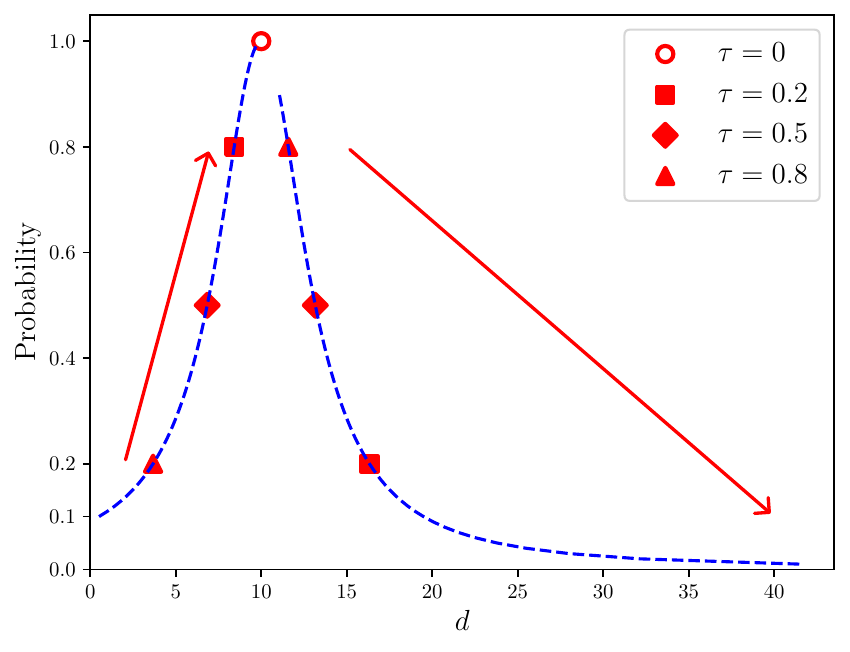}
    \caption{An illustration of the $M_{\varsigma, \tau}(\theta)$ in  one-dimension space with $\mu=10 $ and $\sigma = \varsigma =\sqrt{10} $. The upper bound for $\tau$ is $\tau_{\max} = 10/11$. The horizontal and vertical axes show the locations and probabilities of the two atoms in $M_{\varsigma, \tau}(\theta)$, respectively. As $\tau$ decreases, the right atom moves rightward with decreasing probability, while the left atom approaches $\mu=10$ with increasing probability. The limit of $M_{\varsigma, \tau}(\theta)$ is the Dirac distribution supported at the mean value $\mu$ as $\tau \downarrow 0$, with the right atom vanishing to zero probability at infinity.}
    \label{fig:Mtau}
    
\end{figure}

The mechanism $M_{\varsigma, \tau}$ possesses several desirable properties. First, $M_{\varsigma, \tau}$ is arguably simple: it generates a discrete distribution with just two scenarios -- high demand $d_h$ and low demand $d_l$ -- each with an associated probability.
This binary structure mirrors common decision-making frameworks in which practitioners often analyze best- and worst-case scenarios. 
One can also verify that $M_{\varsigma, \tau} (\theta^{(k)}) \in \mathcal{A}(\theta^{(k)})$, meaning that the mechanism $M_{\varsigma, \tau}$ outputs a distribution that lies inside the ambiguity set.\footnote{We provide a rigorous proof of these claims in the proof of Proposition~\ref{prop:value-upper-bound}.}  
Figure~\ref{fig:Mtau} illustrates how the two-point distribution changes with different values of $\tau$. The forecasting team can fine-tune $\varsigma$ and $\tau$ to optimize the induced solution's performance for practical implementation. Since $\varsigma$ is a $\Nd$-dimensional vector, we recommend setting $\varsigma = \kappa\sigma$ with a scalar $\kappa\in[0,1]$ and adjusting two parameters $(\kappa,\tau)$ within their feasible range $(\kappa,\tau) \in  [0, 1]\times [0, \tau_{\max}]$ to simplify the tuning process.

Second, the mechanism $M_{\varsigma, \tau}$ ensures the well-posedness and computational tractability of the lower-level problem in~\eqref{eq:bilevel}.
The established attainability of $g(x,\cdot)$ and the monotonicity of $\varrho$ presumed in Assumption~\ref{a:risk-measure} allow us to interchange the minimization operator and the risk functional~\citep[Proposition~2.1]{ref:shapiro2017interchangeability}. Specifically, for a probability measure $\PP$, it holds that
\be \label{eq:interchange-varrho}
\varrho_{\PP}(g(x,\tilde d)) = \left\{
\begin{array}{cl}
    \min & \varrho_{\PP}(-p^\top \mathbf{y}(\tilde d)) \\
    \st & \mathbf{y}(\tilde d) \in\R_+^{\Ny},~A\mathbf{y}(\tilde d) \le x ,~ H\mathbf{y}(\tilde d) \le \tilde d \qquad \PP\text{-almost surely},
    \end{array}
\right.
\ee
where $\mathbf{y}: \R_+^{\Nd} \to \R_+^{\Ny}$ maps the demand realization $d$ to a second-stage decision $\mathbf{y}(d)$. 
Moreover, suppose that $g(x, d)$ is attained at $\mathbf{y}\opt(d)$ for all possible realizations $d$, then $\mathbf{y}\opt(\cdot)$ also solves the optimization problem in~\eqref{eq:interchange-varrho}.
Leveraging the interchangeability, we consider the problem
 \begin{equation}\label{eq:V-tau-k}
    V_{\varsigma, \tau}(b^{(k)},\theta^{(k)}) \Let \left\{ 
    \begin{array}{cl}
        \min & c^\top x +  \varrho_{M_{\varsigma, \tau}(\theta^{(k)})}( -p^\top \mathbf{y}(\tilde d ) ) \\
        \st & x \in \mathcal X (b^{(k)}),~\mathbf{y}(d_l)\in \R_+^{\Ny},~\mathbf{y}(d_h)\in \R_+^{\Ny} \\
            & A  \mathbf{y}(d_l) \le x ,~ H  \mathbf{y}(d_l) \le d_l \\
            &  A  \mathbf{y}(d_h) \le x ,~ H  \mathbf{y}(d_h) \le d_h,
    \end{array} 
    \right. 
\end{equation}
whose optimal solution in the variable $x$ is $x_{\varsigma, \tau}^{(k)}$ set out in~\eqref{eq:x-opt}.
Problem~\eqref{eq:V-tau-k} integrates two second-stage recourse actions $\mathbf{y}(d_l)$ and $\mathbf{y}(d_h)$ into the outer problem. Consequently, problem~\eqref{eq:V-tau-k} is a one-stage problem in the sense that all decisions $x$, $\mathbf{y}(d_l)$ and $\mathbf{y}(d_h)$ are now determined simultaneously.
Reformulating the lower-level problem as problem~\eqref{eq:V-tau-k} offers significant tractability and computational efficiency advantages. Consolidating all decisions into a single stage eliminates the need to characterize second-stage cost functions explicitly. Moreover, for widely used risk measures, including expectation and CVaR, problem~\eqref{eq:V-tau-k} can be reformulated as a linear program, which can be solved efficiently using off-the-shelf solvers. With only two possible realizations in the mechanism $M_{\varsigma, \tau}$, the problem's size grows linearly with the uncertainty dimension $\Nd$ (e.g., the number of products), maintaining tractability even for large-scale settings with high-dimensional uncertainty.

As a result, the proposed mechanism $M_{\varsigma, \tau}$ in~\eqref{eq:Mall} leads to a tractable optimization problem~\eqref{eq:V-tau-k} for the operations team. Beyond its tractability, problem~\eqref{eq:V-tau-k} can yield an asymptotical optimal first-stage solution to the two-stage DRO problem~\eqref{eq:2stage-DRO} as stated in Theorem~\ref{thm:opt}. This combination of computational efficiency and theoretical guarantee underscores the benefit of our proposed decentralized approach and mechanism in addressing two-stage DRO challenges.

\subsection{Proof of Theorem~\ref{thm:opt}}\label{sec:proof-thm}

To establish the optimality of $M_{\varsigma, \tau}$, we will show that the ratio between $\mathrm{Obj}(x_{\varsigma, \tau}^{(k)}, b^{(k)},\theta^{(k)})$ and $\mathrm{OPT}(b^{(k)},\theta^{(k)})$ converges to one as $k$ increases. However, direct computation of these two cost functions is NP-hard due to the minimax structure and the random right-hand side in the second-stage problem~\citep{ref:bertsimas2010models}.
To circumvent this computational challenge, we develop a sandwich approach: we construct a lower bound for $\mathrm{OPT}(b^{(k)},\theta^{(k)})$ and an upper bound for $\mathrm{Obj}(x_{\varsigma, \tau}^{(k)},b^{(k)},\theta^{(k)})$. By demonstrating that the ratio of these bounds converges to one as $k$ approaches infinity, we prove the mechanism's optimality without requiring exact calculations of the cost.

\subsubsection{A Lower Bound.}

First, we find a lower bound for $\mathrm{OPT}$ through the expectation mechanism $M_0(\theta^{(k)}) = \delta_{k\mu}$. Let $V_0$ be the value function of the lower-level problem under $M_0$:
 \begin{equation}\label{eq:V-0}
    V_{0}(b^{(k)},\theta^{(k)}) \Let \left\{ 
    \begin{array}{cl}
        \min & c^\top x  - p^\top y \\
        \st & x \in \mathcal X (b^{(k)}),~y\in \R_+^{\Ny}\\
            & A y \le x,~ H y \le k \mu,
    \end{array} 
    \right. 
\end{equation}
which is a linear program, and its optimal solution in the variable $x$ is $x_0^{(k)}$. The next proposition asserts that $V_0$ is a lower bound for $\mathrm{OPT}$.

\begin{proposition}[Lower bound] \label{prop:lower}
Suppose that  Assumptions~~\ref{a:risk-measure},~\ref{a:ambiguity}, and~\ref{a:model-param}\ref{a:model-param-1} are satisfied, and the scaling~\eqref{eq:scaling} is in force. The optimal value $V_{0}(b^{(k)},\theta^{(k)})$ is finite and attainable. Moreover, $V_{0}$ is a lower bound for $\mathrm{OPT}$, that is, $
V_{0}(b^{(k)},\theta^{(k)}) \le \mathrm{OPT}(b^{(k)},\theta^{(k)})$ for all integers $k$. 
\end{proposition}

\subsubsection{An Upper Bound.}\label{sec:upper-bound}

Next, we construct an upper bound for $\mathrm{Obj}(x_{\varsigma, \tau}^{(k)}, b^{(k)},\theta^{(k)})$ defined in~\eqref{eq:obj}. 
Remind that evaluating $\mathrm{Obj}(x_{\varsigma, \tau}^{(k)}, b^{(k)},\theta^{(k)})$ is computationally intractable as we have no analytic form of the optimal second-stage cost function $g$. Nevertheless, we obtain an upper bound for $\mathrm{Obj}$ by restricting the second-stage recourse function $\mathbf{y}(\cdot)$ to an appropriate family of parametric functions. For a given first-stage decision $x$, a feasible recourse function needs to satisfy 
\be \label{eq:tldr}
\mathbf{y}(d) \ge 0,~A \mathbf{y}(d) \leq x,~H \mathbf{y}(d) \leq d \quad \forall d \ge 0.
\ee
If we reduce $\mathbf{y}(\cdot)$ to constant functions, also known as static decision rules, then the only feasible function that can satisfy $H \mathbf{y}(d) \le d$ and $\mathbf{y}(d) \ge 0$ for all values of $d$ is the zero function, i.e., $\mathbf{y} \equiv 0$. In this case, the first-stage decision $x$ must also be zero, leading to a trivial solution. Alternatively, if we consider the linear decision rule $\mathbf{y}(d) \Let Qd+q$ with $Q\neq 0$, then $\mathbf{y}(d)$ grows linearly with $d$. When $d$ tends to infinity, the constraint $A\mathbf{y}(d) \le x$ will be violated whenever the first-stage decision $x$ is finite. Therefore, the linear decision rule cannot yield a feasible policy.  This observation inspires us to use a truncated linear decision rule (TLDR) set:
\[
\mathcal{Y} \Let \left\{ \mathbf{y}: d \mapsto \min\{ v,Ud\} \big| v\in\R_+^{\Ny},U\in\R_+^{\Ny\times \Nd }\right\} .
\]
The parameter $v$ acts as a threshold vector, and each row of $U$ acts as a slope vector. 
We formalize the feasibility conditions on the parameters $v$ and $U$ in Proposition~\ref{prop:HU}.

\begin{proposition}[Constraints reduction]\label{prop:HU}
    Suppose that $A$ and $H$ are non-negative matrices. For a fixed $x\in\R_+^{\Nx}$, the second-stage policy $\mathbf{y}: d \mapsto \min\{v,Ud\}$ is
    feasible if and only if $(v,U)\in\R_+^{\Ny} \times\R_+^{\Ny\times\Nd}$ satisfies
    \[
    Av \le x ~~~\text{and}~~~ HU \le I.
    \]
\end{proposition}

Based on Proposition~\ref{prop:HU}, we find a feasible second-stage policy under the first-stage decision $x_{\varsigma, \tau}^{(k)}$ by solving
\begin{subequations}\label{eq:V-bar-tldr}
 \begin{equation}\label{eq:V-bar}
 \bar V_{\varsigma, \tau} (b^{(k)},\theta^{(k)}) \Let \left\{
    \begin{array}{cl}
    \min & c^\top x_{\varsigma, \tau}^{(k)} + \varrho_{M_{\varsigma, \tau}(\theta^{(k)})} (- p^\top \min\{ v ,U \tilde d\} )\\
        \st & v \in \R_+^{\Ny},~ U\in\R_+^{\Ny\times \Nd} \\
            &  A v \leq x_{\varsigma, \tau}^{(k)},~ H U\leq I,  
    \end{array} 
    \right. 
\end{equation}
where the constant $c^\top x_{\varsigma, \tau}^{(k)}$ in the objective does not impact its optimal solution set. Let $(v_{\varsigma, \tau}^{(k)},U_{\varsigma, \tau}^{(k)})$ denote the optimal solution in~\eqref{eq:V-bar}, we define the TLDR
\be\label{eq:policy}
\mathbf{y}_{\varsigma, \tau}^{(k)}(d) \Let \min\{v_{\varsigma, \tau}^{(k)},U_{\varsigma, \tau}^{(k)} d \}.
\ee
\end{subequations}
The cost generated by the solution $(x_{\varsigma, \tau}^{(k)}, \mathbf{y}_{\varsigma, \tau}^{(k)})$ is denoted by:
\[
C(x_{\varsigma, \tau}^{(k)},\mathbf{y}_{\varsigma, \tau}^{(k)},b^{(k)},\theta^{(k)}) = c^\top x_{\varsigma, \tau}^{(k)} + \sup\limits_{\PP\in\mathcal{A}(\theta^{(k)}) } \varrho_{\PP}(-p^\top \mathbf{y}_{\varsigma, \tau}^{(k)}(\tilde d)) \ge \mathrm{Obj}(x_{\varsigma, \tau}^{(k)},b^{(k)},\theta^{(k)}),
\]
which serves as an upper bound for $\mathrm{Obj}(x_{\varsigma, \tau}^{(k)},b^{(k)},\theta^{(k)})$ due to 
the possible suboptimality of the second-stage TLDR $\mathbf{y}_{\varsigma, \tau}^{(k)}$, i.e., $-p^\top \mathbf{y}_{\varsigma, \tau}^{(k)}(d) \ge g(x_{\varsigma, \tau}^{(k)},d)$ for all $d\ge 0$. Consequently, we can  bound the difference $\mathrm{Obj}(x_{\varsigma, \tau}^{(k)},b^{(k)},\theta^{(k)}) - \mathrm{OPT}(b^{(k)},\theta^{(k)})$ from above by $C(x_{\varsigma, \tau}^{(k)},\mathbf{y}_{\varsigma, \tau}^{(k)},b^{(k)},\theta^{(k)}) - V_{0}(b^{(k)},\theta^{(k)}) $. We further separate this upper bound into two parts as
\be\label{eq:separation}
\begin{aligned}
     & C(x_{\varsigma, \tau}^{(k)},\mathbf{y}_{\varsigma, \tau}^{(k)},b^{(k)},\theta^{(k)}) - V_{0}(b^{(k)},\theta^{(k)})  \\
    = & \underbrace{C(x_{\varsigma, \tau}^{(k)},\mathbf{y}_{\varsigma, \tau}^{(k)},b^{(k)},\theta^{(k)}) - \bar V_{\varsigma, \tau}(b^{(k)},\theta^{(k)})}_{\mathrm{(I)}} + \underbrace{\bar V_{\varsigma, \tau}(b^{(k)},\theta^{(k)}) -  V_{0}(b^{(k)},\theta^{(k)})}_{\mathrm{(II)}}.
\end{aligned}
\ee
The next two results provide the upper bounds for $\mathrm{(I)}$ and $\mathrm{(II)}$, respectively.
\begin{proposition}[Loss upper bound (I)] \label{prop:loss-upper-bound}
    Suppose that Assumptions~\ref{a:risk-measure} and~\ref{a:ambiguity} are satisfied, and that the scaling~\eqref{eq:scaling} is in force. For the mechanism $M_{\varsigma,\tau}$ in~\eqref{eq:Mall} , the following hold:
    \begin{enumerate}\renewcommand{\labelenumi}{(\roman{enumi})}
        \item Under the expectation mechanism $M_0$,
         \[
C(x_{0}^{(k)},\mathbf{y}_{0}^{(k)},b^{(k)},\theta^{(k)}) - \bar V_{0}(b^{(k)}, \theta^{(k)}) \le (\frac{1}{2}+\alpha ) k^{\frac{s}{2 }} p^\top U_{0}^{(k)} \sigma.
    \] 
    \item Under the mechanism $M_{\varsigma,\tau}$ with $(\varsigma, \tau) \in ( 0,\sigma] \times (0,\tau_{\max}]$, 
    \[
C(x_{\varsigma, \tau}^{(k)},\mathbf{y}_{\varsigma, \tau}^{(k)},b^{(k)},\theta^{(k)}) - \bar V_{\varsigma, \tau}(b^{(k)},\theta^{(k)}) \le  (\frac{1}{2}+\alpha) k^{\frac{s}{2 }} p^\top U_{\varsigma, \tau}^{(k)} \sigma +  \sqrt{\frac{1-\tau}{\tau} }k^{\frac{s}{2 }} p^\top U_{\varsigma, \tau}^{(k)} \varsigma.
    \] 
    \end{enumerate}
\end{proposition}
\begin{proposition}[Value upper bound (II)]\label{prop:value-upper-bound}
   Suppose that Assumptions~\ref{a:risk-measure},~\ref{a:ambiguity} and~\ref{a:model-param}\ref{a:model-param-1} are satisfied, and the scaling~\eqref{eq:scaling} is in force. For the mechanism $M_{\varsigma,\tau}$ in~\eqref{eq:Mall}, the following hold:
    \begin{enumerate}\renewcommand{\labelenumi}{(\roman{enumi})}
    \item Under the expectation mechanism $M_0$, we have $
    \bar V_{0}(b^{(k)},\theta^{(k)}) = V_{0} (b^{(k)},\theta^{(k)})$.
    \item Under the mechanism $M_{\varsigma,\tau}$ with $(\varsigma, \tau) \in ( 0,\sigma] \times (0,\tau_{\max}]$,
    \[
     \bar V_{\varsigma, \tau}(b^{(k)},\theta^{(k)}) - V_{0} (b^{(k)},\theta^{(k)}) \le (\frac{1}{2}+\alpha )k^{\frac{s}{2}} p^\top U_{0}^{(k)} \sigma + \Big( \sqrt{\frac{1-\tau}{\tau}} + \sqrt{\frac{\tau}{1-\tau}} \Big) k^{\frac{s}{2}} p^\top \bar U_{\varsigma, \tau}^{(k)} \varsigma ,
    \]
     where $\bar U_{\varsigma, \tau} ^{(k)}\in \R_+^{\Ny\times\Nd}$ is a bounded matrix that satisfies $H\bar U_{\varsigma, \tau}^{(k)} \le I$. 
     \end{enumerate}
\end{proposition}

The last technical prerequisite of the proof is the following lemma, which asserts that the value function $V_{0}$ defined in~\eqref{eq:V-0} is positive homogeneous under the scaling scheme~\eqref{eq:scaling}.

\begin{lemma}[Positive homogeneity] \label{lemma:risk-pos-homo}
Under the scaling scheme~\eqref{eq:scaling}, the value function $V_{0}$ is positively homogeneous of degree 1 in the sense that $V_{0}(b^{(k)},\theta^{(k)})=kV_{0}(b^{(1)},\theta^{(1)})$ for any $k>0$.
\end{lemma}

The complete proof of Theorem~\ref{thm:opt}, which builds on the above results, appears in Appendix~\ref{appendix:moment} to maintain the flow of the presentation.

\section{Optimal Mechanisms for Data-driven Wasserstein Ambiguity Sets} \label{sec:wass}
This section focuses on the data-driven Wasserstein DRO setting. In practice, the forecasting team can access a finite set of $N$ samples, e.g., the historical realization of uncertain demand $\tilde d$, and devise a nominal distribution $\hat\PP$ based on these samples. The forecasting team constructs the ambiguity set as the Wasserstein metric neighborhood around $\hat\PP$. We use the $2$-Wasserstein metric $\Wass$, which computes the distance between two distributions $\PP$ and $\PP'$ by
\[
\Wass(\PP,\PP') \Let \min \limits_{\pi\in\Pi(\PP,\PP')} \big (\mathbb{E}_{\pi}[ \lVert \tilde d - \tilde d' \rVert^2] \big)^{\frac{1}{2}},
\]
where $\Pi(\PP,\PP')$ denotes the set of all joint contributions of random vectors $\tilde d$ and $\tilde d'$ with marginal distributions $\PP$ and $\PP'$, respectively. The Wasserstein distance represents the cost of an optimal mass transportation plan for moving a mass distribution described by $\PP$ to another described by $\PP'$~\citep{ref:villani2003topics}.
We formalize the definition of our $2$-Wasserstein ambiguity set in Assumption~\ref{a:ambiguity-wass}.

\begin{assumption}[Wasserstein ambiguity set]\label{a:ambiguity-wass}
Denote by $\theta = (\hat\PP,\varepsilon)$, where $\hat \PP$ is the nominal distribution, and $\varepsilon$ is the radius. We define the $2$-Wasserstein ambiguity set $\mathcal{A}(\theta)$ as
\be \label{eq:wass-set}
\mathcal A (\theta) \Let \left\{\PP \in \mathcal M_2: \Wass(\PP, \hat \PP) \le \varepsilon ,~ \PP(\tilde d \in\R_+^{\Nd}) =1\right\},
\ee
which contains all the probability distributions with non-negative support within the $2$-Wasserstein ball of radius $\varepsilon$ centered at $\hat \PP$.   
\end{assumption}

We first introduce a scaling scheme for the budget $b$ and the data-driven ambiguity information $\theta = (\hat\PP,\varepsilon)$ as follows. Fix a base $b^{(1)} = b \in\R_+^K$ and $\theta^{(1)} =( \hat\PP^{(1)}, \varepsilon^{(1)} )$. Suppose that $\hat\PP^{(1)}$ has mean $\hat\mu$ and covariance matrix $\hat\Sigma$. For the $k$-th problem, the budget parameter $b^{(k)}$ scales linearly to $kb$ and the ambiguity radius $\varepsilon^{ (k) }$ scales \textit{sub}linearly to $ k^{\frac{s}{2}} \varepsilon $ for some $s\in [1,2)$. Further, the nominal distribution $\hat\PP^{(k)}$ has mean $k\hat\mu$ and covariance matrix $\hat\Sigma^{(k)} \preceq k^s \hat\Sigma$.  Overall, the scaling scheme is:
\begin{subequations}\label{eq:wass-scaling}
\begin{equation}
    b^{(k)} = k b \quad \text{ and } \quad  \theta^{(k)} = (\hat\PP^{(k)} , \varepsilon^{(k)} ) ,
\end{equation}
where the moments $(\hat\mu^{(k)}, \hat\Sigma^{(k)})$ of $\hat\PP^{(k)}$ and the radius $ \varepsilon^{(k)}$ satisfy
\begin{equation}
     \hat\mu^{(k)}= k \hat \mu,~ \hat\Sigma^{(k)} \preceq k^s \hat \Sigma, ~ \varepsilon^{(k)} = k^{\frac{s}{2}}\varepsilon ~~\text{for an}~~ s\in [1,2).
\end{equation}
\end{subequations}

The condition that $\hat\Sigma^{(k)} \preceq k^s \hat \Sigma$ ensures that $\hat\Sigma_{ii}^{(k)} \le k^{s}\hat\Sigma_{ii}$, that is, the marginal variance scales sub-quadratically with $k$ as $s\in[1,2)$. The linear growth rate of mean $\hat\mu^{(k)}$ and sub-quadratic growth rate of marginal variance $\hat\Sigma_{ii}^{(k)}$ for all $i\in[\Nd]$ are, again, consistent with Taylor's law~\citep{ref:taylor1961aggregation}. The ambiguity radius $\varepsilon^{(k)}$ scales sub-linearly with $k$, ensuring that the ambiguity radius does not escalate as quickly as the budget and average demand. This, in turn, prevents overly conservative decisions in larger budget and market scenarios.

\subsection{An Optimal Mechanism} \label{sec:wass-optimal}
Consider a fixed base ambiguity parameter $\theta^{(1)} = (\hat\PP,\varepsilon)$, where $\hat\PP$ has mean vector $\hat\mu$ and covariance matrix $\hat\Sigma$. We introduce a vector $\hat\sigma \in\R_+^{\Nd}$ with $\hat\sigma_i \Let \sqrt{2\mathrm{Tr}(\hat\Sigma) + 2\varepsilon^2}$ for all $ i\in[\Nd]$. Under the scaling scheme~\eqref{eq:wass-scaling}, we consider the mechanism $M_{\varsigma, \tau}$ parameterized by two parameters $(\varsigma, \tau) \in \R_+^{\Nd} \times \R_+$:
\begin{subequations}\label{eq:Mall-wass}
\be \label{eq:Mtau-wass}
    M_{\varsigma, \tau}(\theta^{(k)}) \Let (1-\tau) \delta_{d_l} + \tau \delta_{d_h},
\ee
where the locations $d_l$ and $d_h$ are defined as
\be
d_l \Let k\hat\mu -  \sqrt{\tau/ (1-\tau) } k^{\frac{s}{2}}\varsigma \text{ and }d_h \Let k\hat\mu + \sqrt{(1-\tau)/\tau} k^{\frac{s}{2}} \varsigma .
\ee
Further, we restrict the parameters $(\varsigma, \tau)$ to the range
\begin{equation} \label{eq:param-range-wass}
    (\varsigma, \tau) \in [0,\hat\sigma] \times [0, \hat\tau_{\max}], \quad \text{ where } \hat\tau_{\max} \Let \frac{ \hat\gamma^2}{1+ \hat\gamma^2} \text{ and } \hat\gamma \triangleq \min_{i\in[\Nd]}  \frac{\hat\mu_{i}}{ \varsigma_{i} },
\end{equation}
\end{subequations}
and we emphasize that the above range depends only on the base vector $(\hat\PP,\varepsilon)$ but not on the scaling factor $k$. Similarly as before, we denote by $M_0(\theta^{(k)}) = \delta_{k\hat\mu}$ the expectation mechanism.

The following result asserts that the mechanism $M_{\varsigma, \tau}$ defined in~\eqref{eq:Mall-wass} is an optimal mechanism for the data-driven Wasserstein DRO setting.

\begin{theorem}[Optimal Mechanism - Wasserstein DRO] \label{thm:W-opt}
 Suppose that Assumptions~\ref{a:risk-measure},~\ref{a:model-param}, and~\ref{a:ambiguity-wass} are satisfied, that $\varrho_{\PP}(\cdot) \ge \E_{\PP}[\cdot]$, and that the scaling~\eqref{eq:wass-scaling} is in force. The mechanism $M_{\varsigma, \tau}$ defined in~\eqref{eq:Mall-wass} is optimal for the bilevel problem~\eqref{eq:bilevel}. This means the lower-level solution $x_{\varsigma, \tau}^{(k)}$ satisfies
    \[
    \lim\limits_{k \to \infty}~\frac{\mathrm{Obj}(x_{\varsigma, \tau}^{(k)}, b^{(k)}, \theta^{(k)})}{ \mathrm{OPT} (b^{(k)}, \theta^{(k)})} =  1.
    \]    
\end{theorem}
We introduce one more condition that $ \varrho_{\PP}(\cdot) \ge \E_{\PP}[\cdot] $, which requires that the risk value $\varrho_{\PP}(\tilde Z)$ is larger than the expected value $\E_{\PP}[\tilde Z]$ for any nonconstant random variable $\tilde Z$~\citep{ref:rockafellar2007coherent}. This condition is satisfied by many popular risk measures, including those in Table~\ref{table:risk}.

The proposed optimal mechanism for the data-driven setting maintains the same structure as its moment-based counterpart with $(\hat\mu,\hat\sigma)$, thereby inheriting desirable properties of simplicity and tractability. Specifically, the mechanism $M_{\varsigma,\tau}$ maps the uncertainty information $\theta = (\hat\PP,\varepsilon)$ to a simple two-point distribution that encodes both the moments of the nominal distribution and the Wasserstein ball radius. Notably, the formulation in~\eqref{eq:Mall-wass} matches that of~\eqref{eq:Mall}, ensuring the lower-level problem aligns with~\eqref{eq:V-tau-k} and can be solved efficiently. Next, we provide a proof sketch for its theoretical property as summarized in Theorem~\ref{thm:W-opt}.

\subsection{Proof Sketch of Theorem~\ref{thm:W-opt}}

The Wasserstein ambiguity set does not explicitly prescribe the mean and marginal variance of the distributions in the set. Nevertheless, recent advances in Wasserstein DRO have revealed a strong connection between the $2$-Wasserstein and the moment-based ambiguity set~\citep{ref:nguyen2021mean}, allowing us to control the moment information of the Wasserstein ambiguity set.

\begin{proposition}[Outer approximation]\label{prop:outer-approx}
Under the scaling scheme~\eqref{eq:wass-scaling} and $\mathcal{A}(\theta^{(k)})$ defined in~\eqref{eq:wass-set}, define $ \hat\sigma_{i} \Let \sqrt{ 2\mathrm{Tr}(\hat\Sigma) + 2\varepsilon^2 }$ for all $i \in[\Nd]$. It holds that
\[
\mathcal A(\theta^{(k)}) \subseteq \mathcal R(k\hat\mu,k^{\frac{s}{2}}\hat\sigma,k^{\frac{s}{2}}\varepsilon) \Let \left\{ \PP\in\mathcal{M}_2:
\begin{array}{l}
   \E_{\PP}[\tilde d] = \mu,~   \lVert \mu - k\hat \mu \rVert \le k^{\frac{s}{2}}\varepsilon  \\
   \E_{\PP}[(\tilde d_i - \mu_i)^2] \le k^s\hat\sigma_i^2 \quad \forall i \in[\Nd] 
\end{array}
\right\}.
\]
\end{proposition} 

One can intuitively recognize that the scaling scheme~\eqref{eq:wass-scaling} controls the (co)variance of the nominal distribution to grow sub-quadratically in $k$ and the radius of the Wasserstein ball to grow sub-linearly in $k$. Proposition~\ref{prop:outer-approx} shows that the scaling scheme~\eqref{eq:wass-scaling} further controls the marginal standard deviation growth of any distribution in the Wasserstein ambiguity to be sub-linear, which is reminiscent of the scaling scheme in the moment-based setting of Section~\ref{sec:moment}. However, extending the proofs from the moment-based ambiguity to the Wasserstein setting is nontrivial. The primary challenges arise from the absence of $M_{\varsigma,\tau}(\theta^{(k)}) \in \mathcal{A}(\theta^{(k)})$ and the flexibility in the first-moment value. To address these challenges, we outline the key steps of our proof below.

The proof strategy involves establishing a lower bound and an upper bound for $\mathrm{OPT}(b^{(k)},\theta^{(k)})$ and $\mathrm{Obj}(x_{\varsigma,\tau}^{(k)},b^{(k)},\theta^{(k)})$, respectively. For the lower bound construction, the condition that $\varrho_{\PP}(\cdot) \ge \E_{\PP}[\cdot]$ ensures that $V_0(b^{(k)},\theta^{(k)})$ still serves as a lower bound for the optimal cost $\mathrm{OPT}(b^{(k)},\theta^{(k)})$ which follows from the convexity of the second-stage cost function $g(x,\cdot)$ for all $x\in\R_+^{\Nx}$. Furthermore, it holds that $V_0(b^{(k)},\theta^{(k)}) = k V_0(b^{(1)},\theta^{(1)})$, implying that the optimal cost under the expectation mechanism decreases linearly with $k$ and tends towards negative infinity as $k$ increases. To bound $\mathrm{Obj}(x_{\varsigma,\tau}^{(k)},b^{(k)},\theta^{(k)})$ from above, we construct a feasible TLDR $\mathbf{y}_{\varsigma,\tau}^{(k)}$ corresponding to the first-stage decision $x_{\varsigma,\tau}^{(k)}$. The cost generated by $(x_{\varsigma,\tau}^{(k)}, \mathbf{y}_{\varsigma,\tau}^{(k)})$ can be further upper bounded by the linear negative term $V_0(b^{(k)},\theta^{(k)})$ plus a sub-linear positive term in $k$, which also tends towards negative infinity as $k$ increases. Consequently, the ratio between $\mathrm{Obj}(x_{\varsigma,\tau}^{(k)},b^{(k)}, \theta^{(k)})$ and $\mathrm{OPT}(b^{(k)}, \theta^{(k)})$ is larger than the ratio between the established upper and lower bounds, which converges to one as $k$ goes to infinity. By the sandwich theorem, we conclude that our mechanism $M_{\varsigma,\tau}$ achieves optimal performance as its objective value matches the theoretical optimum for problem~\eqref{eq:bilevel}. Appendix~\ref{appendix:wass} presents the complete derivation of the technical proofs.

\section{Numerical Studies}\label{sec:numerical}

We focus on the moment-based uncertainty information $\theta=(\mu,\sigma)$ in Section~\ref{sec:moment}. Our analysis consists of three main experiments: (i) Section~\ref{sec:asym-exp} compares our decentralized framework against the TLDR approximation of two-stage DRO problems; (ii) Section~\ref{sec:robustness} examines the robustness of our first-stage (lower-level) decisions facing distributional shifts between training and testing datasets; and (iii) Section~\ref{sec:real} evaluates the practical effectiveness of our approach using a real-world case study from a China-based retail company's two-echelon inventory problem.

Parts (i) and (ii) focus on assemble-to-order problems without product flexibility under which $H=I$. Appendix~\ref{appendix:instances} describes the details of the model and constructs 320 combinations of $(c,p, A,\theta)$ for analysis. All optimization problems were modeled with CVXPY in Python and solved by MOSEK 10.2.3 on a server with 128 GB RAM and a 2.6-GHz 104-core Intel Xeon processor. The source code for reproducing the results can be found in the supplementary material.

\subsection{Performance Comparison against DRO+TLDR}\label{sec:asym-exp}
The first experiment compares our decentralized framework against the truncated linear decision rule (TLDR) approximation, a classic method to solve two-stage problems.
The TLDR $\mathbf{y}: d \mapsto \{ v, Ud \}$ provides the necessary flexibility to generate feasible second-stage policies as analyzed in Section~\ref{sec:upper-bound}.
The TLDR approximation is tractable under a particular condition: $\varrho_{\PP}(\cdot) = \E_{\PP}[\cdot]$ and $H=I$. This unique case represents the only scenario where a comparison between TLDR and our decentralized framework is possible, making it the focus of our analysis. In such a case, $U=I$ becomes an optimal solution with respect to the variable $U$, simplifying the TLDR to $\mathbf{y}: d \mapsto \min\{ v, d \}$ with $v$ to be optimized.\footnote{We provide a formal proof of this claim in Proposition~\ref{prop:wc-dual}.} Consequently, the expected loss $\E_{\PP}[-p^\top \min\{ v, \tilde d \} ]$ is separable in $\tilde d$, which enables a tractable reformulation formalized in Proposition~\ref{prop:wc-dual}. 

\begin{proposition}[Second-order cone reformulation]\label{prop:wc-dual}
    Suppose that Assumption~\ref{a:ambiguity} is satisfied and $H=I\in\R_+^{N\times N}$, i.e., $\Ny=\Nd = N$. The optimal truncated linear decision rule for the second-stage problem~\eqref{eq:g} takes the form $\mathbf{y}:d \mapsto \min\{v,d\}$ with $v$ to be optimized. Furthermore, for $\theta=(\mu,\sigma)$, the worst-case expected loss has the reformulation:
     \begin{equation} \label{eq:risk-reformulation}
   \sup_{\PP\in\mathcal A(\theta)} \E_{\PP}[-p^\top \min\{ v ,\tilde d \}]  = \left\{
         \begin{array}{cll}
          \min  & \sum\limits_{i=1}^{N} \left[\mu_i \lambda_i + (\mu_i^2 + \sigma_i^2)s_i + r_i \right]  \\
            \st & \lambda \in \R^N,~s \in \R_+^N,~r\in\R^N,~q_1 \in \R_+^N,~q_2\in\R_+^{N} \\
                & (s_i + r_i + p_i v_i ,  q_{1,i} - \lambda_i , s_i - r_i - p_i v_{i} ) \in \mathcal Q^3 &\forall i \in [N] \\
                & (s_i + r_i ,  q_{2,i} - \lambda_i - p_i , s_i - r_i ) \in \mathcal Q^3 &\forall i \in [N],
        \end{array}\right.
    \end{equation}
where $\mathcal Q^3$ is the second-order cone in $\R^3$ defined as $\{ (t, z_1, z_2)\in \R^3: t \ge \lVert (z_1,z_2)\rVert\}$.
\end{proposition}

We investigate the asymptotic behavior of two approximation schemes and examine instances characterized by $(c,p, A,\theta)$ under two budget scenarios: $b\in\{0.5c^\top A\mu, 2c^\top A\mu \}$. With $(b^{(1)},\theta^{(1)}) = (b,\theta)$, we scale the budget and uncertainty moments $(b^{(k)},\theta^{(k)})$ according to the scheme outlined in~\eqref{eq:scaling} with $s=1$, where $k$ ranges from 1 to 50.

For the $k$-th problem, we find the optimal solution to the TLDR approximation problem:
\be\label{eq:tldr-dro}
(x_{\mathrm{TLDR}}^{(k)},v_{\mathrm{TLDR}}^{(k)} ) \in \left\{
\begin{array}{rl}
    \arg\min & c^\top x + \sup\limits_{\PP\in\mathcal A(\theta^{(k)})} \E_{\PP}[-p^\top\min\{v,\tilde d\} ]   \\
    \st  & x\in\mathcal X(b^{(k)}),~v\in\R_+^N,~Av\le x,
\end{array} \right.
\ee
by solving a second-order cone programming problem based on Proposition~\ref{prop:wc-dual}. Under the first-stage decision $x_{\mathrm{TLDR}}^{(k)}$, the optimal second-stage TLDR is $\mathbf{y}_{\mathrm{TLDR}}^{(k)}: d\mapsto \min\{v_{\mathrm{TLDR}}^{(k)},d\}$. The optimal objective value of problem~\eqref{eq:tldr-dro} equals $C(x_{\mathrm{TLDR}}^{(k)},\mathbf{y}_{\mathrm{TLDR}}^{(k)},b^{(k)},\theta^{(k)})$. 
Under the decentralized framework, we find the lower-level decision $x_{\varsigma, \tau}^{(k)}$ by solving problem~\eqref{eq:V-tau-k} under the mechanism $M_{\varsigma, \tau}(\theta^{(k)})$. 
We employ its second-stage decision rule $\mathbf{y}_{\varsigma,\tau}^{(k)}: d \mapsto \{ v_{\varsigma,\tau}^{(k)},d \}$ constructed in~\eqref{eq:V-bar}, whose associated cost $C(x_{\varsigma,\tau}^{(k)}, \mathbf{y}_{\varsigma,\tau}^{(k)}, b^{(k)}, \theta^{(k)})$ provides an upper bound for $\mathrm{Obj}(x_{\varsigma,\tau}^{(k)}, b^{(k)}, \theta^{(k)})$ and can be computed through Proposition~\ref{prop:wc-dual}. 

To quantify the performance gap between the two approaches, we compare the cost of $(x_{\varsigma,\tau}^{(k)}, \mathbf{y}_{\varsigma,\tau}^{(k)})$ and $(x_{\mathrm{TLDR}}^{(k)}, \mathbf{y}_{\mathrm{TLDR}}^{(k)})$ through the worst-case ratio~(WCR):
\[
    \text{WCR} \triangleq \frac{ C(x_{\varsigma,\tau}^{(k)}, \mathbf{y}_{\varsigma,\tau}^{(k)}, b^{(k)}, \theta^{(k)} ) }{C(x_{\mathrm{TLDR}}^{(k)}, \mathbf{y}_{\mathrm{TLDR}}^{(k)}, b^{(k)}, \theta^{(k)} )}.
\]
Since $(x_{\varsigma,\tau}^{(k)},v_{\varsigma,\tau}^{(k)})$ is feasible for problem~\eqref{eq:tldr-dro}, its suboptimality implies that for all $(\varsigma,\tau)\in[0,\sigma]\times [0,\tau_{\max}]$, we have $
C(x_{\varsigma,\tau}^{(k)}, \mathbf{y}_{\varsigma,\tau}^{(k)}, b^{(k)}, \theta^{(k)} ) \ge C(x_{\mathrm{TLDR}}^{(k)}, \mathbf{y}_{\mathrm{TLDR}}^{(k)}, b^{(k)}, \theta^{(k)} )$.
Hence, for profitable instances with negative cost values, the WCR is bounded above by one, and a higher value indicates a smaller gap between our decentralized approach and the TLDR approximation.

\begin{figure}[htbp]

    \centering
\includegraphics[width=\linewidth]{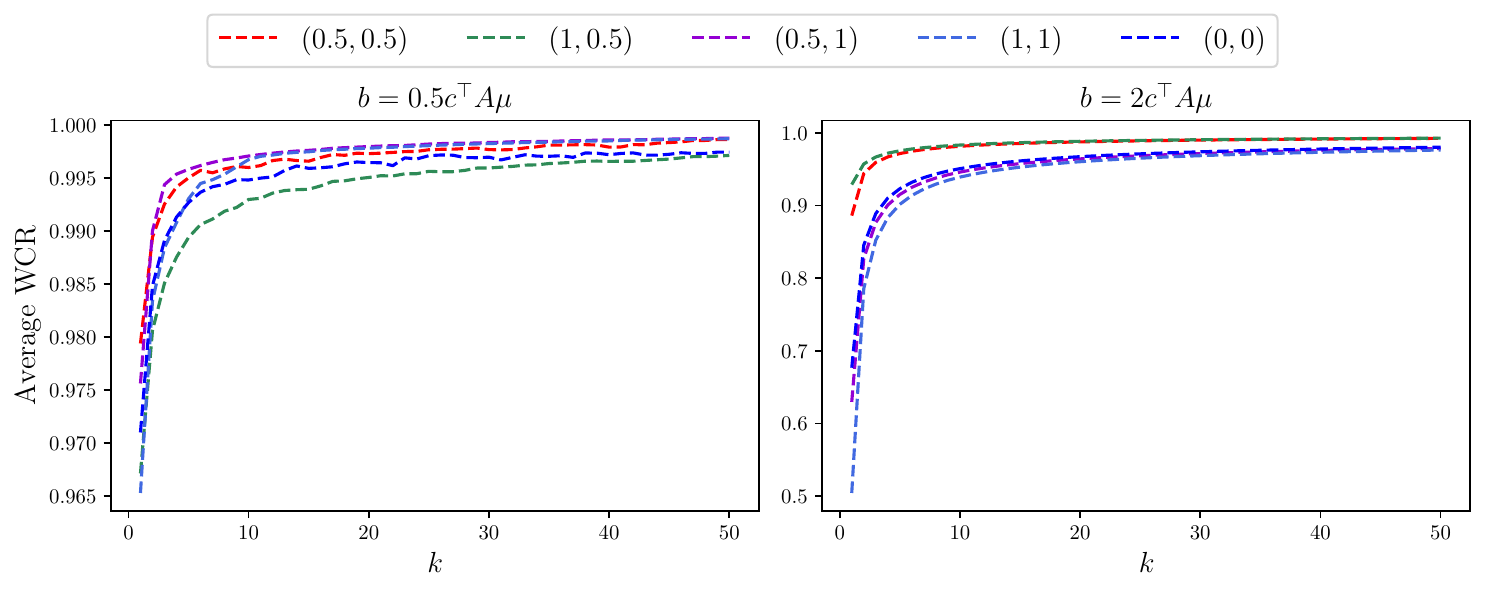}
    \caption{Average WCR comparison for varying mechanisms under two budget scenarios as the problem scale $k$ grows. The mechanism $M_{\varsigma,\tau}$ is parameterized by $(\kappa,\eta)$ with $(\varsigma, \tau) = (\kappa\sigma,\eta\tau_{\max})$. A higher WCR indicates a smaller gap between our decentralized approach and TLDR approximation.}
    \label{fig:comparison-exp}
\end{figure}
We examine five mechanisms: the expectation mechanism $M_0$ and four configurations with $\varsigma = \kappa \sigma$ for $\kappa \in \{0.5,1\}$, and $\tau = \eta\tau_{\max}$ for 
$\eta\in\{0.5,1\}$. Since $\tau_{\max}$ is nonlinear in $\varsigma$, $(\kappa,\eta) = (0.5,1)$ and $(1,0.5)$ yield different mechanisms. For consistency, we designate the expectation mechanism $M_0$ with the pair $(\kappa,\eta) = (0,0)$.
Figure~\ref{fig:comparison-exp} shows how the average WCR in 320 instances varies with the scaling factor $k$ under different mechanisms and two budget scenarios. The WCRs of all the mechanisms converge to one as the problem scale $k$ increases. Remarkably, in a small budget $b=0.5c^\top A\mu$, the average WCR achieves over 0.965 for all mechanisms and problem scales, indicating that our decentralized approach performs nearly as well as the TLDR approximation. Under a large budget $b=2c^\top A\mu$, mechanisms with $\tau = 0.5\tau_{\max}$ achieve WCR consistently above 0.9.
While TLDR approximation generates marginally better solutions for the specific case with $\varrho_{\PP}(\cdot) = \E_{\PP}[\cdot]$ and $H=I$, our decentralized framework offers two significant advantages. It naturally extends to general risk measures and constraint matrices. Further, it demonstrates superior computational efficiency, running about 100 times faster than the TLDR approximation method when handling 500 products; see Appendix~\ref{appendix:time-comparison} for an experiment comparing the running time. These advantages and the minimal performance gap establish our decentralized approach as a practical and versatile solution for real-world applications.

\subsection{Robustness and Data-Driven Performance}\label{sec:robustness}

The second experiment showcases that our approach is robust to distributional shifts. In this experiment, we make decisions based on the training data, then measure their performance on the test data. Here, the test data follows a \textit{different} distribution than the training data. This is a common benchmark when we expect that the training data may not reflect correctly the future data, which may arise when there are temporal shifts (such as seasonality), or population shifts (such as customer behaviors changing before and after COVID).

We generate data based on the truncated multivariate normal distributions, where the training and testing distributions share diagonal components in their covariance matrices but differ in non-diagonal components determined by their respective correlation coefficients $\rho_{\tr}$ and $\rho_{\te}$. Specifically, the covariance matrix $\Sigma$ satisfies $ \Sigma_{ii}= \sigma_i^2  $ for all $i$, and $\Sigma_{ij}= \rho \sigma_i \sigma_j $ for $i \neq j$. We examine two training scenarios $\rho_{\tr} \in \{ 0, 1\}$ against eleven testing scenarios $\rho_{\te} \in \{ 0, 0.1,\ldots,1\}$, which then define the covariance matrices $\Sigma_{\tr}$ and $\Sigma_{\te}$. 
Subsequently, we generate training datasets $\mathcal{D}_{\tr}$ of varying sample sizes $n_{\tr}\in\{100,200,400\}$ and testing datasets $\mathcal{D}_{\te}$ of fixed sample size $n_{\te} = 1000$. Notably, the number of products in our instances ranges from 2 to 30, with half of them comprising no more than 10 products. Consequently, even the smallest training sample size of 100 is considered relatively substantial.

We conduct $5$-fold cross-validation (CV) on the training dataset $\mathcal{D}_{\tr}$ to tune the mechanism parameters $(\varsigma,\tau)$ of $M_{\varsigma,\tau}$ defined in~\eqref{eq:Mall}. 
For each training fold with empirical moments $(\hat\mu,\hat\sigma)$, we parameterize $(\varsigma,\tau)$ as $(\varsigma,\tau) = (\kappa \hat\sigma, \eta\hat\tau_{\max})$, where $\hat\tau_{\max}$ is computed from $(\hat\mu,\varsigma)$. By searching the grid $\kappa,\eta\in\{0,0.05,0.10,\ldots,1\}$, we identify the optimal pair $(\kappa\opt,\eta\opt)$ that minimizes the average cost in the validation fold.
Consequently, the CV-tuned mechanism uses
$(\varsigma\opt,\tau\opt) = (\kappa\opt \sigma_{\tr}, \eta\opt \tau_{\max})$, where $(\mu_{\tr},\sigma_{\tr})$ are the empirical moments of the \textit{full} training dataset $\mathcal{D}_{\tr}$ and $\tau_{\max}$ is derived from $(\mu_{\tr},\varsigma\opt)$. The cross-validation procedure is computationally efficient, requiring only two linear programs~\eqref{eq:V-tau-k} and~\eqref{eq:V-bar} per iteration. Further details about the dataset generation and cross-validation procedure are provided in Appendix~\ref{appendix:simulation-robust}. 

To emphasize the dependency on $\mathcal D_{\tr}$, we denote by $ x_{\varsigma\opt, \tau\opt} (\mathcal D_{\tr})$ the first-stage solution derived from the mechanism $M_{\varsigma\opt,\tau\opt}$. We benchmark against the sample average approximation (SAA) method, which minimizes the cost on the training dataset $\mathcal D_{\tr}$ to yield the SAA solution
\be\label{eq:saa}
x_{\text{SAA}}(\mathcal D_{\tr}) \in \arg \min\limits _ {x\in\mathcal X(b)}  ~ c^\top x + \hat\PP_{\tr}\text{-CVaR}_{0.05}( g(x,\tilde d) ),
\ee
where $\hat\PP_{\tr}$ is the empirical distribution of the training dataset $\mathcal D_{\tr}$. For any solution $x$, we evaluate its performance on the testing dataset using $J(x, \mathcal{D}_{\te}) \Let c^\top x + \hat\PP_{\te}\text{-CVaR}_{0.05} ( g(x,\tilde d) )$, where $\hat\PP_{\te}$ is the empirical distribution of $\mathcal{D}_{\te}$. Let $J\opt(\mathcal D_{\te}) \Let J(x_{\text{SAA}}(\mathcal D_{\te}), \mathcal D_{\te})$ denote the optimal cost with perfect information, attained at $x_{\text{SAA}}(\mathcal D_{\te})$ which solves the SAA problem on the testing data.

We evaluate the magnitude of out-of-sample performance improvement over the SAA method using the metric
\begin{equation}\label{eq:robust-index}
\text{Robustness Index} \triangleq \frac{ J(x_{\varsigma\opt,\tau\opt }(\mathcal D_{\tr}) , \mathcal D_{\te})  -   J(x_{\text{SAA}} (\mathcal D_{\tr}), \mathcal D_{\te} ) }{ J\opt(\mathcal D_{\te}) }.
\end{equation}
As the goal is to minimize the cost, a smaller negative value of the nominator means that the solution $x_{\varsigma\opt,\tau\opt } ( \mathcal D_{\tr} )$ generates lower cost than the SAA solution $x_{\text{SAA}} (\mathcal D_{\tr})$. In profitable scenarios, the denominator $J\opt(\mathcal{D}_{\te})$ is strictly negative; hence, a larger positive Robustness Index indicates that our mechanism has better out-of-sample performance than the SAA method.
 \begin{figure}[!ht]
    \centering
        \includegraphics[width=\linewidth]{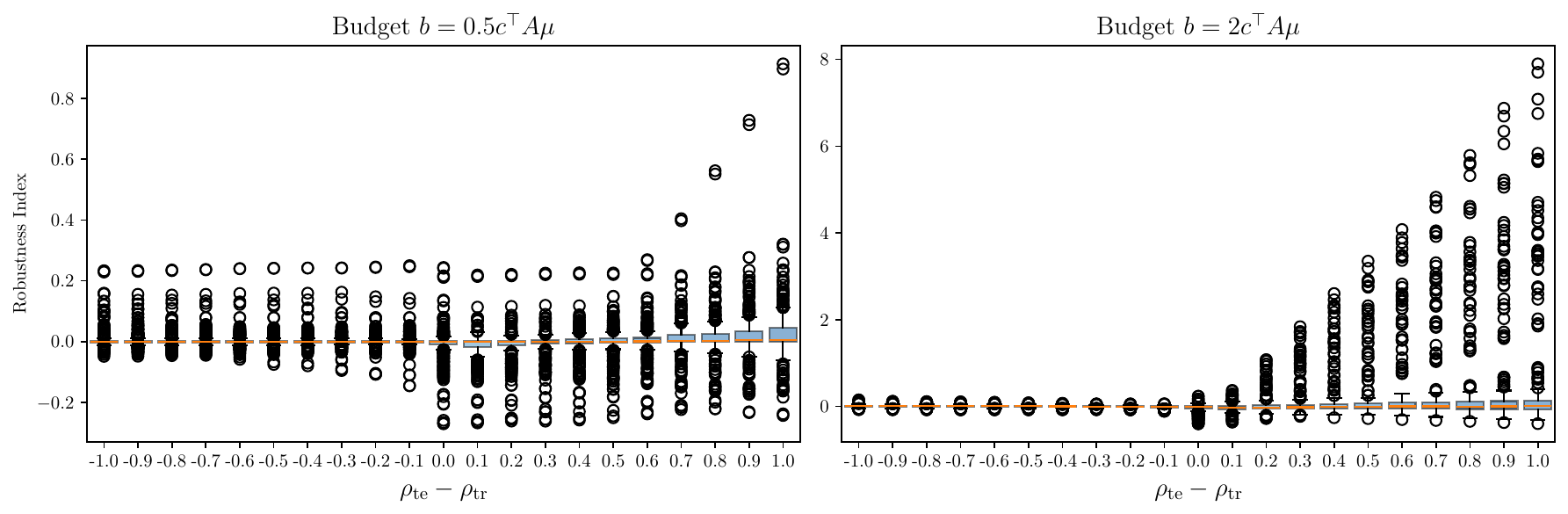}
    \caption{The Robustness Index of CV-tuned solution $x_{\varsigma\opt,\tau\opt}$ across 320 instances under different correlation coefficient shifts $\rho_{\te} - \rho_{\tr}$ for two budget scenarios when the training dataset has $n_{\tr}=200$ samples. A higher positive Robustness Index implies greater out-of-sample performance improvement over the SAA method.}
    \label{fig:robust-cvar}
    
\end{figure}
We compare the out-of-sample performance of our CV-tuned solution $x_{\varsigma\opt,\tau\opt}$ against the SAA solution $x_{\text{SAA}}$ under different correlation shifts $\rho_{\te} - \rho_{\tr} \in [-1, 1]$. We consider two budget scenarios $b=0.5c^\top A\mu$ and $b=2c^\top A\mu$ to analyze the impact of the budget magnitudes. 
Figure~\ref{fig:robust-cvar} focuses on training sample size $n_{\tr}=200$ and presents the Robustness Index of our CV-tuned solution $x_{\varsigma\opt,\tau\opt}$ under different budgets and correlation coefficient shifts across 320 instances. 
In the small budget scenario ($b=0.5c^\top A\mu$), the robustness index maintains a relatively contained range between $-0.2$ and $0.8$. While the central tendency demonstrates modest variation, notable outliers emerge, particularly for larger correlation shifts. An interesting asymmetric pattern appears in these outliers: though both positive and negative outliers exist, the positive outliers exhibit larger absolute deviations from the median compared to their negative counterparts.
The large budget scenario ($b=2c^\top A\mu$) displays an even more pronounced asymmetry in outlier distribution. The positive outliers reach substantially higher values (up to around 8) than the magnitude of negative outliers, with this asymmetry intensifying as correlation shifts increase, especially when $\rho_{\te} > \rho_{\tr} $. The asymmetric outlier pattern reveals that our CV-tuned solution has bounded downside risk compared to the SAA solution, as evidenced by the limited magnitude of negative outliers. Meanwhile, the large positive outliers, especially prominent in the high-budget scenario, indicate our solution can significantly outperform the SAA solution while maintaining limited potential losses.

\begin{table}[htbp]
  \centering
  \caption{The $p$-values of the two one-sided tests under different correlation coefficient shifts $ \rho_{\te} - \rho_{\tr}$ for two budget scenarios with training sample size $n_{\tr} =200$. Bold values indicate that the $p$-value is smaller than 0.05. }  \label{tab:signed-test}
    \begin{tabular}{cccccccc}
    \toprule
    \multirow{2}[2]{*}{$\rho_{\te} - \rho_{\tr} $ } & \multicolumn{2}{c}{Budget $b= 0.5c^\top A\mu$}       &       & \multicolumn{2}{c}{ Budget $b= 2c^\top A\mu$ } \\
\cmidrule{2-3}\cmidrule{5-6}     & $\mathcal{H}_+$    & $\mathcal{H}_-$   &   &  $\mathcal{H}_+$     & $\mathcal{H}_-$     \\
\cmidrule{1-3}\cmidrule{5-6}  
    -1.0    & \textbf{4.97E-04} & 1.00E+00 &       & \textbf{7.00E-03} & 9.96E-01 \\
    -0.9  & \textbf{3.19E-04} & 1.00E+00 &       & \textbf{7.00E-03} & 9.96E-01 \\
    -0.8  & \textbf{4.14E-05} & 1.00E+00 &       & \textbf{1.08E-02} & 9.93E-01 \\
    -0.7  & \textbf{8.43E-05} & 1.00E+00 &       & \textbf{1.08E-02} & 9.93E-01 \\
    -0.6  & \textbf{6.45E-04} & 1.00E+00 &       & \textbf{7.00E-03} & 9.96E-01 \\
    -0.5  & \textbf{1.52E-02} & 9.90E-01 &       & \textbf{7.00E-03} & 9.96E-01 \\
    -0.4  & \textbf{4.09E-02} & 9.72E-01 &       & 1.18E-01 & 9.11E-01 \\
    -0.3  & \textbf{1.91E-02} & 9.87E-01 &       & 4.06E-01 & 6.54E-01 \\
    -0.2  & \textbf{1.26E-02} & 9.92E-01 &       & 6.54E-01 & 4.06E-01 \\
    -0.1  & \textbf{2.88E-03} & 9.98E-01 &       & 9.93E-01 & \textbf{1.08E-02} \\
    0.0     & 3.00E-01 & 7.39E-01 &       & 1.00E+00 & \textbf{7.21E-13} \\
    0.1   & 6.30E-01 & 4.34E-01 &       & 1.00E+00 & \textbf{4.75E-05} \\
    0.2   & 3.39E-01 & 7.20E-01 &       & 9.97E-01 & \textbf{4.44E-03} \\
    0.3   & \textbf{2.45E-02} & 9.84E-01 &       & 9.66E-01 & \textbf{4.83E-02} \\
    0.4   & \textbf{6.86E-03} & 9.96E-01 &       & 8.48E-01 & 1.92E-01 \\
    0.5   & \textbf{5.75E-04} & 1.00E+00 &       & 7.10E-01 & 3.46E-01 \\
    0.6   & \textbf{4.83E-04} & 1.00E+00 &       & 4.69E-01 & 5.94E-01 \\
    0.7   & \textbf{5.68E-05} & 1.00E+00 &       & 4.06E-01 & 6.54E-01 \\
    0.8   & \textbf{2.52E-08} & 1.00E+00 &       & 4.06E-01 & 6.54E-01 \\
    0.9   & \textbf{1.28E-09} & 1.00E+00 &       & 3.46E-01 & 7.10E-01 \\
    1.0     & \textbf{1.11E-11} & 1.00E+00 &       & 3.46E-01 & 7.10E-01 \\
    \bottomrule
    \end{tabular} 
    
\end{table}

We further use one-sided sign tests to investigate whether the probability of our CV-tuned solution achieving \textit{lower} out-of-sample costs versus the SAA solution is higher or lower than $50\%$. Toward this goal, we set the null hypothesis as ``$\mathcal{H}_0$: our CV-tuned solution achieves \textit{lower} out-of-sample costs than the SAA solution with probability $50\%$". We test it against two alternative hypotheses as ``$\mathcal{H}_+$: our CV-tuned solution achieves \textit{lower} out-of-sample costs than the SAA solution with probability strictly higher than $50\%$" and its the reverse direction  ``$\mathcal{H}_-$: our CV-tuned solution achieves \textit{lower} out-of-sample costs than the SAA solution with probability strictly lower than $50\%$". Table~\ref{tab:signed-test} presents the $p$-values for two directions across different combinations of budgets and correlation shifts when the training dataset has $n_{\tr}=200$ samples. Lower $p$-values indicate stronger statistical evidence to reject the null hypothesis in favor of the alternative hypothesis.  
Under the small budget, our solution demonstrated statistically significant superiority over the SAA solution across most correlation shifts, with exceptions occurring at $\rho_{\te}-\rho_{\tr}\in[0.1,0.3]$. Notably, the consistently high $p$-values for $\mathcal{H}_-$ indicate that the SAA solution never dominates our solution across any correlation shifts. The large budget scenario reveals different patterns. Our solution has a higher probability of generating lower cost when $\rho_{\te} - \rho_{\tr} \le -0.5$ while the SAA solution outperforms when $\rho_{\te} - \rho_{\tr}\in[-0.1,0.3]$. In other cases, neither solution demonstrated clear dominance over the other.
Similar patterns emerge for other training sample sizes $n_{\tr}\in\{100,400\}$. The complete results and details can be found in Appendix~\ref{appendix:simulation-robust}.

\subsection{Real-World Experiment} \label{sec:real}
The third experiment benchmarks the performance of our framework in real-world deployment. We collaborate with a vibrant Chinese-based retail company committed to offering customers cost-effective products and has expanded to over 800 stores in more than 100 cities. It uses a central warehouse to supply multiple retail stores in each region. To manage inventory efficiently, the company must determine optimal warehouse stock levels \textit{before} observing actual customer demand. This scenario represents a two-echelon inventory problem with a single warehouse, as detailed in Appendix~\ref{appendix:inventory}.

We collect sales data that span six months from April to October 2024. A combination of a retail store and a product uniquely identifies each data entry. Due to the short shelf life of most products, the product list changes frequently, resulting in only 71 entries with consistent sales records over 184 days. To ensure a reliable analysis, we generate multiple instances for evaluation through a two-step process. Firstly, we randomly select 60 entries from the available 71 store-product combinations. Subsequently, for selected entries, we divide the dataset into testing and training samples by assigning a proportion $r\%$ of the data to the testing dataset $\mathcal{D}_{\te}$ and assigning the remaining $(100-r)\%$ to the training dataset $\mathcal{D}_{\tr}$. We repeat both the entries selection and dataset splitting procedures ten times, creating 100 instances of $(\mathcal{D}_{\tr},\mathcal{D}_{\te})$ for evaluation.

For each instance, we build the moment-based ambiguity set $\mathcal{A}(\hat\theta_{\tr})$, where $\hat{\theta}_{\tr} = (\hat{\mu}_{\tr}, \hat{\sigma}_{\tr})$ represents the estimated moments of the training dataset $\mathcal{D}_{\tr}$. We design the mechanism $M_{\varsigma,\tau}$ according to Section~\ref{sec:moment} and focus on the risk measure $\text{CVaR}_{0.05}$.
For parameter tuning, we perform 5-fold cross-validation on the training dataset $\mathcal{D}_{\tr}$ to determine the optimal parameter ratios $(\kappa\opt, \eta\opt)$. We set $ (\varsigma\opt, \tau\opt) = (\kappa\opt \hat{\sigma}_{\tr}, \eta\opt \hat{\tau}_{\max}) $ and find the lower-level solution $ x_{\varsigma\opt, \tau\opt} $ under the mechanism $ M_{\varsigma\opt, \tau\opt}(\hat{\theta}_{\tr}) $. We benchmark the out-of-sample performance of our CV-tuned solution $x_{\varsigma\opt,\tau\opt}$ against that of the SAA solution $ x_{\text{SAA}}$ obtained by minimizing the cost on the training dataset $\mathcal{D}_{\tr}$. As before, we consider two budget scenarios: $b=0.5c^\top A\hat\mu_{\tr}$ and $b=2c^\top A\hat\mu_{\tr}$, where $c$ is the practical unit ordering cost, $A$ is the topology matrix describing the store-product relationship (constructed in Appendix~\ref{appendix:inventory}), and $\hat{\mu}_{\tr}$ is the estimated average of $\mathcal{D}_{\tr}$. To examine how the size of the training sample affects performance, we vary the splitting proportion $r\%$ between three values: 25\%, 50\%, and 75\%. For each proportion, we generate 100 instances characterized by store-product entries and dataset-splitting results.
\begin{figure}
    \centering
\includegraphics[width=\linewidth]{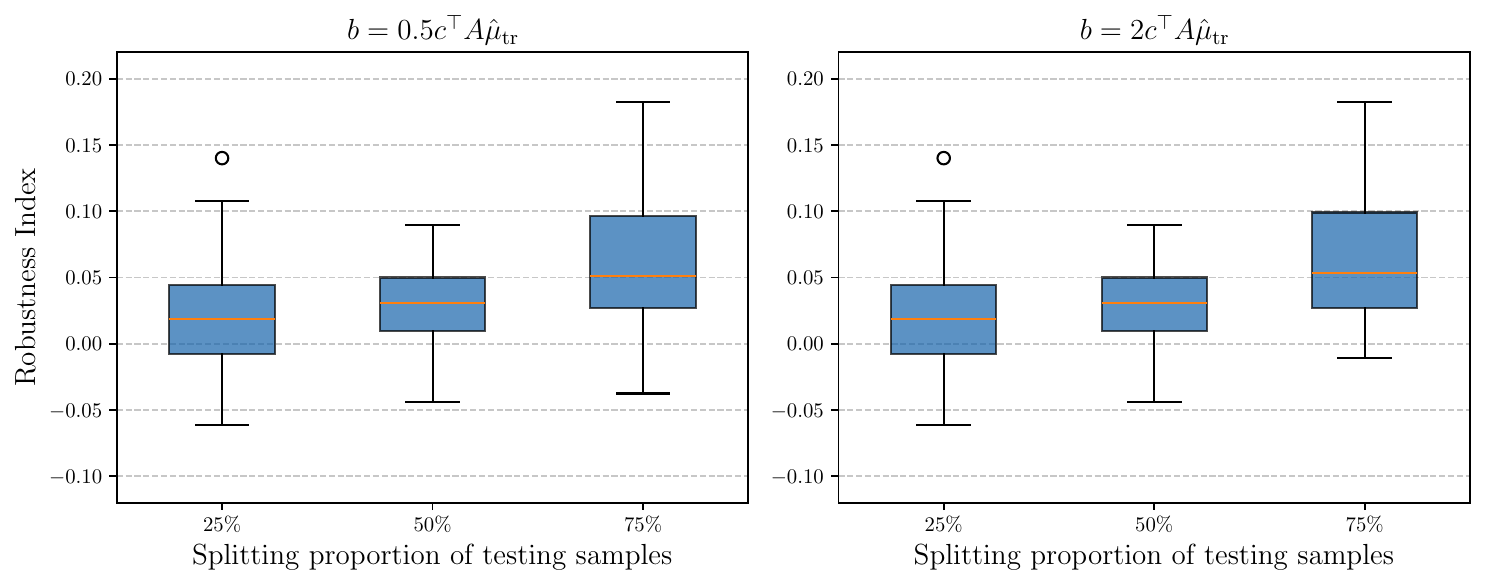}
    \caption{Robustness Index of CV-tuned solution $x_{\varsigma\opt,\tau\opt}$ under different proportions of testing samples for two budget scenarios. The training sample size decreases with the proportion of testing samples. A positive Robustness Index implies out-of-sample performance improvement over the SAA method; the higher, the better.}
    \label{fig:real-world}
    
\end{figure}

We begin our analysis with one-sided sign tests similar to the previous section. Our solution $x_{\varsigma\opt,\tau\opt}$ achieves $p$-values below 0.0004 across all splitting ratios in both budget scenarios. These statistically significant results confirm that the out-of-sample performance improvements are consistent and reliable across all 100 real-world instances. To quantify the improvement magnitude, we use the Robustness Index defined in~\eqref{eq:robust-index}, where larger positive values indicate larger out-of-sample performance improvement.
Figure~\ref{fig:real-world} shows box plots of the Robustness Index across different testing sample splitting proportions under two budget scenarios in Figure~\ref{fig:real-world}. Both budget scenarios exhibit similar patterns in their box plot distributions, suggesting stable performance across different budgets. In both figures, we can see that the Robustness Index exhibits positive values with a probability exceeding $70\%$, indicating that the solution $x_{\varsigma\opt, \tau\opt}$ has superior out-of-sample performance under risk-averse scenarios. Additionally, the Robustness Index increases as the testing proportion increases (i.e., as the training sample size decreases), highlighting the robustness of our solution when the training data are limited.

\section{Conclusion}\label{sec:conclusion}
We have introduced a novel decentralized framework for solving two-stage risk-averse distributionally robust optimization (DRO) problems, modeling the interaction between forecasting and operations teams as a Stackelberg game to reflect common organizational structures. By formalizing this game as a bilevel optimization problem, we have shown that a simple two-point distribution optimally solves the bilevel problem for both moment-based and Wasserstein ambiguity sets. This finding has significant practical implications: the operations team can solve a tractable optimization problem (reducible to a linear program under mild assumptions) to obtain asymptotically optimal first-stage decisions. This result effectively resolves the computational challenges typically associated with two-stage DRO problems. Using real sales data from a leading Chinese near-expiry goods retailer with integrated warehouse-retail operations, we have demonstrated that our solution method significantly outperforms traditional sample-average approximation methods in out-of-sample performance for risk-averse scenarios, validating its practical applicability. Our paper establishes the asymptotic optimality of lower-level solutions through a conservative bound for the critical ratio $\mathrm{Obj}/\mathrm{OPT}$ - a comprehensive analysis of its exact convergence rate warrants a separate investigation in future research.

\newpage
\bibliographystyle{informs2014}
\bibliography{ref}

\bigskip
\setcounter{equation}{0}
\numberwithin{equation}{section}

The appendix is organized as follows. Appendix~\ref{appendix:proof} contains detailed proofs of all technical results, with Appendix~\ref{appendix:moment} focusing on results related to moment-based ambiguity sets in Section~\ref{sec:proof-thm} and Appendix~\ref{appendix:wass} presenting proofs for Wasserstein ambiguity sets. Appendix~\ref{appendix:application} explores practical applications fit for our framework. Appendix~\ref{appendix:simulation} details our simulation studies, where Appendix~\ref{appendix:instances} constructs synthetic instances, Appendix~\ref{appendix:time-comparison} compares the computational time with TLDR approximation methods, and Appendix~\ref{appendix:simulation-robust} presents experimental details of robustness and data-driven performance.
\appendix
\section{Proofs of Technical Results} \label{appendix:proof}

\subsection{Proofs for Moment-based Ambiguity Set }\label{appendix:moment}
\begin{proof}[Proof of Theorem~\ref{thm:opt}.]
For any mechanism $M_{\varsigma,\tau}$ in~\eqref{eq:Mall}, let $x_{\varsigma,\tau}^{(k)}$ and $\mathbf{y}_{\varsigma, \tau}^{(k)}: d\mapsto \min\{ v_{\varsigma, \tau}^{(k)} , U_{\varsigma, \tau}^{(k)} d \}$ solve problem~\eqref{eq:V-tau-k} and~\eqref{eq:V-bar}, respectively.
As $x_{\varsigma,\tau}^{(k)}$ and $\mathbf{y}_{\varsigma, \tau}^{(k)}$ are feasible but not necessarily optimal solutions to problem~\eqref{eq:2stage-DRO} and the problem defining $g(x_{\varsigma,\tau}^{(k)}, \cdot)$, respectively, we have 
\be \label{eq:cost-ineq}
\mathrm{OPT}(b^{(k)}, \theta^{(k)}) \le \mathrm{Obj}(x_{\varsigma, \tau}^{(k)}, b^{(k)} ,\theta^{(k)}) \le C(x_{\varsigma, \tau}^{(k)},\mathbf{y}_{\varsigma, \tau}^{(k)},b^{(k)},\theta^{(k)}).
\ee 
We prove the optimality by showing that $C(x_{\varsigma, \tau}^{(k)},\mathbf{y}_{\varsigma, \tau}^{(k)},b^{(k)},\theta^{(k)})$ approaches $\mathrm{OPT}(b^{(k)}, \theta^{(k)})$ as $k$ goes to infinity, which implies that the limiting cost ratio between $\mathrm{Obj}(x_{\varsigma, \tau}^{(k)}, b^{(k)},\theta^{(k)})$ and $\mathrm{OPT}(b^{(k)}, \theta^{(k)})$ achieves its upper bound of one.

We start with the expectation mechanism $M_0$. For the first-stage decision $x_0^{(k)}$ and its second-stage decision rule $\mathbf{y}_{0}^{(k)}: d\mapsto \min\{ v_{0}^{(k)} , U_{0}^{(k)} d \}$, it holds that
\be\label{eq:C-upper}
C(x_{0}^{(k)},\mathbf{y}_{0}^{(k)},b^{(k)},\theta^{(k)}) \le  V_{0}(b^{(k)}, \theta^{(k)}) + (\frac{1}{2}+\alpha) k^{\frac{s}{2}} p^\top U_{0}^{(k)} \sigma,
\ee
which follows from Proposition~\ref{prop:loss-upper-bound} and $\bar V_{0}(b^{(k)}, \theta^{(k)}) =V_{0}(b^{(k)}, \theta^{(k)})$ in Proposition~\ref{prop:value-upper-bound}. 
By Lemma~\ref{lemma:risk-pos-homo} and Assumption~\ref{a:model-param}\ref{a:model-param-2}, we have $V_{0}(b^{(k)},\theta^{(k)}) = k V_{0}(b^{(1)},\theta^{(1)})$ and $V_{0}(b^{(1)},\theta^{(1)}) <0$. For large $k$, all terms in~\eqref{eq:C-upper} and~\eqref{eq:cost-ineq} become negative because $s<2$ and the linear negative term $ V_0(b^{(k)},\theta^{(k)})$ dominates. As all terms are negative, their ratios are equivalent to the ratios of their absolute values. We find that
\[
    1 \ge \frac{ \mathrm{Obj}(x_{0}^{(k)}, b^{(k)} ,\theta^{(k)}) }{ \mathrm{OPT}(b^{(k)}, \theta^{(k)})}   \ge \frac{C(x_{0}^{(k)},\mathbf{y}_{0}^{(k)},b^{(k)},\theta^{(k)})}{ \mathrm{OPT}(b^{(k)}, \theta^{(k)})}  \ge  \frac{k V_{0} (b^{(1)},\theta^{(1)}) + (\frac{1}{2}+\alpha)k^{\frac{s}{2}} p^\top U_{0}^{(k)}\sigma }{k V_{0} (b^{(1)},\theta^{(1)}) },
\]
where the first two inequality follows from the magnitude relationship
\[
| C(x_{0}^{(k)},\mathbf{y}_{0}^{(k)},b^{(k)},\theta^{(k)}) | \le | \mathrm{Obj}(x_{0}^{(k)}, b^{(k)} ,\theta^{(k)}) | \le | \mathrm{OPT}(b^{(k)}, \theta^{(k)})|,
\]
and the last inequality holds because 
\[
\begin{cases} |C(x_{0}^{(k)},\mathbf{y}_{0}^{(k)},b^{(k)},\theta^{(k)})|  \ge | k V_{0} (b^{(1)},\theta^{(1)}) + (\frac{1}{2}+\alpha)k^{\frac{s}{2}} p^\top U_{0}^{(k)}\sigma|  \\
|\mathrm{OPT}(b^{(k)}, \theta^{(k)})|  \le  k |V_{0} (b^{(1)},\theta^{(1)})| .
\end{cases}
\]
As the matrix $H\in\R_+^{\Nd \times \Ny }$ has no zero columns, the matrix $U\in\R_+^{\Ny\times \Nd}$ satisfying $HU\le I$ is restricted to a bounded region, which implies that $U_0^{(k)}$ remains bounded for all $k\ge 1$.
Together with the condition that $s < 2$, taking the limit as $k$ tends to infinity gives
\[
\lim\limits_{k\to\infty}\frac{k V_{0} (b^{(1)},\theta^{(1)}) + (\frac{1}{2}+\alpha)k^{\frac{s}{2}} p^\top U_{0}^{(k)}\sigma }{k V_{0} (b^{(1)},\theta^{(1)}) } = 1 \quad \implies \quad 
\lim\limits_{k\to\infty} \frac{ \mathrm{Obj}(x_{0}^{(k)}, b^{(k)} ,\theta^{(k)}) }{ \mathrm{OPT}(b^{(k)}, \theta^{(k)})}  = 1.
\]

Next, we prove the result for $M_{\varsigma, \tau}$ for $(\varsigma, \tau) \in  (0,\sigma] \times (0,\tau_{\max}]$. As $\varsigma\le \sigma$, applying Proposition~\ref{prop:loss-upper-bound} and Proposition~\ref{prop:value-upper-bound} to the separation~\eqref{eq:separation} yields that
\[
\begin{aligned}
& C(x_{\varsigma, \tau}^{(k)},\mathbf{y}_{\varsigma, \tau}^{(k)},b^{(k)},\theta^{(k)}) \\
\le & V_{0}(b^{(k)},\theta^{(k)}) + k^{\frac{s}{2}} p^\top \left[ (\frac{1}{2}+\alpha) (U_{\varsigma, \tau}^{(k)} + U_{0}^{(k)} ) + \sqrt{\frac{1-\tau}{\tau}} (U_{\varsigma, \tau}^{(k)} + \bar U_{\varsigma, \tau}^{(k)} ) + \sqrt{\frac{\tau}{1-\tau}}   \bar U_{\varsigma, \tau}^{(k)} \right] \sigma,
\end{aligned}
\]
where $\bar U_{\varsigma, \tau}^{(k)} \in\R_+^{\Ny\times\Nd}$ satisfies $H \bar U_{\varsigma, \tau}^{(k)} \le I$ and is specified in the proof of Proposition~\ref{prop:value-upper-bound}. 
Similar to the case under $M_0$, the upper bound for $ C(x_{\varsigma, \tau}^{(k)},\mathbf{y}_{\varsigma, \tau}^{(k)},b^{(k)},\theta^{(k)})$ constitutes a negative linear term and a positive sub-linear term as $s<2$, which implies that all the cost terms under $M_{\varsigma,\tau}$ are negative when 
$k$ is sufficiently large. Following the same steps as we proved for the expectation mechanism $M_0$, we have
\[
\begin{aligned}
   1 \ge & \frac{\mathrm{Obj}( x_{\varsigma, \tau}^{(k)}, b^{(k)}, \theta^{(k)} ) } { \mathrm{OPT}(b^{(k)},\theta^{(k)}) } \ge  \frac{ C(x_{\varsigma, \tau}^{(k)},\mathbf{y}_{\varsigma, \tau}^{(k)},b^{(k)},\theta^{(k)}) } { V_{0}(b^{(k)},\theta^{(k)}) }  \\
 \ge & \frac{V_{0}(b^{(k)},\theta^{(k)}) + k^{\frac{s}{2}} p^\top \left[ (\frac{1}{2}+\alpha) (U_{\varsigma, \tau}^{(k)} + U_{0}^{(k)} ) + \sqrt{\frac{1-\tau}{\tau}} (U_{\varsigma, \tau}^{(k)} + \bar U_{\varsigma, \tau}^{(k)} ) + \sqrt{\frac{\tau}{1-\tau}}   \bar U_{\varsigma, \tau}^{(k)} \right] \sigma  }{V_{0}(b^{(k)},\theta^{(k)})} \\
 = & 1 + \frac{ p^\top \left[ (\frac{1}{2}+\alpha) (U_{\varsigma, \tau}^{(k)} + U_{0}^{(k)} ) + \sqrt{\frac{1-\tau}{\tau}} (U_{\varsigma, \tau}^{(k)} + \bar U_{\varsigma, \tau}^{(k)} ) + \sqrt{\frac{\tau}{1-\tau}}   \bar U_{\varsigma, \tau}^{(k)} \right] \sigma  }{ V_{0}(b^{(1)},\theta^{(1)} )} k^{\frac{s}{2} -1 },
\end{aligned}
\]
where the coefficient of $k^{\frac{s}{2}-1}$ is finite and negative as $V_{0}(b^{(1)},\theta^{(1)} ) < 0$. The constraints $HU\le I$ and $U\ge 0$ ensure that $U_{\varsigma, \tau}^{(k)}$ and $\bar U_{\varsigma, \tau}^{(k)}$ are bounded for all $k$, as is $U_0^{(k)}$.
Given this boundedness and $s< 2$, taking the limit as $k\to\infty$ yields
\[
\lim\limits_{k\to\infty} \frac{\mathrm{Obj}( x_{\varsigma, \tau}^{(k)}, b^{(k)}, \theta^{(k)} ) } { \mathrm{OPT}(b^{(k)},\theta^{(k)}) } = 1.
\]
This completes the proof. 
\end{proof}

\begin{proof}[Proof of Proposition~\ref{prop:lower}.]
As $M_{0}(\theta^{(k)})$ concentrates at $k \mu$, we have 
\[
\varrho_{M_{0}(\theta^{(k)})}(-p^\top \mathbf{y}(\tilde d ) )  = -p^\top \mathbf{y}(k\mu),
\]
which leads to the problem 
\begin{equation*}
    V_{0}(b^{(k)},\theta^{(k)})=\left\{
\begin{array}{cl}
    \min &c^\top x - p^\top  y \\  
    \st & x \in \mathcal X (b^{(k)}),~y\in \R_+^{\Ny} \\
        & Ay\le x,~Hy\le k\mu.
\end{array} 
\right.
\end{equation*}
By construction, the ambiguity set $\mathcal A(\theta^{(k)})$ always contains the Dirac distribution that concentrates unit mass at $k\mu$, which coincides with $M_{0}(\theta^{(k)})$ following the notion from~\eqref{eq:Mall}. Let $(x\opt, \mathbf{y}\opt)$ be the optimal solution of problem~\eqref{eq:2stage-DRO} where $g(x\opt,d)$ is attained at $\mathbf{y}\opt(d)$ for all $d\in\R_+^{\Nd}$. We have 
\[ 
\begin{aligned}
    \mathrm{OPT}(b^{(k)},\theta^{(k)}) & = c^\top x\opt + \sup_{\PP \in \mathcal A(\theta^{(k)})} \varrho_{\PP}(-p^\top \mathbf{y}\opt (\tilde d) ) \ge c^\top x\opt + \varrho_{M_{0}(\theta^{(k)})}(-p^\top \mathbf{y}\opt (\tilde d) ) = c^\top x\opt -p^\top \mathbf{y}\opt(k\mu) ,
\end{aligned}
\]
where the inequality follows from $M_{0}(\theta^{(k)}) \in \mathcal A(\theta^{(k)})$.
By the feasibility of $(x\opt, \mathbf{y}\opt)$, it holds that $ x\opt \in \mathcal X(b^{(k)}),~A\mathbf{y}\opt(k\mu) \le x\opt$ and $H \mathbf{y}\opt(k\mu) \le k\mu$, which means $(x\opt, \mathbf{y}\opt(k\mu)) \in \R^{\Nx} \times \R^{\Ny}$ is a feasible solution to the problem defining $V_{0}(b^{(k)},\theta^{(k)})$ and its suboptimality leads to the result that $
\mathrm{OPT}(b^{(k)},\theta^{(k)}) \ge c^\top x\opt -p^\top \mathbf{y}\opt(k\mu) \ge  V_{0}(b^{(k)},\theta^{(k)})$. 
This completes the proof. 
\end{proof}

\begin{proof}[Proof of Proposition~\ref{prop:HU}.]
For the ``if'' part, when $(v,U) \in \R_+^{\Ny} \times \R_+^{\Ny\times\Nd} $, we have $\min\{ v, Ud\} \ge 0$ for $d\ge 0$. As $A$ and $H$ are non-negative matrices, for any $d\ge 0$, it holds that
\[
\left.
\begin{array}{l}
    A\min\{v,Ud\} \le Av \le x  \\
     H\min\{v,Ud\} \le HUd \le d
\end{array} \right\} \forall d\ge 0,
\]
which implies the feasibility of $\mathbf{y}: d\mapsto \min\{ v,Ud\}$.

Next, we prove the ``only if'' part. To satisfy $\min\{v,Ud\} \in\R_+^{\Ny}$, there must be $v\in\R_+^{\Ny}$ and $U\in\R_+^{\Ny\times\Nd}$, otherwise we can find a large $d$ such that $Ud$ has negative elements. Denote by $U_i$ the $i$-th row of $U$ for $i\in[\Ny]$. If $v_i  = 0$, we have $\min\{v_i , U_i  d\} \equiv 0$, which enables us to set $U_i  = 0$. Similarly, if $U_i  = 0$, we can set $v_i =0$. 
Note that if both $v_i $ and $U_i  $ are zero, then the $i$-th element of $\mathbf{y} (d)$ does not affect the first-stage decision and the second-stage cost $-p^\top \mathbf{y} (d)$. Therefore, to simplify the statement, we further assume that both $v_i $ and $U_i  $ are nonzero for all $i\in[\Ny]$.

 The feasibility of $(x , \mathbf{y} )$ implies that 
 $
 A\mathbf{y}  (d) = A\min\{ v , U  d \} \le x .
 $
 As $U_i $ is nonzero, there exists $d\ge 0$ large enough such that $U_i  d \ge v_i ~\forall i\in[\Ny]$, which leads to the constraint $Av \le x $. 
 For the constraint $H\mathbf{y} (d)\le d$, we consider $d = e_j/t$ where $e_j \in \R_+^{\Nd}$ denotes the $j$-th column of the identity matrix $I \in \R^{\Nd\times\Nd}$. Let $t$ be large enough such that $U  e_j/t \le v $, then we have
 \[
 H\mathbf{y} \left(\frac{e_j}{t} \right) = \frac{1}{t}HU  e_j \le \frac{1}{t}e_j,
 \]
 which means the $j$-th column of $HU $ cannot exceed $e_j$, i.e., $HU  \le I$. 
 
\end{proof}

The following lemma is critical for proving Proposition~\ref{prop:loss-upper-bound}.
\begin{lemma}[Risk upper bound] \label{lemma:risk-upper}
    Let $\theta^\prime = (\mu^\prime, \sigma^\prime) \in \R_+^{2\Nd}$ be an arbitrary moment parameter. Under  Assumptions~\ref{a:risk-measure} and~\ref{a:ambiguity}, for any $p \in \R_+^{\Nd},v \in \R^{\Ny}$ and $U\in\R_+^{\Ny\times\Nd}$, we have
    \[
    \sup\limits_{\PP \in \mathcal A(\theta^\prime)} \varrho_{\PP}(p^\top \max\{0, v - U \tilde d\})\leq (\frac{1}{2}+\alpha) p^\top  U \sigma^\prime+ \sum_{i \in [\Nd]} p_i (v - U \mu^\prime)_i \mathbb{I}_{( v - U \mu^\prime)_i > 0}.
    \]
\end{lemma}

\begin{proof}[Proof of Lemma~\ref{lemma:risk-upper}.]
Pick any $\PP \in \mathcal A(\theta^\prime)$, we have $\E_{\PP}[U \tilde d] = U \mu^\prime$ and $\mathrm{Var}((U \tilde d)_i) \le (  {U_i} \sigma^\prime)^2$,
where $(U \tilde d)_i$ is the $i$-th element of $U \tilde d$, and $U_i$ is the $i$-th row vector of $U$. Subsequently, we have 
$\E_{\PP}[v - U \tilde d] = v - U \mu^\prime$,
and 
$\mathrm{Var}( (v - U \tilde d)_i) \le (  {U_i} \sigma^\prime)^2$. By~\citet[Proposition~1]{ref:gallego1992minmax}, for any random variable $\tilde \zeta$, it holds that
\begin{equation*}
    \mathbb{E}[\max\{0, \tilde \zeta\}]\leq \frac{ \sqrt{\mathrm{Var}(\tilde \zeta )+ \E[\tilde \zeta ]^2} + \E[\tilde \zeta ] }{2}.
\end{equation*}
Let $\tilde \xi = \max \{0, v- U \tilde d \}$ and then $\tilde \xi _i = \max \{0, (v- U \tilde d)_i\}$. We have
\[
\begin{aligned}
    \E[\tilde \xi_i ] & \le \frac{\sqrt{(  {U_i} \sigma^\prime)^2+(v - U \mu^\prime)_i^2}+(v - U \mu^\prime)_i}{2} \\
    & \le \frac{  {U_i} \sigma^\prime +  | (v - U \mu^\prime)_i| +(v - U \mu^\prime)_i}{2} \\
    &= \frac{  {U_i}  \sigma^\prime +  2(v - U \mu^\prime)_i \mathbb{I}_{( v - U \mu^\prime)_i > 0}}{2}.
\end{aligned}
\]
The first inequality follows from~\citet[Proposition~1]{ref:gallego1992minmax} and $\mathrm{Var}((v- U \tilde d)_i) \le (  {U_i} \sigma^\prime)^2$. The second inequality follows from $U_i \sigma^\prime \ge 0$ as $U\in\R_+^{\Ny \times \Nd }$. Moreover, the variance of $\tilde \xi$ can be upper bounded by $( {U_i} \sigma^\prime)^2$ as $\mathrm{Var}(\max \{0, (v- U \tilde d)_i\}) \le \mathrm{Var}((v- U \tilde d)_i) \le (  {U_i} \sigma^\prime)^2 $. 

Consider now a vector $ p\in\R_+^{\Nd}$, let $m$ and $s^2$ be the mean and variance of $p^\top \tilde \xi$, respectively. Based on the derived upper bound for the mean and variance of $\tilde \xi $, under any $\PP \in \mathcal A(\theta^\prime)$, we have
\[
m\in\Big(0,\frac{1}{2}\sum\limits_{i\in[\Nd]} p_i ( {U_i} \sigma^\prime +2(v - U \mu^\prime )_i \mathbb{I}_{( v - U \mu^\prime)_i > 0})\Big]\quad\mathrm{and}\quad s^2\in\Big(0,( \sum\limits_{i\in[\Nd]} p_i  {U_i} \sigma^\prime )^2\Big].
\]
Next, we bound the worst-case risk $\sup_{\PP \in \mathcal A(\theta^\prime)}~\varrho_{\PP}(p^\top \tilde \xi) $ using the notion of standard risk coefficient. From the positive homogeneity and translation invariance of $\{\varrho_{\PP} \} $, it holds that
\[
\begin{aligned}
     \sup_{\PP \in \mathcal A(\theta^\prime)}~\varrho_{\PP}(p^\top \tilde \xi) & = \sup_{\PP \in \mathcal A(\theta^\prime)}~m+s\varrho_{\PP}(\frac{1}{s}(p^\top \tilde \xi -m) )  \\
      &\le \sup_{\PP \in \mathcal A(\theta^\prime)}~m+\alpha s\\
      &\le (\frac{1}{2}+\alpha )\sum\limits_{i\in[\Nd]} p_i  {U_i} \sigma^\prime+ \sum\limits_{i\in[\Nd]} p_i (v - U \mu^\prime)_i \mathbb{I}_{( v - U \mu^\prime)_i > 0} , 
    \end{aligned}
\]
where the first inequality follows because $(p^\top \tilde \xi -m)/ s$ has zero mean and unit variance and $\alpha$ captures the worst-case risk measure of standard random variable. The second inequality follows from the uniform upper bound for $(m,s^2)$ under any $\PP\in \mathcal{A} (\theta^\prime) $. 
Therefore, we have 
\[ 
     \sup_{\PP \in \mathcal A(\theta^\prime)}~\varrho_{\PP}(p^\top \max\{0, v - U \tilde d\})=\sup_{\PP \in \mathcal A(\theta^\prime)}~\varrho_{\PP}(p^\top \tilde \xi ) \le (\frac{1}{2}+\alpha) p^\top U \sigma^\prime + \sum\limits_{i\in[\Nd]} p_i (v - U \mu^\prime)_i \mathbb{I}_{( v - U \mu^\prime)_i > 0},
\]
which completes the proof. 
\end{proof}

\begin{proof}[Proof of Proposition~\ref{prop:loss-upper-bound}.]
We first rewrite $C( x_{\varsigma, \tau}^{(k)},\mathbf{y}_{\varsigma, \tau}^{(k)},b^{(k)},\theta^{(k)} )$ as
\[ 
\begin{aligned}
C( x_{\varsigma, \tau}^{(k)},\mathbf{y}_{\varsigma, \tau}^{(k)},b^{(k)},\theta^{(k)} ) &= c^\top x_{\varsigma, \tau}^{(k)} + \sup_{\PP \in \mathcal A(\theta^{(k)})}~\varrho_{\PP}(- p^\top  \min\{v_{\varsigma, \tau}^{(k)}, U_{\varsigma, \tau}^{(k)}\tilde d\}) \\
&= c^\top x_{\varsigma, \tau}^{(k)} + \sup_{\PP \in \mathcal A(\theta^{(k)})}~\varrho_{\PP}(- p^\top  v_{\varsigma, \tau}^{(k)} - p^\top  \min\{ 0, U_{\varsigma, \tau}^{(k)} \tilde d - v_{\varsigma, \tau}^{(k)}\}) \\
&= c^\top x_{\varsigma, \tau}^{(k)} - p^\top v_{\varsigma, \tau}^{(k)} + \sup_{\PP \in \mathcal A(\theta^{(k)})}~\varrho_{\PP}(- p^\top  \min\{ 0, U_{\varsigma, \tau}^{(k)} \tilde d - v_{\varsigma, \tau}^{(k)} \}) \\
&= c^\top x_{\varsigma, \tau}^{(k)} - p^\top v_{\varsigma, \tau}^{(k)} + \sup_{\PP \in \mathcal A(\theta^{(k)})}~\varrho_{\PP}( p^\top \max\{ 0, v_{\varsigma, \tau}^{(k)} - U_{\varsigma, \tau}^{(k)} \tilde d \}) ,
\end{aligned} 
\]
where the first equality follows from the definition of $C$ and the third equality follows from the translation invariance of $\{ \varrho_{\PP} \}$. Observe that $- v_{\varsigma, \tau}^{(k)} \le -\min\{ v_{\varsigma, \tau}^{(k)},U_{\varsigma, \tau}^{(k)} d \} $, we have $ \varrho_{M_{\varsigma, \tau}(\theta^{(k)})}(-p^\top\min\{v_{\varsigma, \tau}^{(k)} ,U_{\varsigma, \tau}^{(k)} \tilde d \}) \ge - p^\top v_{\varsigma, \tau}^{(k)}$ as $p$ is non-negative and $\varrho$ is monotonic. Substituting $- p^\top v_{\varsigma, \tau}^{(k)}$ with $\varrho_{M_{\varsigma, \tau}(\theta^{(k)})}(-p^\top\min\{v_{\varsigma, \tau}^{(k)} ,U_{\varsigma, \tau}^{(k)} \tilde d \})$ yields
\be\label{eq:to-be-used-1}
\begin{aligned}
& C( x_{\varsigma, \tau}^{(k)},\mathbf{y}_{\varsigma, \tau}^{(k)},b^{(k)},\theta^{(k)} ) \\
 \le &c^\top x_{\varsigma, \tau}^{(k)} + \varrho_{M_{\varsigma, \tau}(\theta^{(k)})}(-p^\top\min\{v_{\varsigma, \tau}^{(k)},U_{\varsigma, \tau}^{(k)} \tilde d \}) + \sup_{\PP \in \mathcal A(\theta^{(k)})}~\varrho_{\PP}( p^\top \max\{ 0, v_{\varsigma, \tau}^{(k)} - U_{\varsigma, \tau}^{(k)} \tilde d \})   \\
 = & \bar V_{\varsigma, \tau}(b^{(k)},\theta^{(k)}) +  \sup_{\PP \in \mathcal A(\theta^{(k)})}~\varrho_{\PP}( p^\top \max\{ 0, v_{\varsigma, \tau}^{(k)} - U_{\varsigma, \tau}^{(k)} \tilde d \}) \\
 \le & \bar V_{\varsigma, \tau}(b^{(k)},\theta^{(k)}) + (\frac{1}{2}+\alpha) k^{\frac{s}{2}} p^\top U_{\varsigma, \tau}^{(k)} \sigma + \sum\limits_{i\in[\Nd]} p_i (v_{\varsigma, \tau}^{(k)} - kU_{\varsigma, \tau}^{(k)} \mu)_i \mathbb{I}_{( v_{\varsigma, \tau}^{(k)} - k U_{\varsigma, \tau}^{(k)} \mu)_i > 0},
\end{aligned} 
\ee
where the equality follows from the definition of $\bar V_{\varsigma, \tau}(b^{(k)},\theta^{(k)})$, and the second inequality follows from Lemma~\ref{lemma:risk-upper} with $\theta^{(k)} = (k\mu, k^{\frac{s}{2}}\sigma)$. Observe the problem defining $\bar V_{\varsigma, \tau}(b^{(k)},\theta^{(k)})$, its optimal solution $(v_{\varsigma, \tau}^{(k)}, U_{\varsigma, \tau}^{(k)} )$ must satisfy $v_{\varsigma, \tau}^{(k)} \le U_{\varsigma, \tau}^{(k)} d_h$, otherwise we can set $v$ to be $ U_{\varsigma, \tau}^{(k)} d_h$ and then $( U_{\varsigma, \tau}^{(k)} d_h, U_{\varsigma, \tau}^{(k)} )$ is a feasible solution with smaller objective, violating the optimality of $(v_{\varsigma, \tau}^{(k)}, U_{\varsigma, \tau}^{(k)} )$. As a consequence, for each $i\in[\Nd]$, we have
\[
(v_{\varsigma, \tau}^{(k)} - k U_{\varsigma, \tau}^{(k)} \mu)_i \mathbb{I}_{( v_{\varsigma, \tau}^{(k)} - k U_{\varsigma, \tau}^{(k)} \mu)_i > 0} = \max\{ 0 , (v_{\varsigma, \tau}^{(k)} - k U_{\varsigma, \tau}^{(k)} \mu)_i \} \le \max\{ 0 , [ U_{\varsigma, \tau}^{(k)} (  d_h - k\mu )]_i \},
\]
where the location of $d_h$ is
\[
d_h \Let  \begin{cases} 
   k\mu & \text{if } \tau = 0 \text{ or } \varsigma=0 ,\\
   k\mu + \sqrt{(1-\tau)/\tau}k^{\frac{s}{2}}\varsigma & \text{otherwise}.
\end{cases}
\]
We can see that $U_{\varsigma, \tau}^{(k)} (  d_h - k\mu ) \in \R_+^{\Ny}$ and its non-negative part is itself. Therefore, we have
\[
(v_{\varsigma, \tau}^{(k)} - U_{\varsigma, \tau}^{(k)} \hat\mu^{(k)})_i \mathbb{I}_{( v_{\varsigma, \tau}^{(k)} - k U_{\varsigma, \tau}^{(k)} \mu  )_i > 0} \le  \begin{cases}
   0 & \text{if } \tau = 0 \text{ or } \varsigma=0  ,\\
  \sqrt{(1-\tau) / \tau} k^{\frac{s}{2}}( U_{\varsigma, \tau}^{(k)}\varsigma )_i & \text{otherwise}.
\end{cases}
\]
Replacing the last term $\sum_{i \in [\Nd]} p_i (v_{\varsigma, \tau}^{(k)} - k U_{\varsigma, \tau}^{(k)} \mu)_i \mathbb{I}_{( v_{\varsigma, \tau}^{(k)} - k U_{\varsigma, \tau}^{(k)} \mu)_i > 0}$ using the above upper bound completes the proof. 
\end{proof}

\begin{proof}[Proof of Proposition~\ref{prop:value-upper-bound}.]

We first show that the mechanism $M_{\varsigma, \tau}$ defined in~\eqref{eq:Mall} satisfies $M_{\varsigma, \tau}(\theta^{(k)}) \in \mathcal{A}(\theta^{(k)})$ for all $k\ge 1$. Under $M_{\varsigma, \tau}(\theta^{(k)})$, the random vector $\tilde d$ has mean $k\mu$ and marginal variance $k^s \varsigma^2 \le k^s \sigma^2 $ because $\varsigma \le \sigma$. Further, the support set is non-negative as the low demand realization satisfies
\[
d_l \ge k\mu - \sqrt{\frac{\tau_{\max}}{1-\tau_{\max}}} k^{\frac{s}{2}}\varsigma = k\mu - k^{\frac{s}{2}}\varsigma\sqrt{ \min_{i\in[\Nd]}  \frac{\mu_{i}^2}{ \varsigma_{i}^2 } }  \ge (k -  k^{\frac{s}{2}})\mu \ge 0,
\]
where the first inequality holds because $\sqrt{\tau/(1-\tau)}$ increases with $\tau$ and we restrict $\tau\in[0,\tau_{\max}]$, the second equality follows from the definition of $\tau_{\max}$, and the last inequality follows because $s\in[1,2)$.

We start with the proof of $\bar V_{0} (b^{(k)},\theta^{(k)}) = V_{0} (b^{(k)},\theta^{(k)})$. As $\mathbf{y}_0^{(k)}(k\mu)$ is suboptimal to problem~\eqref{eq:V-tau-k}, it holds that $\bar V_{0}(b^{(k)},\theta^{(k)}) \ge V_{0}(b^{(k)},\theta^{(k)})$. We will prove the equality by construction. Let $V_{0}(b^{(k)},\theta^{(k)})$ be attained at $(x_{0}^{(k)}, y_{0}^{(k)} )$, then we have $A y_{0}^{(k)} \le x_{0}^{(k)}$ and $H y_{0}^{(k)} \le \hat\mu^{(k)}$. Under Assumption~\ref{a:model-param}\ref{a:model-param-1}, the matrix $H$ has one positive element in each column, and then the constraint $H y_{0}^{(k)} \le \hat\mu^{(k)}$ implies
\[
\sum\limits_{k:H_{ij} > 0} H_{ij} y_{0,j}^{(k)} \le k\mu_i \quad \forall i\in[\Nd],
\]
where $y_{0,j}^{(k)}$ denotes the $j$-th element of the vector $y_{0}^{(k)}$.
Define
\[
(\bar U_{0}^{(k)})_{ji} \Let 
\begin{cases}
  y_{0,j}^{(k)} / (k\mu_i) & \text{if } H_{ij}>0,  \\
    0 &\text{otherwise},
\end{cases}
\]
then we have $H\bar U_{0}^{(k)} \le I$ and $k \bar U_{0}^{(k)} \mu = y_{0}^{(k)}$. For the threshold vector, we set it to be $\bar v_{0}^{(k)} = y_{0}^{(k)} $. Define the second-stage policy as $\bar {\mathbf{y}}^{(k)}_{0}: d \mapsto \min\{\bar v_{0}^{(k)}, \bar U_{0}^{(k)} d\}$, which is feasible problem~\eqref{eq:V-bar}. By the suboptimality, we have
\[
\bar V_{0}(b^{(k)},\theta^{(k)}) \le c^\top x_{0}^{(k)} + \varrho_{M_{0}(\theta^{(k)})}(-p^\top \min\{ \bar v_{0}^{(k)}, \bar U_{0}^{(k)}\tilde d \} ) = c^\top x_{0}^{(k)} - p^\top y_0^{(k)} = V_{0}(b^{(k)},\theta^{(k)}),
\]
which implies that $\bar V_{0}(b^{(k)} ,\theta^{(k)} ) = V_{0}(b^{(k)} ,\theta^{(k)} )$ and $\bar V_{0}(b^{(k)} ,\theta^{(k)} )$ is attained at $(\bar v_0^{(k)} , \bar U_0^{(k)})$.

For $(\varsigma, \tau)\in (0,\sigma]\times(0,\tau_{\max}] $, we introduce the following problem
 \begin{equation}\label{eq:V-hat}
 \hat V_{\varsigma, \tau} (b^{(k)},\theta^{(k)}) \Let \left\{
    \begin{array}{cl}
    \min & c^\top x + \varrho_{M_{\varsigma, \tau}(\theta^{(k)})} (- p^\top \min\{ v ,U \tilde d \}) \\
        \st & x \in \R_+^{\Ny},~v \in \R_+^{\Ny},~ U\in\R_+^{\Ny\times \Nd} \\
            & A v \leq x,~ H U\leq I.  
    \end{array} 
    \right. 
\end{equation}
In the above problem, $x$ is also a decision variable compared with the problem defining $\bar V_{\varsigma, \tau}(b^{(k)},\theta^{(k)})$ in~\eqref{eq:V-bar}. For that reason, $\hat V_{\varsigma, \tau}(b^{(k)},\theta^{(k)}) \le \bar V_{\varsigma, \tau}(b^{(k)},\theta^{(k)})$. Compared with $V_{\varsigma, \tau}(b^{(k)},\theta^{(k)})$, problem~\eqref{eq:V-hat} restricts the second-stage policy to be TLDR, which implies that $V_{\varsigma, \tau}(b^{(k)},\theta^{(k)}) \le \hat V_{\varsigma, \tau}(b^{(k)},\theta^{(k)}) $. Therefore, we have the following order:
\[
V_{\varsigma, \tau}(b^{(k)},\theta^{(k)}) \le \hat V_{\varsigma, \tau}(b^{(k)},\theta^{(k)}) \le \bar V_{\varsigma, \tau}(b^{(k)},\theta^{(k)}).
\]
We rewrite $\bar V_{\varsigma, \tau}(b^{(k)},\theta^{(k)}) - V_{0}(b^{(k)},\theta^{(k)})$ as
\be\label{eq:V-bar-V-0}
\begin{array}{rl}
    & \bar V_{\varsigma, \tau}(b^{(k)},\theta^{(k)}) - V_{0}(b^{(k)},\theta^{(k)}) \\
     = & \bar V_{\varsigma, \tau}(b^{(k)},\theta^{(k)}) - V_{\varsigma, \tau}(b^{(k)},\theta^{(k)}) + V_{\varsigma, \tau}(b^{(k)},\theta^{(k)}) - \hat V_{\varsigma, \tau}(b^{(k)},\theta^{(k)}) + \hat V_{\varsigma, \tau}(b^{(k)},\theta^{(k)}) - V_{0}(b^{(k)},\theta^{(k)}) \\ 
     \le & \underbrace{\bar V_{\varsigma, \tau}(b^{(k)},\theta^{(k)}) - V_{\varsigma, \tau}(b^{(k)},\theta^{(k)}) }_{\mathrm{(I)} } + \underbrace{\hat V_{\varsigma, \tau}(b^{(k)},\theta^{(k)}) - V_{0}(b^{(k)},\theta^{(k)})}_{\mathrm{(II)}},
\end{array}
\ee
where the inequality follows from $V_{\varsigma, \tau}(b^{(k)},\theta^{(k)}) \le \hat V_{\varsigma, \tau}(b^{(k)},\theta^{(k)})$. 

We first upper bound the term $\mathrm{(I)}$. To this end, we construct a feasible TLDR for the problem defining $\bar V_{\varsigma, \tau}(b^{(k)},\theta^{(k)})$ so that the generated cost can bound $\bar V_{\varsigma, \tau}(b^{(k)},\theta^{(k)})$ from above. Under the first-stage decision $x_{\varsigma, \tau}^{(k)}$, let $g(x_{\varsigma, \tau}^{(k)}, \cdot)$ be attained at $\bar{\mathbf{y}}_{\varsigma, \tau}^{(k)}(\cdot)$, then we have $V_{\varsigma, \tau}(b^{(k)},\theta^{(k)}) = c^\top x_{\varsigma, \tau}^{(k)} + \varrho_{M_{\varsigma,\tau}} (-p^\top \bar{\mathbf{y}}^{(k)}_{\varsigma, \tau}(\tilde d)) $ by~\citet[Proposition~2.1]{ref:shapiro2017interchangeability}. As shown in the case of $M_0$, we can find $(\bar v_{\varsigma, \tau}^{(k)},\bar U_{\varsigma, \tau}^{(k)})$ such that $\bar v_{\varsigma, \tau}^{(k)} = \bar U_{\varsigma, \tau}^{(k)} d_h = \bar{\mathbf{y}}_{\varsigma, \tau}^{(k)}(d_h)$. The feasibility and suboptimality of the TLDR: $d\mapsto\min\{ \bar v_{\varsigma, \tau}^{(k)}, \bar U_{\varsigma, \tau}^{(k)} d \}$ yields
\[
\bar V_{\varsigma, \tau}(b^{(k)},\theta^{(k)}) \le c^\top x_{\varsigma, \tau}^{(k)} + \varrho_{M_{\varsigma, \tau}(\theta^{(k)})} ( -p^\top \min\{ \bar v_{\varsigma, \tau}^{(k)}, \bar U_{\varsigma, \tau}^{(k)} \tilde d \}),
\]
and then
\[
\bar V_{\varsigma, \tau}(b^{(k)},\theta^{(k)}) - V_{\varsigma, \tau} (b^{(k)},\theta^{(k)}) \le \varrho_{M_{\varsigma, \tau}(\theta^{(k)})} ( -p^\top \min\{ \bar v_{\varsigma, \tau}^{(k)}, \bar U_{\varsigma, \tau} ^{(k)} \tilde d \}) - \varrho_{M_{\varsigma, \tau}(\theta^{(k)})}(-p^\top \bar{\mathbf{y}}_{\varsigma, \tau}^{(k)}(\tilde d)),
\]
which follows from $\varrho_{M_{\varsigma, \tau}(\theta^{(k)})}(g(x_{\varsigma, \tau}^{(k)},\tilde d)) = \varrho_{M_{\varsigma, \tau}(\theta^{(k)})}(-p^\top \bar{\mathbf{y}}_{\varsigma, \tau}^{(k)}(\tilde d))$. 

As $d_l \le d_h$, one can verify that $-p^\top \bar{\mathbf{y}}_{\varsigma, \tau}^{(k)} (d_h) \le -p^\top \bar{\mathbf{y}}_{\varsigma, \tau}^{(k)}(d_l)$ and it holds that $\bar U_{\varsigma, \tau}^{(k)} d_l \le \bar U_{\varsigma, \tau}^{(k)} d_h$ because $\bar U_{\varsigma, \tau}^{(k)}$ is non-negative. By the monotonicity of $\varrho$, we have
\[
\varrho_{M_{\varsigma, \tau}(\theta^{(k)})}(-p^\top \bar{\mathbf{y}}_{\varsigma, \tau}^{(k)} (\tilde d)) \ge -p^\top \bar{\mathbf{y}}_{\varsigma, \tau}^{(k)} (d_h),
\]
and we also have
\[
\varrho_{M_{\varsigma, \tau}(\theta^{(k)})} ( -p^\top \min\{ \bar v_{\varsigma, \tau}^{(k)}, \bar U_{\varsigma, \tau}^{(k)} \tilde d \}) \le -p^\top \min\{ \bar v_{\varsigma, \tau}^{(k)}, \bar U_{\varsigma, \tau}^{(k)} d_l \} .
\]
Therefore, it holds that
\be \label{eq:term1}
\begin{aligned}
    \bar V_{\varsigma, \tau}(b^{(k)},\theta^{(k)}) - V_{\varsigma, \tau} (b^{(k)},\theta^{(k)}) & \le -p^\top \min\{ \bar v_{\varsigma, \tau}^{(k)}, \bar U_{\varsigma, \tau}^{(k)} d_l \}  +  p^\top \bar{\mathbf{y}}_{\varsigma, \tau}^{(k)}(d_h)\\
    & = p^\top \bar U_{\varsigma, \tau}^{(k)} (d_h - d_l) \\
    & = \left( \sqrt{\frac{1-\tau}{\tau}} + \sqrt{\frac{\tau}{1-\tau}}\right) k^{\frac{s}{2}} p^\top\bar U_{\varsigma, \tau}^{(k)} \varsigma,
\end{aligned}
\ee
where the first equality comes from $\bar U_{\varsigma, \tau}^{(k)} d_l \le \bar U_{\varsigma, \tau}^{(k)} d_h = \bar{\mathbf{y}}_{\varsigma, \tau}^{(k)}(d_h) = \bar v_{\varsigma, \tau}^{(k)}$.

Next, we find an upper bound for the term $\mathrm{(II)}$ via $(v_0^{(k)}, U_0^{(k)})$, at which $\bar V_0(b^{(k)},\theta^{(k)})$ is attained. The optimal TLDR under $x_0^{(k)}$ is $\mathbf{y}_0^{(k)}: d \mapsto \min\{ v_{0}^{(k)}, U_{0}^{(k)} d\}$. Observe that 
\[ 
\begin{aligned}
    C(x_{0}^{(k)}, \mathbf{y}_{0}^{(k)},
b^{(k)},\theta^{(k)} ) & = c^\top x_{0}^{(k)} + \sup_{\PP \in \mathcal A(\theta^{(k)})}~\varrho_{\PP}(- p^\top  \min\{v_{0}^{(k)}, U_{0}^{(k)} \tilde d\}) \\ 
   & \ge c^\top x_{0}^{(k)} + \varrho_{M_{\varsigma, \tau}(\theta^{(k)})}(- p^\top  \min\{v_{0}^{(k)}, U_{0}^{(k)} \tilde d\})  \ge \hat V_{\varsigma, \tau} (b^{(k)},\theta^{(k)}) ,
\end{aligned} 
\]
where the first inequality follows because $M_{\varsigma, \tau}(\theta^{(k)}) \in \mathcal A(\theta^{(k)})$ and the second inequality holds because $(x_{0}^{(k)},v_{0}^{(k)},U_{0}^{(k)})$ is feasible for the minimization problem defining $\hat V_{\varsigma, \tau}(b^{(k)},\theta^{(k)})$. Subsequently, by Proposition~\ref{prop:loss-upper-bound}, it holds that
\be\label{eq:term2}
\begin{aligned}
 \hat V_{\varsigma, \tau}(b^{(k)},\theta^{(k)}) -V_{0}(b^{(k)},\theta^{(k)})  & \le  C(x_{0}^{(k)}, \mathbf{y}_{0}^{(k)}, b^{(k)},\theta^{(k)} ) -    V_{0}(b^{(k)},\theta^{(k)}) \\
 & \le (\frac{1}{2}+\alpha ) k^{\frac{s}{2}} p^\top U_{0}^{(k)} \sigma.
\end{aligned}
\ee

Applying inequalities~\eqref{eq:term1} and~\eqref{eq:term2} to~\eqref{eq:V-bar-V-0} completes the proof.  
\end{proof}

\begin{proof}[Proof of Lemma~\ref{lemma:risk-pos-homo}.]
The base vectors are defined to be $b^{(1)} = b\in\R_+^K$ and $\theta^{(1)} = (\mu, \sigma) \in\R_+^{2\Nd}$.
Let $V_{0}(b^{(1)},\theta^{(1)})$ be attained at $(x_{0}^{(1)}, y_{0}^{(1)} ) \in \R_+^{\Nx}\times \R_+^{\Ny}$. It can be verified that $(k x_{0}^{(1)}, k y_{0}^{(1)} )$ is a feasible solution to the optimization problem defining $V_{0}(b^{(k)},\theta^{(k)})$. Indeed, because $(x_{0}^{(1)} , y_{0}^{(1)}  )$ is a feasible solution to the problem defining $V_{0}(b^{(1)},\theta^{(1)})$, we have
\[
x_{0}^{(1)} \in \mathcal X(b),~ A y_{0}^{(1)}  \le x_{0}^{(1)} ,~ H y_{0}^{(1)} \le \mu,
\]
and then $(k x_{0}^{(1)} , k y_{0}^{(1)}  )$ satisfies
\[
k x_{0}^{(1)}  \in \mathcal X(b^{(k)}),~ A (k y_{0}^{(1)} ) \le k x_{0}^{(1)} , ~ H ( k y_{0}^{(1)} ) \le k \mu.
\]
Hence, $(k x_{0}^{(1)} , k y_{0}^{(1)} )$ is feasible for the optimization problem determining $V_{0}(b^{(k)},\theta^{(k)})$, and it follows from the suboptimality that
\[
V_{0}(b^{(k)},\theta^{(k)}) \le kc^\top x_{0}^{(1)}  -  k p^\top y_{0}^{(1)}  = k \big( c^\top x_{0}^{(1)} -  p^\top y_{0}^{(1)} \big) = k V_{0}(b^{(1)},\theta^{(1)}) .
\]
Applying an analogous argument, suppose that $(x_{0}^{(k)} , y_{0}^{(k)} )$ is a minimizer to the problem defining $V_{0}(b^{(k)},\theta^{(k)})$, then $(  x_{0}^{(k)} / k,  y_{0}^{(k)} / k )$ is a feasible solution to the problem defining $V_{0}(b^{(1)},\theta^{(1)})$. It follows that
$ V_{0}(b^{(k)},\theta^{(k)}) \ge k V_{0}(b^{(1)},\theta^{(1)})$, which completes the proof together with $V_{0}(b^{(k)},\theta^{(k)}) \le k V_{0}(b^{(1)},\theta^{(1)})$.  
\end{proof}

\begin{proof}[Proof of Proposition~\ref{prop:wc-dual}.]
For any given first-stage decision $x\in\R_+^{\Nx}$ and underlying distribution $\PP$ of $\tilde d$, the second-stage approximation problem under TLDR becomes
\begin{equation}\label{eq:tldr-approx}
  (v\opt,U\opt) \in \left\{
    \begin{array}{cl}
    \arg\min & \varrho_{\PP}(- p^\top \min\{ v ,U \tilde d\} )\\
        \st & v \in \R_+^{\Ny},~ U\in\R_+^{\Ny\times \Nd} \\
            &  A v \leq x,~ H U\leq I,  
    \end{array} 
    \right.
\end{equation} 
where the constraints are sufficient and necessary to ensure the feasibility of induced TLDR: $d \mapsto \min\{ v ,U d\}$ as stated in Proposition~\ref{prop:HU}.
We show that $U=I$ is an optimal choice for $U$.
Let $(v\opt ,U\opt)$ be the optimal solution of problem~\eqref{eq:V-bar}, we have $-p^\top \min \{ v\opt,U\opt  d \} \ge -p^\top \min \{ v\opt, d \}$ as $U\opt \le I$ and $\{ \varrho_{\PP}\}$ are monotonic risk measures. Hence, $(v\opt, I)$ is also optimal to problem~\eqref{eq:tldr-approx} because it satisfies all the constraints and achieves the optimal cost. This enables us to focus on the optimal truncated linear decision rule in the form of $\mathbf{y}:d \mapsto \min\{v,d\}$ with $v$ to be optimized.

Under such circumstances, we can reduce the worst-case expected loss in the computation of WCR to a second-order cone program as follows.
By duality, we have
\[
\begin{aligned}
\sup\limits_{\PP\in\mathcal A(\theta)}   \E_{\PP}\left[-p^{\top} \min  \{ v , \tilde d \}\right]
=& \left\{\begin{array}{cl}
   \sup\limits_{\PP\in\mathcal M}   & \E_{\PP} \left[-p^{\top} \min \{ v , \tilde d \} \right] \\
    \st & \E_{\PP}[\tilde d]=\mu \\ 
        & \E_{\PP}\left[\tilde d_i^2\right] \leq \mu_i^2+\sigma_i^2~\forall i\in[N] \\
        & \E_{\PP}[\mathbf{1}_{\R_+^{N}}(\tilde d)]=1 
\end{array}\right. \\
=& \left\{
\begin{array}{cl}
   \min  & \lambda^{\top} \mu+\sum\limits_{i\in[N]} s_i\left(\mu_i^2+\sigma_i^2\right)+ \sum\limits_{i\in[N]} r_i \\
   \st   & \lambda\in\R^N,~s\in\R_+^N,~r\in\R^N \\
   &  \lambda_i d_i+\sum s_i d_i^2+r_i \ge -p_i \min  \{ v_i , d_i \} \quad \forall d_i \in\R_+, i\in[N] 
\end{array} \right. \\
= & \left\{
\begin{array}{cl}
   \min  & \lambda^{\top} \mu+\sum\limits_{i\in[N]} s_i\left(\mu_i^2+\sigma_i^2\right)+ \sum\limits_{i\in[N]} r_i  \\
    \st   & \lambda\in\R^N,~s\in\R_+^N,~r\in\R^N   \\
    & \lambda_i d_i+\sum s_i d_i^2+r_i \ge -p_i v_i \quad \forall d_i \in\R_+ ,i\in[N] \\
        & \lambda_i d_i+\sum s_i d_i^2+r_i \ge -p_i d_i  \quad \forall d_i \in\R_+,i\in[N] ,
\end{array}\right.
\end{aligned}.
\]
where $d_i$ is the $i$-th element of $d$. The second equality follows from the strong duality~\cite[Proposition 3.4]{ref:shapiro2001duality}. The third equality holds because $-p_i \min  \{ v_i , d_i \} $ is no less than both $ -p_i v_i$ and $- p_i d_i$. Next, we transfer the semi-infinite constraints to finite constraints. Take the constraint $\lambda_i d_i + s_i d_i^2+r_i \ge - p_i v_i$ for all $d_i \in\R_+$ for example, by Farkas lemma~\cite[Theorem 2.1]{ref:polik2007survey}, it is equivalent to the statement that
\begin{equation}\label{eq:infinite-inequality}
\exists~q_{1,i} \in \R_+ \quad \st \quad  \lambda_i d_i + s_i d_i^2+r_i +p_i v_i  - q_{1,i} d_i \ge 0~ \forall d_i \in \R.
\end{equation}
Note that $s\in \R_+^N$, and when $s_i > 0$, the left-hand side of the inequality in~\eqref{eq:infinite-inequality} is a quadratic function with a positive leading coefficient. The quadratic function attains its minimum at $d_i = (q_{1,i} - \lambda _i ) / 2s_i$ with minimal value $ r_i + p_i v_i -  (q_{1,i} - \lambda_i )^2 / 4s_i $. Hence, we have
\[
\begin{aligned}
    & \lambda_i d_i + s_i d_i^2+r_i +p_i v_i  - q_{1,i} d_i \ge 0~ \forall d_i \in \R \\
    \Leftrightarrow ~& r_i + p_i v_i -  \frac{ (q_{1,i} - \lambda_i )^2 }  { 4s_i } \ge 0 \\
    \Leftrightarrow ~& (q_{1,i} - \lambda_i )^2 \le 4s_i (r_i + p_i v_i ) = ( s_i + r_i + p_i v_i)^2 - (s_i - r_i - p_i v_i)^2 \\
    \Leftrightarrow ~& \left\lVert \begin{array}{c}
    q_{1,i} - \lambda_i \\
    s_i - r_i - p_i v_i
\end{array} \right\rVert \le s_i + r_i + p_i v_i.
\end{aligned}
\]
That means the statement~\eqref{eq:infinite-inequality} is equivalent to the constraints 
\begin{equation}\label{eq:soc-constraint}
   q_{1,i} \ge 0 \quad \text{and} \quad (s_i + r_i + p_i v_i ,  q_{1,i} - \lambda_i , s_i - r_i - p_i v_i ) \in \mathcal Q^3
\end{equation}
when $s_i > 0$. As for the case when $s_i = 0$, both the statement~\eqref{eq:infinite-inequality} and the constraint~\eqref{eq:soc-constraint} require $q_{1,i}  - \lambda_i = 0 $ and $r_i + p_i v_i \ge 0$, which completes the equivalence between the statement~\eqref{eq:infinite-inequality} and the constraint~\eqref{eq:soc-constraint} for any $s_i \ge 0$.

Following the same step, we can also transfer the semi-infinite constraint 
\[
\lambda_i d_i + s_i d_i^2 + r_i \ge - p_i d_i~\forall d_i \in\R_+
\]
to the constraints
\[
q_{2,i} \ge 0 \quad \text{and} \quad (s_i + r_i , ~ q_{2,i} - \lambda_i - p_i ,~ s_i - r_i ) \in \mathcal Q^3.
\] 
Consequentially, the worst-case expected loss $\sup_{\PP\in\mathcal A(\theta)}   \E_{\PP}[-p^{\top} \min  \{ v , \tilde d \}]$ is equivalent to the following second-order cone programming problem
\[
         \begin{array}{cll}
          \min  & \lambda^{\top} \mu+\sum\limits_{i\in[N]} s_i\left(\mu_i^2+\sigma_i^2\right)+ \sum\limits_{i\in[N]} r_i  \\
            \st & \lambda \in \R^N,~s \in \R_+^N,~r\in\R^N,~q_1 \in \R_+^N,~q_2\in\R_+^{N} \\
                & (s_i + r_i + p_i v_i , ~ q_{1,i} - \lambda_i ,~ s_i - r_i - p_i v_i ) \in \mathcal Q^3~\forall i \in [N] \\
                & (s_i + r_i , ~ q_{2,i} - \lambda_i - p_i ,~ s_i - r_i ) \in \mathcal Q^3 ~\forall i \in [N] .
        \end{array}
\]
This completes the proof.  
\end{proof}

\subsection{Proofs for Wasserstein Ambiguity Set}\label{appendix:wass}

We follow a similar strategy as in Section~\ref{sec:moment}: we aim to establish a lower bound for $ \mathrm{OPT}(b^{(k)},\theta^{(k)})$ and an upper bound for $\mathrm{Obj}(x_{\varsigma,\tau}^{(k)}, b^{(k)},\theta^{(k)})$, respectively. However, the Wasserstein ambiguity sets introduce new challenges necessitating modification to the analysis. First, we cannot directly infer that $ V_{0}(b^{(k)},\theta^{(k)}) \le \mathrm{OPT}(b^{(k)},\theta^{(k)})$ as the expectation mechanism $M_{0}(\theta^{(k)}) \Let \delta_{\hat\mu^{(k)}}$ may not belong to $\mathcal A(\theta^{(k)})$. Second, while the distributions within $\mathcal{A}(\theta^{(k)})$ do not have a fixed mean, their means are confined to a ball centered at $k\hat\mu$ with radius $\varepsilon^{(k)} = k^{\frac{s}{2}} \mu$ as demonstrated in Proposition~\ref{prop:outer-approx}.

\subsubsection{A Lower Bound.}

To address the first issue, we rely on the condition that $\varrho_{\PP}(\cdot)\ge\E_{\PP}(\cdot)$ for all $ \PP\in\mathcal{M}_2$. Under this condition, $V_{0}(b^{(k)},\theta^{(k)})$ remains to be a lower bound for $\mathrm{OPT}(b^{(k)},\theta^{(k)})$ as stated in Proposition~\ref{prop:lower-wass}.
\begin{proposition}[Lower bound - Wasserstein DRO]\label{prop:lower-wass}
   Suppose that Assumption~\ref{a:model-param}\ref{a:model-param-1} and Assumption~\ref{a:ambiguity-wass} hold, that $\varrho_{\PP}(\cdot)\ge\E_{\PP}(\cdot)$, and that the scaling scheme~\eqref{eq:wass-scaling} is in force. We have $
     V_{0}(b^{(k)},\theta^{(k)}) \le \mathrm{OPT}(b^{(k)},\theta^{(k)})$.
\end{proposition}

\begin{proof}[Proof of Proposition~\ref{prop:lower-wass}.]
 As $\PP^{(k)}\in \mathcal A(\theta^{(k)})$, we have
    \[
    \begin{aligned}
       \mathrm{OPT}(b^{(k)},\theta^{(k)}) & = 
          \min\limits_{x\in \mathcal X (b^{(k)})} \left\{ c^\top x+ \sup\limits_{\PP \in \mathcal A(\theta^{(k)})} \varrho_{\PP}( g(x,\tilde d) ) \right\}\\  
          & \ge 
          \min\limits_{x\in \mathcal X (b^{(k)})} \left\{ c^\top x+ \varrho_{\PP^{(k)}}( g(x,\tilde d) ) \right\} \\
        & \ge 
          \min\limits_{x\in \mathcal X (b^{(k)})} \left\{c^\top x+ \E_{\PP^{(k)}}( g(x,\tilde d) ) \right\},
    \end{aligned}
    \]
where the first inequality follows from $\PP^{(k)}\in \mathcal{A}(\theta^{(k)} )$ and the second inequality follows from $\varrho_{\PP}(\cdot) \ge \E_{\PP}(\cdot)$ for all $\PP$. Observe that $g(x,(1-t) d_{0} + t d_1) \le (1-t)g(x, d_{0}) + t g(x, d_1)$ for all $d_{0},d_1\ge 0$ and $t\in[0,1]$, which means $g(x,\cdot)$ is a convex function. By Jensen's inequality, we have $ \E_{\PP^{(k)}} [ g(x,\tilde d) ] \ge g(x,\E_{\PP^{(k)}}[\tilde d]) = g(x,k\hat \mu ) $ and then
\[
       \mathrm{OPT}(b^{(k)},\theta^{(k)}) \ge \min\limits_{x \in \mathcal X (b^{(k)})}  c^\top x + g(x,k\hat\mu )  = V_{0}(b^{(k)},\theta^{(k)}),
\]
which completes the proof. 
\end{proof} 

\subsubsection{An Upper Bound.}

Proposition~\ref{prop:outer-approx} shows that $\mathcal{A}(\theta^{(k)})$ can be outer-approximated by an ambiguity set with bounded mean and variance. We now prove this result.

\begin{proof}[Proof of Proposition~\ref{prop:outer-approx}.]
By~\citet[Theorem~2]{ref:nguyen2021mean}, we have 
\[
\mathcal A(\theta^{(k)}) \subseteq \mathcal G_{\varepsilon^{(k)}}(\hat\mu^{(k)}, \hat\Sigma^{(k)} ) \triangleq  \left\{\PP : 
    \begin{array}{l}
       \E_{\PP}[\tilde d] = \mu,~
       \E_{\PP}[(\tilde d-\E_{\PP}[\tilde d])(\tilde d-\E_{\PP}[\tilde d])^\top ] = \Sigma  \\ [1ex]
       \sqrt{\lVert \mu - \hat\mu^{(k)} \rVert^2 + \mathrm{Tr}[\Sigma + \hat\Sigma^{(k)} - 2(\Sigma^{\frac{1}{2}} \hat\Sigma^{(k)} \Sigma^{\frac{1}{2}} )^{\frac{1}{2}}]} \le \varepsilon^{(k)}
    \end{array}\right\}.
\]
Next, we show that $\mathcal G_{\varepsilon^{(k)}}(\hat\mu^{(k)},\hat\Sigma^{(k)}) \subseteq \mathcal R(k\hat\mu,k^{\frac{s}{2}}\hat\sigma,k^{\frac{s}{2}}\varepsilon )$. 
Because $\mathrm{Tr}[\Sigma + \hat\Sigma^{(k)} - 2(\Sigma^{\frac{1}{2}} \hat\Sigma^{(k)} \Sigma^{\frac{1}{2}} )^{\frac{1}{2}}]$ is non-negative, any $(\mu, \Sigma)$ satisfying
\[
\sqrt{\lVert \mu - \hat\mu^{(k)} \rVert^2 + \mathrm{Tr}[\Sigma + \hat\Sigma^{(k)} - 2(\Sigma^{\frac{1}{2}} \hat\Sigma^{(k)} \Sigma^{\frac{1}{2}} )^{\frac{1}{2}}] } \le \varepsilon ^{(k)}
\]
also satisfy jointly
\[
          \lVert \mu - \hat\mu^{(k)} \rVert \le \varepsilon^{(k)} \quad \text{and} \quad 
       \sqrt{ \mathrm{Tr}[\Sigma + \hat\Sigma^{(k)} - 2(\Sigma^{\frac{1}{2}} \hat\Sigma^{(k)} \Sigma^{\frac{1}{2}} )^{\frac{1}{2}}] } \le \varepsilon^{(k)}.
\]
This implies that
\[
    \mathcal G_{\varepsilon^{(k)}}(\hat\mu^{(k)},\hat\Sigma^{(k)}) \subseteq \left\{\PP : 
    \begin{array}{l}
       \E_{\PP}[\tilde d] = \mu,~
       \E_{\PP}[(\tilde d-\E_{\PP}[\tilde d])(\tilde d-\E_{\PP}[\tilde d])^\top ] = \Sigma  \\ 
       \lVert \mu - \hat\mu^{(k)} \rVert \le \varepsilon^{(k)} \\
       \sqrt{ \mathrm{Tr}[\Sigma + \hat\Sigma^{(k)} - 2(\Sigma^{\frac{1}{2}} \hat\Sigma^{(k)} \Sigma^{\frac{1}{2}} )^{\frac{1}{2}}] } \le \varepsilon^{(k)}
    \end{array}\right\} .
\]
We further find a lower bound for $\mathrm{Tr}[\Sigma + \hat\Sigma^{(k)} - 2(\Sigma^{\frac{1}{2}} \hat\Sigma^{(k)} \Sigma^{\frac{1}{2}} )^{\frac{1}{2}}] $ to enlarge the set.
By~\citet[Proposition~2]{ref:malago2018wasserstein}, we have
\begin{align*}
\mathrm{Tr}[\Sigma + \hat\Sigma^{(k)} - 2(\Sigma^{\frac{1}{2}} \hat\Sigma^{(k)} \Sigma^{\frac{1}{2}} )^{\frac{1}{2}}] &= \left\{
\begin{array}{cl}
    \min & \mathrm{Tr}[\Sigma + \hat\Sigma^{(k)} - 2 C ] \\
    \st & C\in\R^{\Nd\times\Nd} \\
        & \begin{bmatrix}
             \Sigma & C  \\
             C^\top & \hat\Sigma^{(k)} 
         \end{bmatrix} \succeq 0,
\end{array} \right. \\
&=
\left\{
\begin{array}{cl}
    \max & \mathrm{Tr}[(I-R_1)^\top\Sigma + (I-R_2)^\top \hat\Sigma^{(k)} ] \\
    \st & R_1,R_2 \in\R^{\Nd\times\Nd} \\
        & \begin{bmatrix}
             R_1 & -I  \\
             -I & R_2 
         \end{bmatrix} \succeq 0,
\end{array} \right. 
\end{align*}
where the second equality comes from strong duality. 
Observe that $(R_1, R_2) = (\frac{1}{2}I,2I)$ is a feasible solution to the maximization problem on the right-hand side; thus, it holds that 
\[
\mathrm{Tr}[\Sigma + \hat\Sigma^{(k)} - 2(\Sigma^{\frac{1}{2}} \hat\Sigma^{(k)} \Sigma^{\frac{1}{2}} )^{\frac{1}{2}}] \ge \frac{1}{2}\mathrm{Tr}(\Sigma) - \mathrm{Tr}( \hat\Sigma^{(k)}).
\] 
Hence, we have
\[
\begin{aligned}
  \sqrt{\mathrm{Tr}[\Sigma + \hat\Sigma^{(k)} - 2(\Sigma^{\frac{1}{2}} \hat\Sigma^{(k)} \Sigma^{\frac{1}{2}} )^{\frac{1}{2}}]} \le  \varepsilon^{(k)} & \Rightarrow \sqrt{ \frac{1}{2}\mathrm{Tr}(\Sigma) - \mathrm{Tr}(\hat\Sigma^{(k)}) }\le \varepsilon^{(k)}  \\
  &\Leftrightarrow \mathrm{Tr}(\Sigma) \le 2\mathrm{Tr}( \hat\Sigma^{(k)}) + 2k^s\varepsilon^2,
\end{aligned}
\]
where we have plugged in $\varepsilon^{(k)} = k^{\frac{s}{2}}\varepsilon$.
As all the diagonal elements of $\Sigma \succeq 0$ are non-negative, it holds that
\[
\mathrm{Tr}(\Sigma) \le 2\mathrm{Tr}( \hat\Sigma^{(k)}) + 2k^s\varepsilon^2 \Rightarrow \Sigma_{ii} \le 2\mathrm{Tr}( \hat\Sigma^{(k)}) + 2 k^s \varepsilon^2 \le  k^s(2\mathrm{Tr}( \hat\Sigma) + 2\varepsilon^2)~\forall i\in[\Nd],
\]
where $\Sigma_{ii} = \E_{\PP}[(\tilde d_i - \E_{\PP}[\tilde d_i] )^2]$ and the last second inequality in the right follows from $\hat\Sigma^{(k)} \preceq k^s\hat\Sigma$. 
Therefore, $\mathcal G_{\varepsilon^{(k)}}(\hat\mu^{(k)},\hat\Sigma^{(k)})$ can be outer approximated as
\be \label{eq:outer-G}
\mathcal G_{\varepsilon^{(k)}}(\hat\mu^{(k)},\hat\Sigma^{(k)}) \subseteq \left\{\PP : 
    \begin{array}{l}
       \E_{\PP}[\tilde d] = \mu    \\
       \E_{\PP}[(\tilde d-\mu)(\tilde d - \mu)^\top ] = \Sigma  \\ 
       \lVert \mu - \hat\mu^{(k)} \rVert \le \varepsilon^{(k)} \\
      \Sigma_{ii} \le k^{s}(2\mathrm{Tr}(\hat\Sigma) + 2\varepsilon^2) 
    \end{array}\right\} 
= \left\{\PP : 
    \begin{array}{l}
       \E_{\PP}[\tilde d] = \mu    \\
       \E_{\PP}[(\tilde d-\mu)(\tilde d - \mu)^\top ] = \Sigma  \\ 
       \lVert \mu - k\hat\mu \rVert \le k^{\frac{s}{2}}\varepsilon \\
      \Sigma_{ii} \le k^{s}\hat\sigma_i^2
    \end{array}\right\},
\ee
where the equality follows by substituting $\hat\mu^{(k)} = k\hat\mu$ and $\hat\sigma_i = \sqrt{2\mathrm{Tr}(\hat\Sigma) + 2\varepsilon^2}$.
The right-hand-side set in~\eqref{eq:outer-G} coincides with $ \mathcal R(k\hat\mu,k^{\frac{s}{2}}\hat\sigma,k^{\frac{s}{2}}\varepsilon)$ because it only restricts the diagonal elements of its covariance matrix, i.e., the marginal variances of distributions. This completes the proof.  
\end{proof}

Proposition~\ref{prop:outer-approx} establishes a connection between the data-driven Wasserstein ambiguity set and the moment-based ambiguity set, which allows us to control the growth rate of the mean and marginal variance for any distribution in $\mathcal{A}(\theta^{(k)})$. Next, we find an upper bound for $\mathrm{Obj} (x_{\varsigma,\tau}^{(k)}, b^{(k)},\theta^{(k)} )$
by approximating the optimal second-stage decision via truncated linear decision rules (TLDR). Based on constraints on $v$ and $U$ dictated by Proposition~\ref{prop:HU}, given the first-stage decision $x_{\varsigma,\tau}^{(k)}$, we find a feasible truncated linear decision rule by solving
\begin{subequations}
 \begin{equation}\label{eq:V-bar-wass}
 \bar V_{\varsigma, \tau} (b^{(k)},\theta^{(k)}) \Let \left\{
    \begin{array}{cl}
    \min & c^\top x_{\varsigma, \tau}^{(k)} + \varrho_{M_{\varsigma, \tau}(\theta^{(k)})} (- p^\top \min\{ v ,U \tilde d\} )\\ [1ex]
        \st & v \in \R_+^{\Ny},~ U\in\R_+^{\Ny\times \Nd} \\
            &  A v \leq x_{\varsigma, \tau}^{(k)},~ H U\leq I,  
    \end{array} 
    \right. 
\end{equation}
where the constant $c^\top x_{\varsigma, \tau}^{(k)}$ in the objective does not impact its optimal solution set. Let $\bar V_{\varsigma, \tau}(b^{(k)},\theta^{(k)})$ be attained at $(v_{\varsigma, \tau}^{(k)},U_{\varsigma, \tau}^{(k)})$, we define the TLDR
\be\label{eq:policy-wass}
\mathbf{y}_{\varsigma, \tau}^{(k)}(d) \Let \min\{v_{\varsigma, \tau}^{(k)},U_{\varsigma, \tau}^{(k)} d \}.
\ee
\end{subequations}
The cost generated by $(x_{\varsigma,\tau}^{(k)},\mathbf{y}_{\varsigma,\tau}^{(k)})$ is denoted by
\[
C(x_{\varsigma,\tau}^{(k)},\mathbf{y}_{\varsigma,\tau}^{(k)}, b^{(k)}, \theta^{(k)}) = c^\top x_{\varsigma, \tau}^{(k)} + \sup\limits_{\PP\in\mathcal{A}(\theta^{(k)})} \varrho_{\PP} (- p^\top \mathbf{y}_{\varsigma,\tau}^{(k)}(\tilde d)).
\]
By the suboptimality of $\mathbf{y}_{\varsigma,\tau}^{(k)}$, we have $-p^\top\mathbf{y}_{\varsigma,\tau}^{(k)}( d) \ge g(x_{\varsigma,\tau}^{(k)}, d)$ for all $d\ge 0$. As $\varrho_{\PP}(\cdot)$ is monotonic, it follows that $C(x_{\varsigma,\tau}^{(k)},\mathbf{y}_{\varsigma,\tau}^{(k)}, b^{(k)}, \theta^{(k)}) \ge \mathrm{Obj}(x_{\varsigma,\tau}^{(k)}, b^{(k)}, \theta^{(k)})$. Therefore, the suboptimality gap $\mathrm{Obj}(x_{\varsigma,\tau}^{(k)}, b^{(k)}, \theta^{(k)}) - \mathrm{OPT}(b^{(k)},\theta^{(k)})$ is upper bounded by $C(x_{\varsigma,\tau}^{(k)},\mathbf{y}_{\varsigma,\tau}^{(k)}, b^{(k)}, \theta^{(k)}) - V_0(b^{(k)},\theta^{(k)})$. We further separate the latter term into two parts: 
\be\label{eq:separation-wass}
\begin{aligned}
     & C(x_{\varsigma, \tau}^{(k)},\mathbf{y}_{\varsigma, \tau}^{(k)},b^{(k)},\theta^{(k)}) - V_{0}(b^{(k)},\theta^{(k)})  \\
    = & \underbrace{C(x_{\varsigma, \tau}^{(k)},\mathbf{y}_{\varsigma, \tau}^{(k)},b^{(k)},\theta^{(k)}) - \bar V_{\varsigma, \tau}(b^{(k)},\theta^{(k)})}_{\mathrm{(I)}} + \underbrace{\bar V_{\varsigma, \tau}(b^{(k)},\theta^{(k)}) -  V_{0}(b^{(k)},\theta^{(k)})}_{\mathrm{(II)}}.
\end{aligned}
\ee
We first bound term $\mathrm{(I)}$ from above, for which the following lemma is crucial.

\begin{lemma}[Risk upper bound - Wasserstein DRO] \label{lemma:risk-upper-wass}
    Suppose that Assumptions~\ref{a:risk-measure} and~\ref{a:ambiguity-wass} hold, and that the scaling scheme~\eqref{eq:wass-scaling} is in force. For any $p \in \R_+^{\Nd}, v \in \R^{\Ny}$ and $U\in\R_+^{\Ny\times\Nd}$, it holds that
    \[
    \sup\limits_{\PP \in  \mathcal A(\theta^{(k)})}  \varrho_{\PP}(p^\top \max\{0, v - U \tilde d\})\leq (\frac{1}{2}+\alpha) k^{\frac{s}{2}}p^\top  U \hat\sigma +  \sup\limits_{\mu:\lVert\mu-k\hat\mu \rVert\le k^{s/2}\varepsilon} \sum\limits_{i\in[\Nd]} p_i (v - U \mu)_i \mathbb{I}_{ ( v -  U \mu)_i > 0},
    \]
    where $\hat\sigma_i \Let \sqrt{2\mathrm{Tr}(\hat\Sigma) + 2\varepsilon^2} $ and $U_i$ is the $i$-th row vector of $U$.
\end{lemma}

\begin{proof}[Proof of Lemma~\ref{lemma:risk-upper-wass}.]
In Proposition~\ref{prop:outer-approx}, it has been established that $\mathcal A (\theta^{(k)}) \subseteq \mathcal R(k\hat\mu, k^{\frac{s}{2}}\hat\sigma, k^{\frac{s}{2}}\varepsilon)$, which implies
\[
\sup\limits_{\PP \in \mathcal A(\theta^{(k)}) } \varrho_{\PP}(p^\top \max\{0, v - U \tilde d\}) \le  \sup\limits_{\PP \in \mathcal R(k\hat\mu, k^{s/2}\hat\sigma, k^{s/2}\varepsilon)} \varrho_{\PP}(p^\top \max\{0, v - U \tilde d\}).
\]
Observe that
\[
\mathcal R(k\hat\mu, k^{\frac{s}{2}}\hat\sigma, k^{\frac{s}{2}}\varepsilon) = \bigcup\limits_{\mu:\lVert \mu- k\hat\mu \rVert\le k^{s/2}\varepsilon} \mathcal A_M\left((\mu,k^{s/2}\hat\sigma)\right),
\]
where $\mathcal A_M(\cdot)$ is the moment-based ambiguity set defined in~\eqref{eq:ambiguity}. Subsequently, it holds that
\be\label{eq:to-be-used-2}
\begin{aligned}
    &\sup\limits_{\PP \in \mathcal R(k\hat\mu, k^{s/2}\hat\sigma, k^{s/2}\varepsilon)}  \varrho_{\PP}(p^\top \max\{0, v - U \tilde d\}) \\
    = & \sup\limits_{\mu:\lVert\mu- k\hat\mu \rVert\le k^{s/2}\varepsilon} \sup\limits_{\PP \in \mathcal A_M\left((\mu,k^{s/2}\hat\sigma)\right)} \varrho_{\PP}(p^\top \max\{0, v - U \tilde d\}) \\
    \le & \sup\limits_{\mu:\lVert\mu-k\hat\mu \rVert\le k^{s/2}\varepsilon} (\frac{1}{2}+\alpha) k^{\frac{s}{2}}p^\top  U\hat\sigma + \sum\limits_{i\in[\Nd]} p_i (v - U \mu)_i \mathbb{I}_{( v - U \mu)_i > 0} \\
    = & (\frac{1}{2}+\alpha) k^{\frac{s}{2}}p^\top  U\hat\sigma + \sup\limits_{\mu:\lVert\mu-k\hat\mu \rVert\le k^{s/2}\varepsilon} \sum\limits_{i\in[\Nd]} p_i (v - U \mu)_i \mathbb{I}_{( v - U \mu)_i > 0},
\end{aligned}
\ee
where the inequality follows from Lemma~\ref{lemma:risk-upper}. This completes the proof. 
\end{proof}

Before we proceed to establish upper bounds for two terms in~\eqref{eq:separation-wass}, we first define 
\begin{equation} \label{eq:L}
L \Let \max\left \{ \sum\limits_{i\in[\Nd]} p_i \lVert U_{i}\rVert : ~ HU\le I,~ U\in \R_+^{\Ny \times \Nd} \right \},
\end{equation}
which depends on the vector $p\in\R_+^{\Ny}$ and the matrix $H$. Under Assumption~\ref{a:model-param}\ref{a:model-param-1}, the matrix $H\in \R_+^{\Nd \times \Ny} $ has no zero columns, which confines the feasible region of \eqref{eq:L} to be compact, and thus, $L$ exists and is finite.

\begin{proposition}[Loss upper bound (I) - Wasserstein DRO] \label{prop:loss-upper-bound-wass}
    Suppose that Assumptions~\ref{a:risk-measure} and~\ref{a:ambiguity-wass} hold, and that the scaling scheme~\eqref{eq:wass-scaling} is in force. For the mechanism $M_{\varsigma, \tau}$ in~\eqref{eq:Mall-wass}, the following hold:
    \begin{enumerate}\renewcommand{\labelenumi}{(\roman{enumi})}
    \item Under the mechanism $M_0$,
    \[
C(x_{0}^{(k)},\mathbf{y}_{0}^{(k)},b^{(k)},\theta^{(k)}) - \bar V_{0}(b^{(k)}, \theta^{(k)}) \le (\frac{1}{2}+\alpha )k^{\frac{s}{2}} L \lVert \hat\sigma \rVert+ k^{\frac{s}{2}}L\varepsilon.
    \]
    \item Under the mechanism $M_{\varsigma, \tau}$ with $(\varsigma,\tau) \in (0,\hat\sigma] \times (0,\hat\tau_{\max}]$,
    \[
    C(x_{\varsigma, \tau}^{(k)},\mathbf{y}_{\varsigma, \tau}^{(k)},b^{(k)},\theta^{(k)}) - \bar V_{\varsigma, \tau}(b^{(k)},\theta^{(k)}) \le  (\frac{1}{2}+\alpha)k^{\frac{s}{2}} L \lVert \hat\sigma \rVert + k^{\frac{s}{2}}L\varepsilon + \sqrt{\frac{1-\tau}{\tau}}k^{\frac{s}{2}} L\lVert \varsigma \rVert.
    \]
\end{enumerate}
\end{proposition}
\begin{proof}[Proof of Proposition~\ref{prop:loss-upper-bound-wass}.]
In the proof of Proposition~\ref{prop:loss-upper-bound}, we have shown in~\eqref{eq:to-be-used-1} that
\[ 
C( x_{\varsigma, \tau}^{(k)},\mathbf{y}_{\varsigma, \tau}^{(k)},b^{(k)},\theta^{(k)} )
\le \bar V_{\varsigma, \tau}(b^{(k)},\theta^{(k)}) +  \sup_{\PP \in \mathcal A(\theta^{(k)})}~\varrho_{\PP}( p^\top \max\{ 0, v_{\varsigma, \tau}^{(k)} - U_{\varsigma, \tau}^{(k)} \tilde d \}).
\]
Applying Lemma~\ref{lemma:risk-upper-wass} to the last term yields
\be\label{eq:C-wass}
\begin{aligned}
& C( x_{\varsigma, \tau}^{(k)},\mathbf{y}_{\varsigma, \tau}^{(k)},b^{(k)},\theta^{(k)} ) - \bar V_{\varsigma, \tau}(b^{(k)},\theta^{(k)}) \\
\le & (\frac{1}{2}+\alpha) k^{\frac{s}{2}} p^\top U_{\varsigma, \tau}^{(k)} \hat\sigma + \sup\limits_{ \mu:\lVert\mu-k\hat\mu \rVert\le k^{s/2}\varepsilon } \sum\limits_{i\in[\Nd]} p_i (v_{\varsigma, \tau}^{(k)} - U_{\varsigma, \tau}^{(k)} \mu )_i \mathbb{I}_{( v_{\varsigma, \tau}^{(k)} - U_{\varsigma, \tau}^{(k)} \mu)_i > 0} \\
\le & (\frac{1}{2}+\alpha) k^{\frac{s}{2}} L\lVert \hat\sigma \rVert + \sup\limits_{ \mu:\lVert\mu-k\hat\mu \rVert\le k^{s/2}\varepsilon } \sum\limits_{i\in[\Nd]} p_i (v_{\varsigma, \tau}^{(k)} - U_{\varsigma, \tau}^{(k)} \mu )_i \mathbb{I}_{( v_{\varsigma, \tau}^{(k)} - U_{\varsigma, \tau}^{(k)} \mu)_i > 0},
\end{aligned}
\ee
The second inequality in~\eqref{eq:C-wass} follows because 
\[
p^\top U_{\varsigma, \tau}^{(k)} \hat\sigma = \sum\limits_{i\in[\Nd]} p_i U_{\varsigma, \tau,i}^{(k)} \hat\sigma \le \sum\limits_{i\in[\Nd]} p_i \lVert U_{\varsigma, \tau,i}^{(k)} \rVert \lVert \hat\sigma \rVert \le L\lVert\hat\sigma\rVert,
\]
where $U_{\varsigma, \tau, i}^{(k)} $ is the $i$-th row vector of the matrix $U_{\varsigma, \tau}^{(k)}$, and the inequalities follow from Cauchy-Schwarz inequality and the definition of $L$. 
Observe the problem defining $\bar V_{\varsigma, \tau}(b^{(k)},\theta^{(k)})$ in~\eqref{eq:V-bar-wass}: its optimal solution $(v_{\varsigma, \tau}^{(k)}, U_{\varsigma, \tau}^{(k)} )$ must satisfy $v_{\varsigma, \tau}^{(k)} \le U_{\varsigma, \tau}^{(k)} d_h$, otherwise $(U_{\varsigma, \tau}^{(k)} d_h, U_{\varsigma, \tau}^{(k)} )$ is a feasible solution with smaller objective value and violates the optimality of $(v_{\varsigma, \tau}^{(k)}, U_{\varsigma, \tau}^{(k)} )$.
Hence, for any $i\in[\Nd]$, we have
\[
\begin{aligned}
(v_{\varsigma, \tau}^{(k)} - U_{\varsigma, \tau}^{(k)} \mu )_i \mathbb{I}_{( v_{\varsigma, \tau}^{(k)} - U_{\varsigma, \tau}^{(k)} \mu )_i > 0} & = \max\{ 0 , (v_{\varsigma, \tau}^{(k)} - U_{\varsigma, \tau}^{(k)} \mu )_i \} \\
& \le \max\{ 0 , [ U_{\varsigma, \tau}^{(k)} (  d_h - \mu )]_i \} \\
& \le | [ U_{\varsigma, \tau}^{(k)} (  d_h - \mu )]_i | \\
&\le \lVert U_{\varsigma, \tau,i}^{(k)} \rVert \lVert  d_h - \mu \rVert,
\end{aligned}
\]
where the third inequality follows from the Cauchy-Schwarz inequality. By the definition of $d_h$, we have
\[
d_h - \mu = \begin{cases}  
k\hat\mu - \mu & \text{if } \varsigma = 0 \text{ or } \tau = 0, \\ 
k\hat\mu - \mu + \sqrt{(1-\tau) / \tau} k^{\frac{s}{2}}\varsigma & \text{otherwise.}
\end{cases} 
\]
Therefore, summing up over $i\in[\Nd]$ yields that
\[
\begin{aligned}
 \sum\limits_{i\in[\Nd]} p_i (v_{\varsigma, \tau}^{(k)} - U_{\varsigma, \tau}^{(k)} \mu )_i \mathbb{I}_{( v_{\varsigma, \tau}^{(k)} - U_{\varsigma, \tau}^{(k)} \mu)_i > 0} & \le  \sum\limits_{i\in[\Nd]} p_i\lVert U_{\varsigma, \tau,i}^{(k)} \rVert \lVert   d_h - \mu \rVert \\
 & \le  L \lVert   d_h - \mu \rVert \\
& = \begin{cases}
  L \lVert k\hat\mu - \mu \rVert & \text{if } \varsigma = 0 \text{ or } \tau=0 ,\\
   L \lVert k\hat\mu - \mu + \sqrt{(1-\tau) / \tau} k^{\frac{s}{2}}\varsigma \rVert & \text{otherwise},  
\end{cases}
 \\
& \le \begin{cases}
   L\lVert k\hat\mu - \mu \rVert & \text{if } \varsigma = 0 \text{ or } \tau=0 ,\\
   L(\lVert k\hat\mu - \mu \rVert + \sqrt{(1-\tau) / \tau} k^{\frac{s}{2}} \lVert\varsigma \rVert ) & \text{otherwise},
\end{cases}
\end{aligned}
\]
where the second inequality follows from the definition of $L$, and the last inequality follows from the triangle inequality. Taking the supremum over $\{\mu: \lVert \mu - k\hat\mu  \rVert \le k^{\frac{s}{2}} \varepsilon  \}$ yields that
\[
\begin{aligned}
 & \sup\limits_{\mu: \lVert \mu - k\hat\mu  \rVert \le k^{\frac{s}{2}} \varepsilon }\sum\limits_{i\in[\Nd]} p_i (v_{\varsigma, \tau}^{(k)} - U_{\varsigma, \tau}^{(k)} \mu )_i \mathbb{I}_{( v_{\varsigma, \tau}^{(k)} - U_{\varsigma, \tau}^{(k)} \mu)_i > 0} \\
 \le &  \begin{cases}
   k^{\frac{s}{2}} L \varepsilon & \text{if } \varsigma = 0 \text{ or } \tau=0 ,\\
   k^{\frac{s}{2}} L \varepsilon + \sqrt{(1-\tau) / \tau} k^{\frac{s}{2}} L\lVert\varsigma \rVert  & \text{otherwise},
\end{cases}
\end{aligned}
\]
which completes the proof together with~\eqref{eq:C-wass}. 
\end{proof}

\begin{proposition}[Value upper bound (II) - Wasserstein DRO]\label{prop:value-upper-bound-wass}
    Under Assumptions~\ref{a:risk-measure},~\ref{a:model-param}\ref{a:model-param-1} and~\ref{a:ambiguity-wass}, consider the scaling scheme~\eqref{eq:wass-scaling} and a family of mechanism $M_{\varsigma,\tau}$ in~\eqref{eq:Mall-wass}. Then, the following hold: 
    \begin{enumerate}\renewcommand{\labelenumi}{(\roman{enumi})}
        \item Under the mechanism $M_0$,
        \[
        \bar V_{0}(b^{(k)},\theta^{(k)}) = V_{0}(b^{(k)},\theta^{(k)}).
        \]
        \item Under the mechanism $M_{\varsigma,\tau}$ with $(\varsigma, \tau) \in ( 0,\hat\sigma ] \times (0,\hat \tau_{\max} ] $,
        \[
     \bar V_{\varsigma, \tau}(b^{(k)},\theta^{(k)}) - V_{0} (b^{(k)},\theta^{(k)}) \le (\frac{1}{2}+\alpha )k^{\frac{s}{2}} L \lVert \hat\sigma \rVert + k^{\frac{s}{2}}L\varepsilon + \left(\sqrt{\frac{1-\tau}{\tau}} + \sqrt{\frac{\tau}{1-\tau} } \right) k^{\frac{s}{2}} L \lVert \varsigma \rVert.
    \]
    \end{enumerate} 
\end{proposition}

\begin{proof}[Proof of Proposition~\ref{prop:value-upper-bound-wass}.]
Note that the mechanism $M_{\varsigma,\tau}$ also maps the ambiguity information to either a two-point distribution ($\tau > 0$) or a Dirac distribution ($\tau = 0$). Most results established in Proof of Proposition~\ref{prop:value-upper-bound} can be applied directly. The proof of $\bar V_{0}(b^{(k)},\theta^{(k)}) = V_{0}(b^{(k)},\theta^{(k)})$ follows same steps in Proposition~\ref{prop:value-upper-bound}. For the case where $(\varsigma, \tau) \in ( 0,\hat\sigma ] \times (0,\hat \tau_{\max} ] $, we utilize the outer approximation set $\mathcal{R}(k\hat\mu,k^{\frac{s}{2}}\hat\sigma,k^{\frac{s}{2}}\varepsilon)$. We introduce the problem
 \begin{equation}\label{eq:V-hat-wass}
 \hat V_{\varsigma, \tau} (b^{(k)},\theta^{(k)}) \Let \left\{
    \begin{array}{cl}
    \min & c^\top x + \varrho_{M_{\varsigma, \tau}(\theta^{(k)})} (- p^\top \min\{ v ,U \tilde d \}) \\
        \st & x \in \R_+^{\Ny},~v \in \R_+^{\Ny},~ U\in\R_+^{\Ny\times \Nd} \\
            & A v \leq x,~ H U\leq I,  
    \end{array} 
    \right. 
\end{equation}
and it has been established that
\[
\bar V_{\varsigma, \tau} (b^{(k)},\theta^{(k)}) - V_0 (b^{(k)},\theta^{(k)}) \le \bar V_{\varsigma, \tau} (b^{(k)},\theta^{(k)}) - V_{\varsigma, \tau} (b^{(k)},\theta^{(k)}) + \hat V_{\varsigma, \tau} (b^{(k)},\theta^{(k)}) - V_0 (b^{(k)},\theta^{(k)}).
\]
Further, it also holds that
\[
\begin{aligned}
\bar V_{\varsigma, \tau} (b^{(k)},\theta^{(k)}) - V_{\varsigma, \tau} (b^{(k)},\theta^{(k)}) & \le \left(\sqrt{\frac{1-\tau}{\tau}} + \sqrt{\frac{\tau}{1-\tau}} \right)k^{\frac{s}{2}} p^\top \bar U_{\varsigma,\tau}^{(k)} \varsigma \\
& \le \left(\sqrt{\frac{1-\tau}{\tau}} + \sqrt{\frac{\tau}{1-\tau}} \right)k^{\frac{s}{2}}  L \lVert\varsigma\rVert.
\end{aligned}
\]
The difficulty lies the upper bound for $\hat V_{\varsigma, \tau} (b^{(k)},\theta^{(k)}) - V_0 (b^{(k)},\theta^{(k)})$ as $M_{\varsigma,\tau}(\theta^{(k)})$ may not belong to $\mathcal{A}(\theta)$. However, $M_{\varsigma,\tau}(\theta^{(k)})$ has mean $k\hat\mu$ and marginal variance $k^{\frac{s}{2}}\varsigma\le k^{\frac{s}{2}}\hat\sigma$, which implies that $M_{\varsigma,\tau}(\theta^{(k)})\in\mathcal{R}(k\hat\mu,k^{\frac{s}{2}}\hat\sigma,k^{\frac{s}{2}}\varepsilon)$. Hence, we have
\[
\begin{aligned}
\hat V_{\varsigma, \tau} (b^{(k)},\theta^{(k)}) & \le c^\top x_0^{(k)} + \varrho_{M_{\varsigma,\tau}(\theta^{(k)})} ( -p^\top \min\{ v_0^{(k)} , U_0^{(k)} \} ) \\
& \le c^\top x_0^{(k)} + \sup\limits_{\PP\in \mathcal{R}(k\hat\mu,k^{\frac{s}{2}}\hat\sigma,k^{\frac{s}{2}}\varepsilon)} \varrho_{\PP}( -p^\top \min\{ v_0^{(k)} , U_0^{(k)} \} ) \\
& \le \bar V_0(b^{(k)},\theta^{(k)}) + \sup\limits_{\PP\in \mathcal{R}(k\hat\mu,k^{\frac{s}{2}}\hat\sigma,k^{\frac{s}{2}}\varepsilon)} (p^\top\max\{ 0, v_0^{(k)} - U_0^{(k)} \tilde d\}) \\
& \le \bar V_0(b^{(k)},\theta^{(k)}) + (\frac{1}{2}+\alpha) k^{\frac{s}{2}}p^\top  U_0^{(k)} \hat\sigma + \sup\limits_{\mu:\lVert \mu-k\hat\mu \rVert \le k^{s/2}\varepsilon} p^\top U_0^{(k)}(k\hat\mu - \mu)\\
& \le \bar V_0(b^{(k)},\theta^{(k)}) + (\frac{1}{2}+\alpha)k^{\frac{s}{2}} L\lVert\hat\sigma\rVert + k^{\frac{s}{2}}L\varepsilon   , 
\end{aligned}
\]
where the first inequality follows from the suboptimality of $(x_0^{(k)}, v_0^{(k)}, U_0^{(k)})$ to the problem defining $\hat V_{\varsigma,\tau}(b^{(k)},\theta^{(k)})$ and the second inequality follows because $M_{\varsigma,\tau}(\theta^{(k)})\in\mathcal{R}(k\hat\mu,k^{\frac{s}{2}}\hat\sigma,k^{\frac{s}{2}}\varepsilon)$. The third inequality follows from the analysis of~\eqref{eq:to-be-used-1} and the forth inequality follows from~\eqref{eq:to-be-used-2}. The last inequality follows from the Cauchy-Schwarz inequality and the definition of $L$. The equality $\bar V_0(b^{(k)},\theta^{(k)}) = V_0(b^{(k)},\theta^{(k)})$ implies that
\[
\hat V_{\varsigma, \tau} (b^{(k)},\theta^{(k)}) -  V_0 (b^{(k)},\theta^{(k)}) \le (\frac{1}{2}+\alpha)k^{\frac{s}{2}} L\lVert\hat\sigma\rVert + k^{\frac{s}{2}}L\varepsilon,
\]
and thus,
\[
\bar V_{\varsigma, \tau} (b^{(k)},\theta^{(k)}) -  V_0 (b^{(k)},\theta^{(k)}) \le (\frac{1}{2}+\alpha)k^{\frac{s}{2}} L\lVert\hat\sigma\rVert + k^{\frac{s}{2}}L\varepsilon + (\sqrt{(1-\tau)/\tau} +\sqrt{\tau/(1-\tau)} )k^{\frac{s}{2}} L\lVert\varsigma\rVert,
\]
which completes the proof. 
\end{proof}

Note that the budget $b^{(k)}$ and the mean of nominal distribution $\hat\PP^{(k)}$ scales linearly under the scaling scheme~\eqref{eq:wass-scaling}. It follows from Lemma~\ref{lemma:risk-pos-homo} that the positive homogeneity holds: $V_{0}(b^{(k)},\theta^{(k)}) = k V_{0}(b^{(1)},\theta^{(1)})$. This positive homogeneity is crucial to prove Theorem~\ref{thm:W-opt}.

\subsubsection{Proof of Asymptotic Optimality.}

\begin{proof}[Proof of Theorem~\ref{thm:W-opt}.]
We start with the mechanism $M_0$ with the corresponding lower-level solution $x_0^{(k)}$ for each $k$. By Proposition~\ref{prop:loss-upper-bound-wass} and~\ref{prop:value-upper-bound-wass}, it holds that
\begin{align*}
C(x_{0}^{(k)},\mathbf{y}_{0}^{(k)},b^{(k)},\theta^{(k)}) & \le  V_{0}(b^{(k)}, \theta^{(k)}) + \big( (\frac{1}{2}+\alpha)L\lVert\hat\sigma\rVert + L\varepsilon \big) k^{\frac{s}{2}} \\
& = k V_0(b^{(1)},\theta^{(1)}) + \big( (\frac{1}{2}+\alpha)L\lVert\hat\sigma\rVert + L\varepsilon \big) k^{\frac{s}{2}},
\end{align*}
where the equality follows from the positive homogeneity $V_{0}(b^{(k)}, \theta^{(k)}) = k V_0(b^{(1)},\theta^{(1)})$. Assumption~\ref{a:model-param}\ref{a:model-param-2} presumes that $V_0(b^{(1)},\theta^{(1)}) < 0$, under which the negative linear term $k V_0(b^{(1)},\theta^{(1)})$ will dominate the positive sub-linear term $\big( (\frac{1}{2}+\alpha)L\lVert\hat\sigma\rVert + L\varepsilon \big) k^{\frac{s}{2}}$ as $k$ grows, and thus, it holds that
\[
C(x_{0}^{(k)},\mathbf{y}_{0}^{(k)},b^{(k)},\theta^{(k)}) \le k V_0(b^{(1)},\theta^{(1)}) + \big( (\frac{1}{2}+\alpha)L\lVert\hat\sigma\rVert + L\varepsilon \big) k^{\frac{s}{2}} < 0
\]
for sufficiently large $k$. The negativity of $\mathrm{OPT}(b^{(k)},\theta^{(k)})$ and $\mathrm{Obj}(x_0^{(k)},b^{(k)},\theta^{(k)})$ inherit from the suboptimality relationship
\[
\mathrm{OPT}(b^{(k)},\theta^{(k)}) \le \mathrm{Obj}(x_0^{(k)},b^{(k)},\theta^{(k)}) \le C(x_{0}^{(k)},\mathbf{y}_{0}^{(k)},b^{(k)},\theta^{(k)}).
\]
Note that the ratio between negative terms equals the ratio between their absolute values. 
Hence, it holds that
\begin{align*}
1  \ge \frac{\mathrm{Obj}(x_{0}^{(k)}, b^{(k)},\theta^{(k)})} { \mathrm{OPT}(b^{(k)},\theta^{(k)} ) } & \ge \frac{C(x_{0}^{(k)},\mathbf{y}_{0}^{(k)},b^{(k)},\theta^{(k)}) }{ V_{0}(b^{(k)},\theta^{(k)})} \\
& \ge \frac{kV_{0}(b^{(1)},\theta^{(1)}) + \big( (\frac{1}{2}+\alpha)L\lVert\hat\sigma\rVert + L\varepsilon \big) k^{\frac{s}{2}} }{kV_{0}(b^{(1)},\theta^{(1)})} \\
& = 1 + \frac{ (\frac{1}{2}+\alpha)L\lVert\hat\sigma\rVert + L\varepsilon }{V_{0}(b^{(1)},\theta^{(1)})} k^{\frac{s}{2}-1} \to 1 \text{ as } k \to\infty,    
\end{align*}
where the convergence to $1$ follows because $s\in[1,2)$ and then $ k^{\frac{s}{2}-1} \to 0$ as $k \to\infty$. By the sandwich theorem, we have
\[
\lim\limits_{k\to\infty } \frac{\mathrm{Obj}(x_{0}^{(k)}, b^{(k)},\theta^{(k)})} { \mathrm{OPT}(b^{(k)},\theta^{(k)} ) } = 1 .
\]

Now we proceed with $(\varsigma,\tau)\in (0,\hat\sigma]\times (0,\hat\tau_{\max} ]$. By Proposition~\ref{prop:loss-upper-bound-wass} and Proposition~\ref{prop:value-upper-bound-wass}, we have
\[
C(x_{\varsigma,\tau}^{(k)},\mathbf{y}_{\varsigma,\tau}^{(k)},b^{(k)},\theta^{(k)}) \le  k V_{0}(b^{(1)}, \theta^{(1)}) + \big( (1+2\alpha)\lVert\hat\sigma\rVert + 2 \varepsilon + 2\sqrt{(1-\tau)/\tau} + \sqrt{\tau/(1-\tau) }  \big)L k^{\frac{s}{2}},
\]
which also consists of a negative linear term and a positive sub-linear term. Consequently, for sufficiently large $k$, it holds that
\[
\mathrm{OPT}(b^{(k)},\theta^{(k)}) \le \mathrm{Obj}(x_{\varsigma,\tau}^{(k)},b^{(k)},\theta^{(k)}) \le C(x_{\varsigma,\tau}^{(k)},\mathbf{y}_{\varsigma,\tau}^{(k)},b^{(k)},\theta^{(k)}) < 0.
\]
Hence, taking the ratio yields that
\begin{align*}
1  & \ge \frac{\mathrm{Obj}(x_{\varsigma,\tau}^{(k)}, b^{(k)},\theta^{(k)})} { \mathrm{OPT}(b^{(k)},\theta^{(k)} ) }  \ge \frac{C(x_{\varsigma,\tau}^{(k)},\mathbf{y}_{\varsigma,\tau}^{(k)},b^{(k)},\theta^{(k)}) }{ V_{0}(b^{(k)},\theta^{(k)})} \\
& \ge \frac{kV_{0}(b^{(1)},\theta^{(1)}) + \big( (1+2\alpha)\lVert\hat\sigma\rVert + 2 \varepsilon + 2\sqrt{(1-\tau)/\tau} + \sqrt{\tau/(1-\tau) }  \big)L k^{\frac{s}{2}} }{kV_{0}(b^{(1)},\theta^{(1)})} \\
& = 1 + \frac{ \big( (1+2\alpha)\lVert\hat\sigma\rVert + 2 \varepsilon + 2\sqrt{(1-\tau)/\tau} + \sqrt{\tau/(1-\tau) }  \big)L }{V_{0}(b^{(1)},\theta^{(1)})} k^{\frac{s}{2}-1} \\
& \to 1 \text{ as } k \to\infty,    
\end{align*}
which completes the proof. 
\end{proof}

\section{Applications} \label{appendix:application}

We showcase that many applications satisfy Assumption~\ref{a:model-param}. 
\subsection{Capacity Management Problem with Product Substitution} \label{sec:ato}

We examine the capacity management problem with product upgrades in~\cite{ref:shumsky2009dynamic}. In product upgrades, when the capacity for a specific product is exhausted, the firm may offer a substitute product to fulfill the customer's request. Product upgrades can effectively handle uncertain demand from multiple customer segments, and various service industries also provide upgrades to their customers. For example, car rental agencies offer upgrades by providing customers with more luxurious or expensive car models than initially booked~\citep{ref:carroll1995evolutionary}.
Similar practices occur in hotels with room grade allocations and airlines with class upgrades during overbooking~\citep{ref:karaesmen2004overbooking}. The uncertain demand for each product is always non-negative. We illustrate the modeling framework via a classical assemble-to-order (ATO) problem.

In an ATO system, the firm maintains component stocks $x \in \R_+^{\Nx}$ and upon receiving orders $d\in\R_+^{\Nd}$, it assembles the component stocks into products $y \in \R_+^{\Ny}$ to meet the orders. The unit ordering cost for the components is $c\in\R_+^{\Nx}$, while the unit selling price for products is $p\in\R_+^{\Ny}$. We use an incidence matrix $A\in\mathbb{N}^{\Nx\times\Ny}$ to model the componential topology where the entry $A_{ij}$ represents the quantity of component $i$ utilized by one unit of product $j$.
The firm may upgrade the customer to a higher-quality alternative if the requested product is unavailable. However, if the customer refuses any substitution, we have $\Ny=\Nd$, and the matrix $H$ becomes an identity matrix. Trivially, we can choose $U = I_{\Nd}$ to satisfy the second requirement in Assumption~\ref{a:model-param}, and it is the optimal solution for $U$.

In scenarios where product substitutions are feasible, determining the inventory level of components involves identifying the set of products $\mathcal{J}_i$ that can serve the demand from the $i$-th customer segment. This prioritization allows the firm to fulfill as much demand as possible by leveraging available resources and product options.
Without loss of generality, we assume that the initial $n_1=|\mathcal{J}_1|$ products are designated for the first customer segment, followed by the subsequent $n_2=|\mathcal{J}_2|$ products allocated to the second customer segment, and so on. This ordering scheme allows us to establish a clear correspondence between the products and the respective customer segments they serve. Furthermore, we can express $\Ny$ as the sum of the individual customer segment sizes, i.e., $\Ny=\sum_{i\in[\Nd]}n_i$. The technical matrix $H$ has the following expression:
\begin{equation*}\label{eq:H-def}
H=\left[
\begin{array}{cccc}
    \mathbf{1}_{n_1}^\top &  & & \\
     & \mathbf{1}_{n_2}^\top & & \\
     & & \ldots&\\
     & & & \mathbf{1}_{n_{\Nd}}^\top 
\end{array}
\right] \in\R_+^{\Nd \times \Ny},
\end{equation*}
where $\mathbf{1}_{n}$ is an $n$-dimensional column vector of all ones. For any $y\in\R_+^{\Ny},z \in\R_+^{\Nd}$ such that $Hy\le z$, we can always find one desired $U\in\R_+^{\Ny\times \Nd}$ to satisfy $HU\le I_{\Nd\times\Nd}$ and $y\le U z$. Consider $U$ in the form
\[
U=\left[
\begin{array}{cccc}
s_{\mathcal J_1} &  & & \\
& s_{\mathcal J_2} & & \\
& & \ldots&\\
& & & s_{\mathcal J_{\Nd}}
\end{array}
\right]\in\R_+^{\Ny\times\Nd},
\]
where $s_{\mathcal J_i}\in\R_+^{n_i}$ and $\sum_{j\in\mathcal J_i } s_{j}\le 1$. Then $U\in\mathcal U$, i.e., $U\geq 0$ and 
\[
HU=\diag\left(\sum_{j\in\mathcal J_1} s_j,\ldots,\sum_{j\in \mathcal J_{\Nd} } s_j\right)\leq I_{\Nd}. 
\]
Moreover, we can find $s_{\mathcal J_1},\ldots,s_{\mathcal J_{\Nd}}$ such that $y\leq U z$, e.g., for $j\in\mathcal J_i$, set $s_j=y_j / \sum_{k\in\mathcal J_i} y_k $.
Because  $ z_i\geq \sum_{k\in\mathcal J_i} y_k$, we further have
\[ \begin{aligned}
U z & =\left( z_1 s_{\mathcal J_1}^\top,\ldots,z_{\Nd} s_{\mathcal J_{\Nd}}^\top \right)^\top =\left( \frac{y_1}{\sum\limits_{k\in\mathcal J_1} y_k} z_1,~\ldots,~ \frac{y_{n_1}}{\sum\limits_{k\in\mathcal J_1} y_k} z_1,~\frac{y_{n_1+1}}{\sum\limits_{k\in\mathcal J_2} y_k} z_2,~\ldots\right)^\top  \geq y,
\end{aligned} \]
thus the second assumption in Assumption~\ref{a:model-param} is also satisfied.

\subsection{Two-echelon Inventory Problem}\label{appendix:inventory}

Our model framework extends to the two-echelon inventory problem, accommodating scenarios with either a single warehouse or multiple warehouses. We begin by considering the single warehouse case in \citet{ref:jonsson1987analysis}. A central warehouse determines its inventory levels at the start of each ordering cycle. Subsequently, the warehouse allocates products to various retail stores to maximize revenue. This control problem naturally fits a two-stage optimization structure.

Assume the warehouse manages inventory for $\Nx$ products. Before observing the actual demand, we must determine their warehouse inventory levels, represented by the vector $x \in \R_+^{\Nx}$. We denote the unit ordering costs by $c \in \R_+^{\Nx}$. Let $N_s$ denote the number of stores, where each store $i \in [N_s]$ offers a subset of products $\mathcal{J}_i \subseteq [\Nx]$. For each store $i$, the inventory levels and demands for its products are represented by vectors $y^i \in \R_+^{|\mathcal{J}_i|}$ and $d^i \in \R_+^{|\mathcal{J}_i|}$, respectively. The selling prices for products in each store $i$ are given by the vector $p^i \in \R_+^{|\mathcal{J}_i|}$. Consequently, there are a total of $\Ny = \sum_{i=1}^{N_s} |\mathcal{J}_i|$ store-product combinations. We define the overall inventory level vector as $y = (y^1; \ldots; y^{N_s})$, the overall demand vector as $d = (d^1; \ldots; d^{N_s})$, and the overall unit selling price vector as $p = (p^1; \ldots; p^{N_s})$.

We introduce the topology matrix $A \in \R_+^{\Nx \times \Ny}$ to capture the relationship between products and store-product combinations. This matrix is pivotal in determining the allocation plan. Specifically, consider the first $|\mathcal{J}_1|$ columns of $A$, at $(k,j)\in[\Nx]\times[|\mathcal{J}_1|]$, the element is defined as 
\[
A_{kj} = \begin{cases}
    1 & \text{if $k$ is the $j$-th element in $\mathcal{J}_1$}, \\
    0 & \text{otherwise.}
\end{cases}
\]
The next $|\mathcal{J}_2|$ columns of $A$ are defined to be
\[
A_{kj} = \begin{cases}
    1 & \text{if $k$ is the $(j-|\mathcal{J}_1|)$-th element in $\mathcal{J}_2$}, \\
    0 & \text{otherwise,}
\end{cases}
\]
where $k\in[\Nx]$ and $j\in\{ |\mathcal{J}_1| + 1, \ldots, |\mathcal{J}_1| + |\mathcal{J}_2| \}$. The next $|\mathcal{J}_3|$ columns of $A$ are defined to be
\[
A_{kj} = \begin{cases}
    1 & \text{if $k$ is the $(j-|\mathcal{J}_1|-|\mathcal{J}_2|)$-th element in $\mathcal{J}_3$}, \\
    0 & \text{otherwise,}
\end{cases}
\]
where $k\in[\Nx]$ and $j\in\{ |\mathcal{J}_1| + |\mathcal{J}_2| + 1, \ldots, |\mathcal{J}_1| + |\mathcal{J}_2| + |\mathcal{J}_3|\}$.
This procedure is applied sequentially to cover all $\Ny = \sum_{i\in[N_s]} |\mathcal{J}_i|$ columns, thereby fully constructing the topology matrix $A$.

With these definitions, the second-stage optimization problem has the formulation
\[
g(x,d) \Let \left\{
\begin{array}{cl}
    \min & -p^\top y \\
    \st  & y\in\R_+^{\Ny} \\
        & Ay\leq x,~ y\leq d,
    \end{array}
\right.
\]
where the matrix $H$ is the identity matrix, and the constraints ensure that the allocations do not exceed the available inventory $x$ or the realized demand $d$.

The modeling paradigm can be generalized to multi-warehouse, multi-store systems as explored in \citet{ref:miao2022asymptotically}. In such extensions, the decision variables $x$ and $y$ have higher dimensions to capture the replenishment structure between stores and warehouses and the topology matrix $A$. Additionally, the topology matrix $H$ must be defined to reflect scenarios where store $i$ can replenish product $j$ from multiple warehouses, introducing complexities similar to the substitution effects observed in Assemble-to-Order (ATO) problems discussed in Section~\ref{sec:ato}. Consequently, the model parameters would also satisfy Assumption~\ref{a:model-param}.

\section{Simulation Details}\label{appendix:simulation}
\subsection{Synthetic Instances Construction}\label{appendix:instances}

In the synthetic experiments, we focus on the assemble-to-order problems with no product flexibility, i.e., a one-to-one correspondence between the products and the demand types. Hence, the products and demands have the same dimension denoted by $N = \Ny = \Nd $. The recourse matrix $H$ is the identity matrix $I$, and the constraint $Hy\le d$ reduces to $y\le d$. The firm first maintains component stocks $x \in \R_+^{\Nx}$ and then assembles the component stocks into products $y \in \R_+^{N}$ to meet the demand upon receiving orders $d\in\R_+^{N}$. The unit ordering cost for the components is $c\in\R_+^{\Nx}$, while the unit selling price for products is $p\in\R_+^{N}$. The incidence matrix $A\in\mathbb{N}^{\Nx\times N}$ models the componential topology where the entry $A_{ij}$ represents the quantity of component $i$ utilized by one unit of product~$j$.  

We generate a comprehensive set of 640 problem instances characterized by $(c,p, A,b,\theta)$ to evaluate the performance of our proposed mechanisms and decision-making strategies. The parameter $\theta$ collects the mean and marginal standard deviation information, leading to an ambiguity set $\mathcal A(\theta)$ as defined in~\eqref{eq:ambiguity}.
We create 80 combinations of $(c,p,A)$ following~\citet{ref:devalve2020primal}, along with four sets of parameters $\theta=(\mu,\sigma)$ as follows.

\noindent\textbf{Model parameters $(c,p,A)$.} We use eight choices for the assembly structures: we use five structures in~\cite{ref:devalve2020primal} where the number of products ranges from $2$ to $14$; additionally, we use three structures defined as
\[
A = \begin{bmatrix}
     I_{(N-1) \times (N-1)}& \nu\mathbf{1}  \\
    \vartheta^\top & \nu
\end{bmatrix},
\]
where $N\in\{10,20,30\}$, $\nu=2$ and $\vartheta=(1,\ldots,N-1) \in \R^{N-1}$. For the unit cost of components, we consider two sets of $c$:
\begin{enumerate}[label=(\roman*)]
    \item $c= (1, 1, 1, 1, 1, \ldots)$,
    \item $c= (1.5, 2, 1.5, 2, \ldots)$.
\end{enumerate}
Given unit cost $c$ and assemble-to-order matrix $A$, we define the unit price of $j$-th product as $p_j=\gamma_j A_j^\top c$ for $j \in [N]$, where $A_j^\top c$ represents the total component cost of product $j$ and $\gamma_j$ is its markup. We use four sets of markups $\gamma\in\R_+^N$ in~\citet{ref:devalve2020primal}:
\begin{enumerate}[label=(\roman*)]
    \item $ \gamma = ( 1.1, 1.1, 1.1, 1.1, \ldots)$,
    \item $ \gamma = ( 3, 3, 3, 3, \ldots)$,
    \item $ \gamma = ( 2, 1.5, 2, 1.5, \ldots) $,
    \item $ \gamma =(3, 1.1, 3, 1.1, \ldots)$.
\end{enumerate}

\noindent \textbf{Uncertainty Information $\theta$.} For the moment information $\theta=(\mu,\sigma)$, we choose five mean and marginal variance combinations:
\begin{enumerate}[label=(\roman*)]
    \item $\mu=(20,20,20,\ldots)$ and $\sigma^2=(20,20,20,\ldots)$,
    \item $\mu=(20,30,20,30,\ldots)$ and $\sigma^2=(20,20,20,\ldots)$,
     \item $\mu=(20,40,20,40,\ldots)$ and $\sigma^2=(20,40,20,40,\ldots)$,
    \item $\mu=(20,40,20,40,\ldots)$ and $\sigma^2=(20,60,20,60,\ldots)$,
    \item $\mu=(20,30,40,20,30,40,\ldots)$ and $\sigma^2=(20,40,60,20,40,60,\ldots)$.
\end{enumerate}
For the budget $b$, we consider two scenarios: $b = 0.5 c^\top A\mu $ and $b=2 c^\top A\mu$,
which represents the financial budget available for component stocks and defines the first-stage feasible region as $\mathcal X(b) \triangleq \{ x\in\R_+^{\Nx}: c^\top x \le b\}$. We assess the performance of our proposed mechanisms and decisions on these 640 instances to ensure a comprehensive evaluation.

\subsection{Computational Time Comparison against TLDR Approximation}\label{appendix:time-comparison}
We investigate the computational time of our decentralized approach and the TLDR approximation under different dimensions of uncertainty. We consider an assemble-to-order system with $N$ components and $N$ products. The assembly structure has the topology matrix
\[
A=\begin{bmatrix}
    I_{(N-1) \times (N-1)}& \nu\mathbf{1}  \\
    \vartheta^\top & \nu
\end{bmatrix},
\]
where $\nu=2$ and $\vartheta=(1,\ldots,N-1) \in \R^{N-1}$. The problem size relies on the uncertainty dimension $N$, which also represents the number of products. We increment the dimension $N$ from 10 to 500 in steps of 10. For each value of $N$, we construct 40 instances characterized by the parameters $(c,p,\theta)$ following the procedures outlined in Appendix~\ref{appendix:instances}. To ensure the reliability of our analysis, each instance configuration is executed 10 times to obtain stable timing measurements. 
\begin{figure}[htbp]
    \centering
    \includegraphics[width=0.8\linewidth]{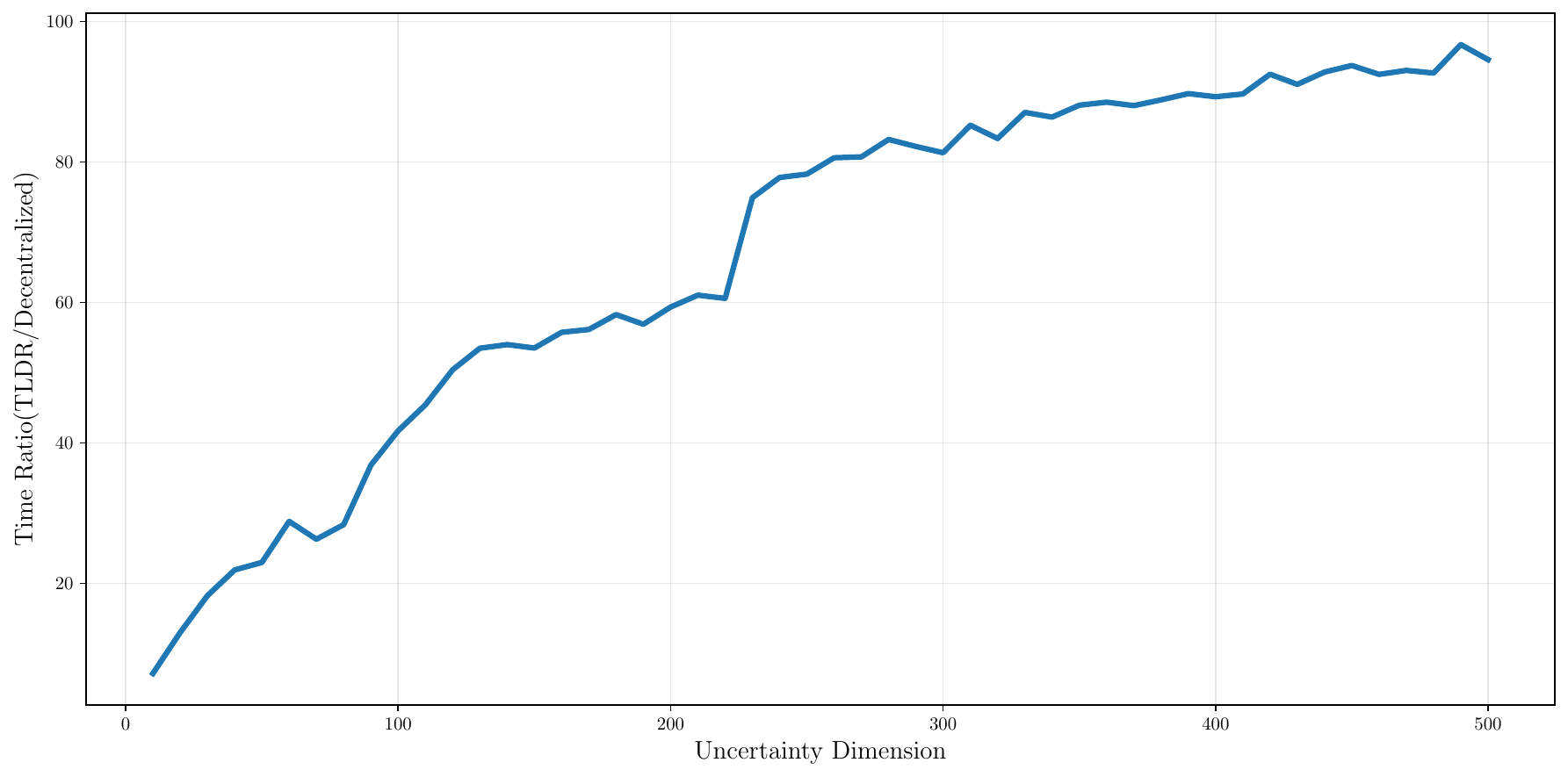}
    \caption{Computational time ratio between the TLDR approximation and our decentralized approach under different uncertainty dimensions. A higher ratio indicates a greater computational advantage of our decentralized approach.}
    \label{fig:time-comparison}
    
\end{figure}

We plot the average computational time ratio between the TLDR approximation and our decentralized approach under different uncertainty dimensions in Figure~\ref{fig:time-comparison}. The TLDR approximation involves solving a second-order cone program, whereas our decentralized approach only requires solving a linear program. Consequently, our approach consistently demonstrates superior computational performance across all tested uncertainty dimensions. At dimension 500, the decentralized approach executes approximately 100 times faster than the TLDR approximation. This remarkable scaling efficiency indicates that our approach is particularly well-suited for large-scale applications. Additionally, the computational advantage of our approach increases with the uncertainty dimension, as evidenced by the upward trend in the ratio. This pattern suggests that our decentralized approach becomes increasingly beneficial as the size of the uncertainty grows.

\subsection{Robustness and Data-Driven Performance}\label{appendix:simulation-robust}
Before presenting the results, we explain how we generate training and testing datasets and conduct cross-validation to find parameters $(\varsigma\opt,\tau\opt)$.

\noindent\textbf{Data generation.}
To control for sampling bias, we first generate a dataset $\mathcal{D} = \{ d_\omega \}$ from the standard multivariate normal distribution $\mathcal{N}(0, I)$, where we use the subscript $\omega$ as the sample index. These samples are then transformed linearly to be $ \Sigma^{\frac{1}{2}} d_\omega + \mu$ to follow the distribution $\mathcal{N}(\mu,\Sigma)$. Lastly, we truncate the samples to $[0, \infty)^{\Nd}$ to ensure non-negativity. We derive the training dataset $\mathcal{D}_{\tr}$ and the testing dataset $\mathcal{D}_{\te}$ by setting $\Sigma$ to be $\Sigma_{\tr}$ and $\Sigma_{\te}$, respectively. Consequently, their correlation coefficients are the primary distinction between training and testing distribution. Further, the number of samples in the training dataset $\mathcal{D}_{\tr}$ varies as $n_{\tr}\in\{100,200,400\}$ to investigate the effect of sample size. The testing dataset $\mathcal{D}_{\te}$ has $n_{\te} = 1000$ samples.

\noindent\textbf{Cross-validation procedure.} We implement 5-fold cross-validation on $\mathcal D_{\tr}$, creating five pairs of training and validation datasets. For each training fold, we estimate its mean and marginal standard deviation $(\hat\mu,\hat\sigma)$, based on which we construct the mechanism $M_{\varsigma, \tau}$ as defined in~\eqref{eq:Mtau}.

To address the variation in empirical moments across training folds, which affects the maximum feasible values of $\varsigma$ and $\tau$, we parameterize these variables as proportions of their maximum feasible values. Specifically, we set $\varsigma=\kappa\hat\sigma$ and $\tau=\eta \hat\tau_{\max}$, where $\hat\tau_{\max}$ depends on $\varsigma$ and $\hat\mu$. We then search for optimal proportions $(\kappa\opt,\eta\opt)$ through a grid search over $\kappa,\eta\in\{0,0.05,0.10,\ldots,1\}$. For each proportion combination $(\kappa,\eta)$ and training fold, we first compute the first-stage decision $x_{\varsigma,\tau}$ using $(\varsigma,\tau) = (\kappa\hat\sigma,\eta\hat\tau_{\max})$. Under this first-stage decision, we then solve problem~\eqref{eq:V-bar} to find the optimal TLDR $\mathbf{y}_{\varsigma,\tau}: d\mapsto \min\{ v_{\varsigma,\tau}, d \}$. We evaluate the validation performance through the sum of the first-stage cost and the CVaR of second-stage costs $-p^\top\mathbf{y}_{\varsigma,\tau}(\tilde d)$. The CVaR computation follows an efficient reordering procedure: we first calculate the realized costs $-p^\top\mathbf{y}_{\varsigma,\tau}(d_\omega)$ for each sample $\omega$ in the validation fold, sort them in descending order, and compute the average of the $\lceil\beta n_{\mathrm{val}}\rceil$ largest values, where $\beta$ denotes the CVaR confidence level and $n_{\mathrm{val}}$ represents the validation fold size. Notably, each iteration in the cross-validation procedure requires only two linear programs: one for the first-stage decision and another for the second-stage TLDR.
We select $(\kappa\opt,\eta\opt)$ with the best average performance across five validation datasets. At last, we apply $(\kappa\opt,\eta\opt)$ to the training dataset's estimated moments $(\mu_{\tr},\sigma_{\tr})$, that is, $(\varsigma\opt,\tau\opt) = (\kappa\opt \sigma_{\tr},\eta\opt \tau_{\max})$ where $\tau_{\max}$ is determined by $(\mu_{\tr},\varsigma\opt)$. These parameters are then used to construct $M_{\varsigma\opt,\tau\opt}$ and determine the CV-tuned solution $x_{\varsigma\opt,\tau\opt}$.

\noindent \textbf{Robustness Index under different training sample sizes.} Figure~\ref{fig:robust-cvar_100} shows the Robustness Index across 320 instances under different budgets and correlation coefficient shifts with training sample size $n_{\tr} =100$, and Figure~\ref{fig:robust-cvar_400} plots the results with training sample size $n_{\tr} =400$. Both figures exhibit asymmetric outlier distributions, as with training sample size $n_{\tr} =200$. When $n_{\tr}=100$, the positive outliers reach up to 1.0 (higher than 0.8 when $n_{\tr}\in\{200,400\}$) under large positive correlation coefficient shifts.

\begin{figure}[!ht]
    \centering
       \includegraphics[width=0.95\linewidth]{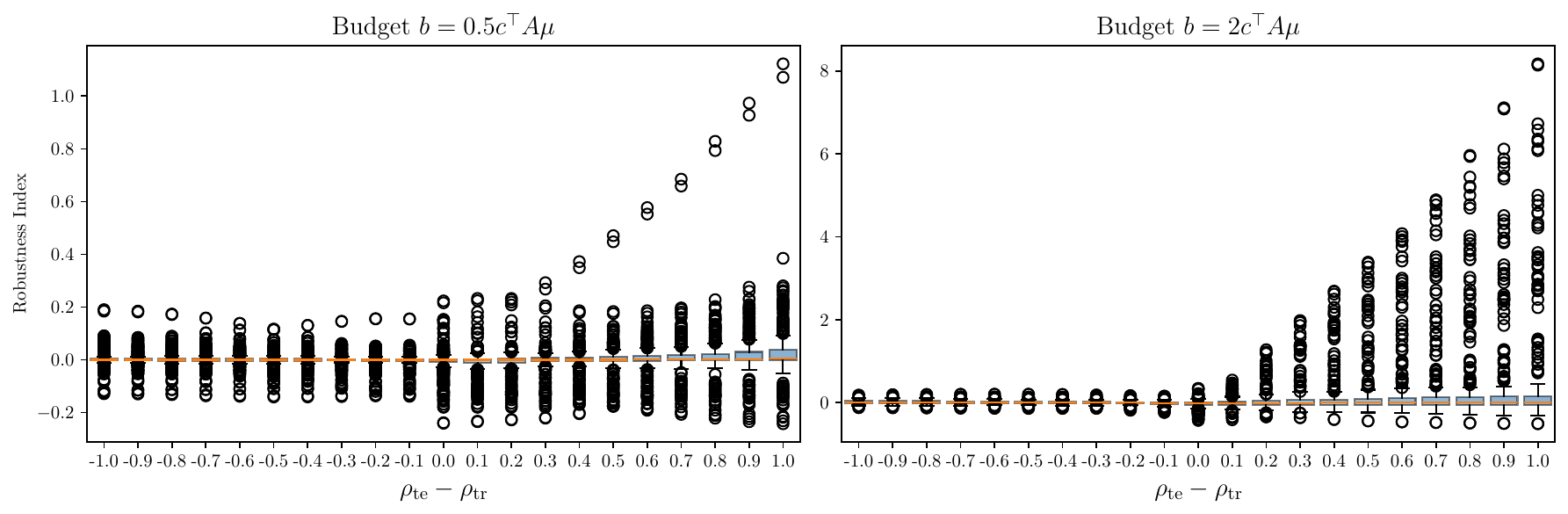}
    \caption{The Robustness Index of CV-tuned solution $x_{\varsigma\opt,\tau\opt}$ across 320 instances under different correlation coefficient shifts $\rho_{\te} - \rho_{\tr}$ for two budget scenarios when the training dataset has $n_{\tr}=100$ samples. A higher positive Robustness Index implies greater out-of-sample performance improvement over the SAA method.}
    \label{fig:robust-cvar_100}
    
\end{figure}

\begin{figure}[!ht]
    \centering
       \includegraphics[width=0.95\linewidth]{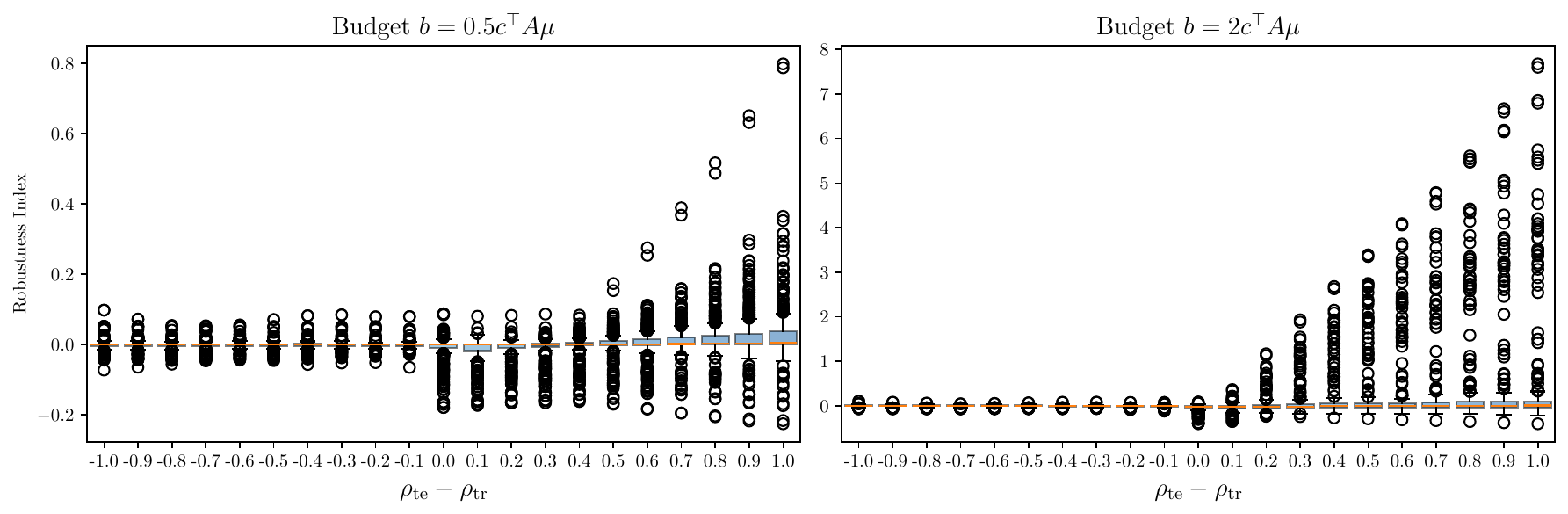}
    \caption{The Robustness Index of CV-tuned solution $x_{\varsigma\opt,\tau\opt}$ across 320 instances under different correlation coefficient shifts $\rho_{\te} - \rho_{\tr}$ for two budget scenarios when the training dataset has $n_{\tr}=200$ samples. A higher positive Robustness Index implies greater out-of-sample performance improvement over the SAA method.}
\label{fig:robust-cvar_400}
    
\end{figure}

\noindent \textbf{One-sided sign test.} We use a one-sided sign test to analyze the statistical significance of the out-of-sample performance improvement over the SAA method. Define the random variable 
\[
\tilde Z_{\varsigma\opt, \tau\opt} \triangleq J(x_{\varsigma\opt, \tau\opt}( \mathcal{D}_{\tr} ) , \mathcal{D}_{\te}) - J(x_{\text{SAA}} ( \mathcal{D}_{\tr} ), \mathcal{D}_{\te}),
\]
where $J(x, \mathcal{D}_{\te}) \Let c^\top x + \hat\PP_{\te}\text{-CVaR}_{0.05} ( g(x,\tilde d) )$.
The random variable $\tilde Z_{\varsigma\opt, \tau\opt}$ depends on the parameters $(c,p,A,b,\theta,\mathcal{D_{\tr}}, \mathcal{D_{\te}} )$, where $\mathcal{D_{\tr}}$ and $\mathcal{D_{\te}}$ are further determined by the training sample size $n_{\tr}$ and the correlation coefficients $\rho_{\tr},\rho_{\te}$.
For each combination of $ ( b, |\rho_{\te} - \rho_{\tr} |, n_{\tr} )$, we have 640 observations of $Z_{\varsigma\opt, \tau\opt} $, based on which we apply a one-sided sign test to assess whether there is a significant out-of-sample performance improvement over the SAA method. 

\begin{table}[htbp]
  \centering
  
  \caption{The $p$-values of the two one-sided tests under different correlation coefficient shifts $ \rho_{\te} - \rho_{\tr}$ for two budget scenarios with training sample size $n_{\tr} =100$. Bold values indicate that the $p$-value is smaller than 0.05. }  \label{tab:signed-test-100}
    \begin{tabular}{cccccccc}
    \toprule
    \multirow{2}[2]{*}{$\rho_{\te} - \rho_{\tr} $ } & \multicolumn{2}{c}{Budget $b= 0.5c^\top A\mu$}       &       & \multicolumn{2}{c}{ Budget $b= 2c^\top A\mu$ } \\
\cmidrule{2-3}\cmidrule{5-6}     & $\mathcal{H}_+$    & $\mathcal{H}_-$   &   &  $\mathcal{H}_+$     & $\mathcal{H}_-$     \\
\cmidrule{1-3}\cmidrule{5-6}  
        -1.0    & \textbf{3.38E-05} & 1.00E+00 &       & \textbf{3.80E-10} & 1.00E+00 \\
    -0.9  & \textbf{3.38E-05} & 1.00E+00 &       & \textbf{3.16E-09} & 1.00E+00 \\
    -0.8  & \textbf{3.14E-04} & 1.00E+00 &       & \textbf{1.21E-08} & 1.00E+00 \\
    -0.7  & \textbf{7.01E-04} & 1.00E+00 &       & \textbf{3.16E-09} & 1.00E+00 \\
    -0.6  & \textbf{7.01E-04} & 1.00E+00 &       & \textbf{1.54E-06} & 1.00E+00 \\
    -0.5  & \textbf{1.49E-03} & 9.99E-01 &       & \textbf{2.06E-04} & 1.00E+00 \\
    -0.4  & \textbf{3.04E-03} & 9.98E-01 &       & 6.55E-02 & 9.48E-01 \\
    -0.3  & \textbf{2.14E-03} & 9.99E-01 &       & 6.93E-01 & 3.48E-01 \\
    -0.2  & \textbf{1.92E-02} & 9.85E-01 &       & 1.00E+00 & \textbf{1.34E-04} \\
    -0.1  & \textbf{1.09E-02} & 9.92E-01 &       & 1.00E+00 & \textbf{1.52E-07} \\
    0.0     & 7.23E-01 & 3.04E-01 &       & 1.00E+00 & \textbf{1.28E-21} \\
    0.1   & 7.99E-01 & 2.34E-01 &       & 1.00E+00 & \textbf{4.97E-07} \\
    0.2   & 1.44E-01 & 8.80E-01 &       & 9.99E-01 & \textbf{1.49E-03} \\
    0.3   & 8.11E-02 & 9.34E-01 &       & 9.48E-01 & 6.55E-02 \\
    0.4   & \textbf{1.09E-02} & 9.92E-01 &       & 7.31E-01 & 3.07E-01 \\
    0.5   & \textbf{2.14E-03} & 9.99E-01 &       & 6.10E-01 & 4.33E-01 \\
    0.6   & \textbf{2.06E-04} & 1.00E+00 &       & 6.10E-01 & 4.33E-01 \\
    0.7   & \textbf{2.09E-05} & 1.00E+00 &       & 4.33E-01 & 6.10E-01 \\
    0.8   & \textbf{8.24E-08} & 1.00E+00 &       & 3.07E-01 & 7.31E-01 \\
    0.9   & \textbf{8.63E-11} & 1.00E+00 &       & 3.48E-01 & 6.93E-01 \\
    1.0     & \textbf{7.21E-13} & 1.00E+00 &       & 2.34E-01 & 7.99E-01 \\
    \bottomrule
    \end{tabular} 
    
\end{table}

Note that a smaller negative value of $\tilde Z_{\varsigma\opt, \tau\opt}$ indicates better robustness of our first-stage decision $x_{\varsigma\opt, \tau\opt}$ than the first-stage SAA solution $x_{\text{SAA}}$. Hence, if $x_{\varsigma\opt, \tau\opt}$ is likely to outperform $x_{\text{SAA}}$, we have $\text{Pr}(\tilde Z_{\varsigma\opt, \tau\opt} < 0 ) > 0.5 $ and vice versa. This motivates us to test the null hypothesis
\[
\mathcal H_{\mathrm{null}} : \text{Pr}(\tilde Z_{\varsigma\opt, \tau\opt} < 0 ) = 0.5,
\]
against the alternative hypotheses
\[
\mathcal H_{+} : \text{Pr}(\tilde Z_{\varsigma\opt, \tau\opt} < 0 ) > 0.5.
\quad \text{and} \quad
\mathcal H_{-} : \text{Pr}(\tilde Z_{\varsigma\opt, \tau\opt} < 0 ) < 0.5,
\]
respectively.
The $p$-value is a statistical measure that helps decide whether to reject the null hypothesis. A smaller $p$-value indicates more substantial evidence against the null hypothesis in favor of the alternative hypothesis. 
\begin{table}[htbp]
  \centering
  
  \caption{The $p$-values of the two one-sided tests under different correlation coefficient shifts $ \rho_{\te} - \rho_{\tr}$ for two budget scenarios with training sample size $n_{\tr} =400$. Bold values indicate that the $p$-value is smaller than 0.05. }  \label{tab:signed-test-400}
    \begin{tabular}{cccccccc}
    \toprule
    \multirow{2}[2]{*}{$\rho_{\te} - \rho_{\tr} $ } & \multicolumn{2}{c}{Budget $b= 0.5c^\top A\mu$}       &       & \multicolumn{2}{c}{ Budget $b= 2c^\top A\mu$ } \\
\cmidrule{2-3}\cmidrule{5-6}     & $\mathcal{H}_+$    & $\mathcal{H}_-$   &   &  $\mathcal{H}_+$     & $\mathcal{H}_-$     \\
\cmidrule{1-3}\cmidrule{5-6}  
        -1.0    & \textbf{8.11E-02} & 9.34E-01 &       & \textbf{2.06E-04} & 1.00E+00 \\
    -0.9  & 1.20E-01 & 9.01E-01 &       & \textbf{2.06E-04} & 1.00E+00 \\
    -0.8  & 1.20E-01 & 9.01E-01 &       & \textbf{4.72E-04} & 1.00E+00 \\
    -0.7  & 4.33E-01 & 6.10E-01 &       & \textbf{7.01E-04} & 1.00E+00 \\
    -0.6  & \textbf{9.92E-02} & 9.19E-01 &       & \textbf{7.01E-04} & 1.00E+00 \\
    -0.5  & \textbf{2.01E-01} & 8.29E-01 &       & \textbf{8.05E-03} & 9.94E-01 \\
    -0.4  & 4.78E-01 & 5.67E-01 &       & 3.07E-01 & 7.31E-01 \\
    -0.3  & 3.90E-01 & 6.52E-01 &       & 8.80E-01 & 1.44E-01 \\
    -0.2  & 7.66E-01 & 2.69E-01 &       & 9.85E-01 & \textbf{1.92E-02} \\
    -0.1  & 8.29E-01 & 2.01E-01 &       & 1.00E+00 & \textbf{2.66E-06} \\
    0.0     & 1.00E+00 & \textbf{1.16E-04} &       & 1.00E+00 & \textbf{8.44E-40} \\
    0.1   & 1.00E+00 & \textbf{3.14E-04} &       & 1.00E+00 & \textbf{1.66E-12} \\
    0.2   & 9.96E-01 & \textbf{5.88E-03} &       & 1.00E+00 & \textbf{3.38E-05} \\
    0.3   & 4.78E-01 & 5.67E-01 &       & 9.94E-01 & \textbf{8.05E-03} \\
    0.4   & \textbf{1.92E-02} & 9.85E-01 &       & 8.80E-01 & 1.44E-01 \\
    0.5   & \textbf{2.06E-04} & 1.00E+00 &       & 6.10E-01 & 4.33E-01 \\
    0.6   & \textbf{2.09E-05} & 1.00E+00 &       & 1.44E-01 & 8.80E-01 \\
    0.7   & \textbf{2.77E-07} & 1.00E+00 &       & \textbf{2.51E-02} & 9.81E-01 \\
    0.8   & \textbf{1.66E-12} & 1.00E+00 &       & 5.24E-02 & 9.59E-01 \\
    0.9   & \textbf{1.29E-18} & 1.00E+00 &       & \textbf{4.15E-02} & 9.68E-01 \\
    1.0     & \textbf{1.52E-19} & 1.00E+00 &       & 5.24E-02 & 9.59E-01 \\
    \bottomrule
    \end{tabular} 
    
\end{table}

Table~\ref{tab:signed-test-100} and~\ref{tab:signed-test-400} report the $p$-values under the training sample size $n_{\tr}=100$ and $n_{\tr}=400$, respectively. The results with $n_{\tr}=100$ follow similar patterns to those observed with $n_{\tr}=200$ reported in the main text. When $n_{\tr}$ increases to 400, in the small budget scenario, the SAA method yields statistically significant results ($p$-values less than 0.05) for small correlation shifts $\rho_{\te}-\rho_{\tr}\in\{0,0.1,0.2\}$. This suggests that SAA becomes more effective with larger training datasets, but only if there are small correlation shifts.

\end{document}

%% file: style.tex
\usepackage{natbib}
 \bibpunct[, ]{(}{)}{,}{a}{}{,}

\usepackage{setspace}
\usepackage{textcase}
\usepackage{tikz}
\usetikzlibrary{arrows.meta}
\usepackage{adjustbox}
\usepackage{ulem}
\usepackage{rotating}

\usepackage{multicol,multirow}
\usepackage{enumitem}   
\usepackage{amsmath,amssymb,mathtools}   
\usepackage{dsfont}
\usepackage{booktabs}
\usepackage{fix-cm}
\usepackage{url}
\usepackage{xcolor}

\usepackage[margin=1in]{geometry}
\graphicspath{{graph/}}

\usepackage{fancyhdr}
\fancypagestyle{firstpage}{
     \fancyhf{}
    \fancyfoot[C]{\scriptsize\thepage}  
}

\newcommand{\be}{\begin{equation}}
\newcommand{\ee}{\end{equation}}

\newcommand{\R}{\mathbb{R}}

\newcommand{\E}{\mathbb{E}}
\newcommand{\PP}{\mathbb{P}}
\newcommand{\st}{\mathrm{s.t.}}

\newcommand{\Let}{\triangleq}
\newcommand{\Nx}{N_x}
\newcommand{\Ny}{N_y}
\newcommand{\Nd}{N_d}
\newcommand{\diag}{\mathrm{diag}}

\newcommand{\opt}{^\star}
\newcommand{\Wass}{\mathds W}
\newcommand{\tr}{\mathrm{tr}}
\newcommand{\te}{\mathrm{te}}

\usepackage{amsthm}
\newtheorem{theorem}{Theorem}
\newtheorem{lemma}{Lemma}
\newtheorem{proposition}{Proposition}

\newtheorem{definition}{Definition}
\newtheorem{assumption}{Assumption}

\graphicspath{{figures/}}

\usepackage[utf8]{inputenc}    
\usepackage[T1]{fontenc}       
\usepackage{tabularx}
\usepackage{lmodern}           
\usepackage{xcolor}